\documentclass[11pt,letterpaper]{amsart}
\usepackage{amsmath,amsthm,amscd,amssymb,amsfonts}
\usepackage{mathrsfs,upgreek}
\usepackage{verbatim}
\usepackage[sans]{dsfont}

\usepackage[utf8]{inputenc}
\usepackage[frenchb]{babel}
\usepackage[T1]{fontenc}

\usepackage{times}
\usepackage{setspace}

\usepackage[colorinlistoftodos,bordercolor=orange,backgroundcolor=orange!20,linecolor=orange,textsize=scriptsize]{todonotes}

\usepackage{mathtools}
\mathtoolsset{showonlyrefs}

\numberwithin{equation}{section} 

\theoremstyle{plain}

\newtheorem{letthm}{Th\'eor\`eme}

\newtheorem{letcor}[letthm]{Corollaire}

\newtheorem{Thm}{Th{\'e}or{\`e}me}[section]
\newtheorem{Prop}[Thm]{Proposition}
\newtheorem{Lem}[Thm]{Lemme}
\newtheorem{Cor}[Thm]{Corollaire}
\newtheorem{Prop-def}[Thm]{Proposition-D{\'e}finition}

\newtheorem{fact}{Fait}
\newtheorem{Lem-intr}{Lemme}

\theoremstyle{definition}
\newtheorem{Exemple}[Thm]{Exemple}

\newtheorem{Rem}[Thm]{Remarque}
\newtheorem*{Rem*}{Remarque}

\newtheorem{conj}{Conjecture}

\newcommand{\C}{{\mathbb{C}}}
\newcommand{\D}{{\mathbb{D}}}
\newcommand{\N}{{\mathbb{N}}}

\newcommand{\R}{{\mathbb{R}}}
\newcommand{\Z}{{\mathbb{Z}}}

\newcommand{\HK}{{\mathbb{H}_K}}
\newcommand{\PK}{{\mathbb{P}^{1}_K}}
\newcommand{\PKber}{{\mathsf{P}^{1}_K}}
\newcommand{\AK}{{\mathbb{A}^{1}_K}}
\newcommand{\AKber}{{\mathsf{A}^{1}_K}}
\newcommand{\Bber}{\mathsf{B}}
\newcommand{\xcan}{{x_{\operatorname{can}}}}

\newcommand{\PLber}{{\mathsf{P}^{1}_L}}

\newcommand{\PC}{{\mathbb{P}^{1}_{\C}}}
\newcommand{\Pu}{{\mathbb{L}}}
\newcommand{\PPu}{{\mathbb{P}^{1}_{\mathbb{L}}}}

\newcommand{\fm}{{\mathfrak{m}}}

\newcommand{\fM}{{\mathfrak{M}}}

\newcommand{\cA}{{\mathcal{A}}}
\newcommand{\cB}{{\mathcal{B}}}
\newcommand{\cC}{{\mathcal{C}}}

\newcommand{\cE}{{\mathcal{E}}}
\newcommand{\cF}{{\mathcal{F}}}

\newcommand{\cH}{{\mathcal{H}}}
\newcommand{\cI}{{\mathcal{I}}}
\newcommand{\cJ}{{\mathcal{J}}}
\newcommand{\cK}{{\mathcal{K}}}
\newcommand{\cL}{{\mathcal{L}}}
\newcommand{\cM}{{\mathcal{M}}}

\newcommand{\cO}{{\mathcal{O}}}
\newcommand{\cP}{{\mathcal{P}}}
\newcommand{\cS}{{\mathcal{S}}}
\newcommand{\cT}{{\mathcal{T}}}
\newcommand{\cU}{{\mathcal{U}}}

\newcommand{\cY}{{\mathcal{Y}}}

\newcommand{\sC}{\mathscr{C}}
\newcommand{\sE}{\mathscr{E}}
\newcommand{\sP}{\mathscr{P}}

\newcommand{\aT}{\mathsf{T}}
\newcommand{\aR}{\mathsf{R}}

\newcommand{\hd}{{\hat{d}}}

\newcommand{\hg}{{\hat{g}}}
\newcommand{\hell}{{\hat{\ell}}}

\newcommand{\hr}{{\hat{r}}}
\newcommand{\hx}{{\hat{x}}}
\newcommand{\hB}{{\widehat{B}}}

\newcommand{\hN}{{\widehat{N}}}

\newcommand{\hvarphi}{{\widehat{\varphi}}}
\newcommand{\hpsi}{{\widehat{\psi}}}
\newcommand{\htau}{{\widehat{\tau}}}

\newcommand{\tx}{{\widetilde{x}}}

\newcommand{\tK}{{\widetilde{K}}}

\newcommand{\tR}{{\tilde{R}}}

\newcommand{\ttau}{{\widetilde{\tau}}}

\newcommand{\hcE}{{\widehat{\cE}}}
\newcommand{\hcL}{{\widehat{\cL}}}

\newcommand{\tcC}{{\widetilde{\cC}}}

\newcommand{\tcL}{{\widetilde{\cL}}}

\newcommand{\kB}{{\check{B}}}

\newcommand{\e}{\varepsilon}
\newcommand{\od}{\overline{d}}
\newcommand{\vv}{{\vec{v}}}

\DeclareMathOperator{\ord}{ord}

\DeclareMathOperator{\supp}{supp}

\DeclareMathOperator{\RFix}{RFix}
\DeclareMathOperator{\diam}{diam}
\DeclareMathOperator{\crit}{\mathcal{C}}
\DeclareMathOperator{\wcrit}{\mathcal{I}}
\DeclareMathOperator{\PGL}{PGL}
\DeclareMathOperator{\PSL}{PSL}

\newcommand{\wdeg}[1]{\deg_{\operatorname{i}, #1}}
\newcommand{\wf}{\operatorname{i}} 

\renewcommand{\=}{\coloneqq}
\renewcommand{\:}{\colon}
\newcommand{\dd}{\hspace{1pt}\operatorname{d}\hspace{-1pt}}

\begin{document}
%
%

\setcounter{tocdepth}{1}

\title[Rigidit{\'e}, expansion et entropie]{Rigidit{\'e}, expansion et entropie en dynamique non-archim\'edienne}

\date{\today}
\author{Charles Favre \and Juan Rivera-Letelier}

\address{CMLS, CNRS, \'Ecole polytechnique, Institut Polytechnique de Paris, 91128 Palaiseau Cedex, France}

\email{charles.favre@polytechnique.edu}

\address{Department of Mathematics, University of Rochester.
  Hylan Building, Rochester, NY 14627, U.S.A.}
\email{riveraletelier@rochester.edu}

%
%

\begin{abstract}Nous montrons une propriété de rigidité en dynamique non archimédienne, qui rappelle le théorème de Zdunik en dynamique complexe~: toute fraction rationnelle dont la mesure d'équilibre charge un segment de la droite projective de Berkovich est affine Bernoulli.
  La démonstration s’inspire de la construction du modèle affine par morceaux d’une application multimodale de l’intervalle de Parry et de Milnor et Thurston.
  Ce résultat de rigidité nous permet de démontrer que l’entropie topologique de toute fraction rationnelle modérée est le logarithme d’un entier.
  Pour cela, nous analysons la propriété de (sous-)cobord multiplicatif de la dérivée sphérique et établissons un lien entre le signe de l’exposant de Lyapunov et la ramification sauvage.

\medskip

\noindent {\sc Abstract.}
We prove a rigidity property in non-Archimedean dynamics which is reminiscent of Zdunik' theorem in holomorphic dynamics: any rational map whose equilibrium measure puts some mass on a non-trivial segment of the Berkovich projective line is of affine Bernoulli type. 
The proof is inspired by the construction of the piecewise affine model of a multimodal map of the interval by Parry, Milnor and Thurston. This rigidity result allows us to prove that the topological entropy of any tame rational map is the logarithm of an integer 
To that end, we analyze the properties of the multiplicative (sub)-cocycle given by the spherical metric, and establish a link between the sign of the Lyapunov exponent and the phenomenon of wild ramification. 
 \end{abstract}

\maketitle
\tableofcontents

\section*{Introduction}

\subsection*{Rigidit\'e}
Une fraction rationnelle de degr{\'e} au moins deux est dite \emph{de Latt{\`e}s} lorsqu'elle est induite par un endomorphisme d'une courbe elliptique.
Les propri\'et\'es dynamiques des applications de Latt\`es sont d\'ecrites en d\'etails dans le cas complexe dans l'article de survol~\cite{milnor-lattes}, et on peut trouver dans~\cite[\S~5.1]{theorie-ergo} une discussion analogue du cas non-archim\'edien, voir aussi~\cite[\S~4.4]{benedetto-book}.
Il existe de nombreuses caract\'erisations de ces applications, que ce soit en termes de leurs propri\'et\'es d'int\'egrabilit\'e~\cite{casale}, de leur commutateur~\cite{ritt} ou de leurs propri\'et\'es plus arithm\'etiques~\cite{buium}.
Un r\'esultat c\'el\`ebre de Zdunik~\cite{zdunik} montre que les applications de Latt\`es sont les seuls endomorphismes de la sph\`ere de Riemann dont la mesure d'\'equilibre est absolument continue par rapport {\`a} la mesure de Lebesgue.
La d\'emonstration originelle de ce th\'eor\`eme a \'et\'e depuis simplifi\'ee~\cite[Theorem~1.1]{mayer}, \cite[Theorem~1.11]{JiXie23}, et un analogue de cet \'enonc\'e est disponible pour des {\'e}tats d'{\'e}quilibre \cite[Corollary~45]{SzoUrbZdu15} ainsi qu'en dimension sup\'erieure~\cite[Th{\'e}or{\`e}me~1]{berteloot-dupont}.

Notre premier r{\'e}sultat est une version non archim\'edienne du r\'esultat de Zdunik.
Pour le reste du papier, nous fixons un corps~${(K,|\cdot|)}$ \emph{alg\'ebriquement clos}, non trivialement valu\'e, non archim\'edien, complet et de corps r{\'e}siduel~$\tK$.
Notons~$\PKber$ la droite projective de Berkovich sur~$K$: c'est le compl\'et\'e en un sens ad\'equat de l'espace des boules de~$K$.
Elle est munie d'une topologie compacte pour laquelle elle porte une structure naturelle d'arbre r\'eel.
L'identification d'un point de~$K$ \`a la boule de rayon nul centr\'e en ce point plonge canoniquement~$\PK$ dans l'ensemble des bouts de~$\PKber$.
Les points de~$\PKber$ associ\'es aux boules dont le rayon appartient au groupe des valeurs~$|K^*|$ sont dit de type~II: ils co\"{\i}ncident avec l'orbite sous~$\PGL(2,K)$ du point associ\'e \`a la boule unit\'e.
L'espace~$\HK$ d{\'e}fini par ${\HK \= \PKber \setminus \PK}$ est alors un arbre muni d'une m\'etrique canonique compl\`ete~$d_\HK$ dite \emph{hyperbolique}.
Nous renvoyons \`a \cite{baker-rumely,berko-livre rouge,survey-jonsson} pour plus de d\'etails sur la g\'eom\'etrie de cet espace.

On fixe d\'esormais~$R$ une fraction rationnelle {\`a} coefficients dans~$K$ et de degr\'e~$d$ au moins deux.
Elle induit une application continue not\'ee encore~$R$ de~$\PKber$ dans elle-m\^eme, et on peut naturellement lui associer un ensemble de Julia $J_R$ compact et de Fatou $F_R$ ouvert comme dans le cas complexe. Ces deux ensembles sont totalement invariants.
Il existe de plus une unique mesure de probabilit\'e ergodique~$\rho_R$ support\'ee sur~$J_R$ et \'equilibr\'ee au sens o\`u ${R^* \rho_R = d \times \rho_R}$ \cite{baker-rumely,FavRiv04,theorie-ergo}.

On introduit maintenant une classe de fractions rationnelles dont l'ensemble de Julia est contenu dans un segment.
On dira que~$R$ est \emph{affine Bernoulli}, s'il existe un entier ${k\ge 2}$ et des segments ferm{\'e}s~$I, I_1, \ldots, I_k$ de~$\HK$ tels que d'une part
\begin{equation}
  \label{eq:1}
  R^{-1}(I)
  =
  I_1 \cup \cdots \cup I_k
  \subseteq
  I~,
\end{equation}
et d'autre part, pour chaque~$j$ dans~$\{1, \ldots, k \}$, l'application~$R \colon I_j \to I$ est affine par rapport {\`a} la distance~$d_{\HK}$ de facteur de dilatation~$d_j\ge2$.
L'entier~$k$ est uniquement d{\'e}termin{\'e} par~$R$, ainsi que la suite d'entiers~$(d_1, \ldots, d_k)$, modulo l'involution renversant l'ordre.
On appelle~$d_1, \ldots, d_k$ les \emph{facteurs de dilatation de~$R$}.
Dans le cas particulier~o{\`u}
$$ d = k^2
\text{ et }
d_1 = \cdots = d_k = k~, $$
on parlera de fraction rationnelle \emph{\`a allure Latt{\`e}s}.
Cette terminologie est justifiée par le fait que toute fraction rationnelle de Latt{\`e}s induite par un endomorphisme d'une courbe elliptique {\`a} mauvaise r{\'e}duction est de cette forme, voir \cite[\S~5.1]{theorie-ergo}.
Nous renvoyons \`a la Proposition~\ref{prop:affine Bernoulli} ci-dessous pour plus de détails sur la dynamique des fractions rationnelles affine Bernoulli et \`a~\cite[\S~5.2]{theorie-ergo} pour des exemples.

Le r{\'e}sultat suivant est l'analogue du r{\'e}sultat de Zdunik sus-mentionn\'e.
Notons que l'ensemble de Julia d'une application {\`a} allure Latt{\`e}s est {\'e}gal {\`a} un segment de~$\HK$ et que sa mesure d'{\'e}quilibre est alors la mesure de probabilit{\'e} proportionnelle {\`a} la mesure de Hausdorff 1-dimensionnelle associ{\'e}e {\`a}~$d_{\HK}$.

\begin{letthm}
  \label{thm:main1}
  Soit~$R$ une fraction rationnelle {\`a} coefficients dans~$K$, de degr{\'e} au moins deux.
  Si~$\rho_R$ ne poss\`ede pas d'atome et charge un segment de~$\HK$, alors~$R$ est affine Bernoulli.
  Si de plus~$\rho_R$ n'est pas singuli{\`e}re par rapport {\`a} la mesure de Hausdorff 1\nobreakdash-dimensionnelle associ{\'e}e {\`a}~$d_{\HK}$, alors~$R$ est {\`a} allure Latt{\`e}s.
\end{letthm}

Rappelons qu'une fraction rationnelle dont la mesure~$\rho_R$ poss{\`e}de un atome a \emph{bonne r\'eduction potentielle}~\cite[Th{\'e}or{\`e}me~E]{theorie-ergo}.
Dans ce cas, $J_R$ est r\'eduit \`a un point fixe de type~II et, quitte \`a conjuguer~$R$, la r\'eduction de~$R$ dans~$\tK$ est de m{\^e}me degr\'e que~$R$.

Il existe des applications rationnelles dont le support de la mesure d'\'equilibre est un arbre ayant une infinit\'e de points de branchement, et qui ne sont donc pas affine Bernoulli.
Une \'etude de tels exemples et le calcul de leur entropie topologique sont donn\'es dans~\cite{benedetto-connected}.
On ne peut donc pas remplacer l'hypoth\`ese~$\rho_R$ charge un segment par~$J_R$ contient un segment dans le th\'eor\`eme pr\'ec\'edent.

Notons que toute perturbation assez petite d'une application {\`a} allure Latt{\`e}s reste de cette forme de telle sorte que l'on ne peut
pas non plus rempla{\c{c}}er ``{\`a} allure Latt{\`e}s'' par ``de Latt{\`e}s'' dans le Th{\'e}or{\`e}me~\ref{thm:main1}, voir \cite[Remarque~5.3]{theorie-ergo}.

La d{\'e}monstration du Th{\'e}or{\`e}me~\ref{thm:main1} s'inspire de la construction du mod{\`e}le affine par morceaux d'une application multimodale de l'intervalle introduite par Parry~\cite{Par66} et {\'e}tudi{\'e}e par \mbox{Milnor} et Thurston~\cite{MilThu88}.
On {\'e}crase~$\PKber$ sur un arbre r{\'e}el compact de longueur finie, où la mesure d'{\'e}quilibre devient la mesure des longueurs, et la fraction rationnelle induit une application uniform{\'e}ment expansive et ergodique.
L'{\'e}tape cl{\'e} consiste {\`a} montrer que ce mod{\`e}le est un segment.
Pour cela, on combine le caract{\`e}re m{\'e}langeant de l'application induite avec le fait que tout arbre r{\'e}el de longueur finie ressemble localement {\`a} un segment.

%
%

\subsection*{Exposant de \mbox{Lyapunov} et ramification sauvage}
Nous \'etudions les propri\'et\'es ergodiques d'une fraction rationnelle en fonction de la positivité de son exposant de \mbox{Lyapunov}.
Comme nous le verrons, la pr\'esence de ramification sauvage procure de s\'erieuses difficult\'es {\`a} cet {\'e}gard.

Pour {\^e}tre plus pr{\'e}cis, fixons une fraction rationnelle~$R$ {\`a} coefficients dans~$K$ et de degr{\'e} au moins deux.
On appelle \emph{exposant de \mbox{Lyapunov} de~$R$} celui de la mesure d'{\'e}quilibre~$\rho_R$, que l'on note par~$\chi(R)$ et qui est d{\'e}fini par
\[
  \chi(R)
  \=
  \int_{\PKber} \log \| R' \| \, \dd \rho_R~,
\]
o\`u $\| R' \|$ est la d\'eriv\'ee sph\'erique de $R$, voir \S~\ref{sec:derivee-Lyap} pour plus de pr\'ecisions.
Cette quantit{\'e} est finie si~$R$ est s\'eparable, et {\'e}gale {\`a}~$-\infty$ si~$R$ est ins{\'e}parable, voir le Corollaire~\ref{cor:justif-lyap}, ou~\cite[Proposition~3.3]{theorie-ergo} et~\cite[Lemme 4.2]{petits-points}, ainsi que~\cite[Th\'eor\`eme~4.9]{CLT} en dimension sup\'erieure.
Lorsque~$K$ est de caract{\'e}ristique nulle, l'exposant de \mbox{Lyapunov} de~$R$ peut {\^e}tre approxim{\'e} par des moyennes d'exposants de \mbox{Lyapunov} de points p{\'e}riodiques, voir \cite[Theorem~1]{okuyama}, ainsi que \cite[Theorems~1 et~2]{okuyama-approx-lyap} et~\cite[Theorem~A]{GOV2} pour des versions quantitatives.

Un point de~$\PKber$ est \emph{critique pour~$R$}, si $R$ n'est pas un isomorphisme local en ce point.
On note~$\crit_R$ l'ensemble des points critiques de~$R$.
Il contient tous les points de~$\PK$ o\`u la d\'eriv\'ee de $R$ s'annule.
Nous dirons qu'un point critique de~$R$ est \emph{inséparable} s'il appartient {\`a}~$\HK$ et~$R$ est \emph{ins{\'e}parable} en ce point.
On note~$\wcrit_R$ l'ensemble des points critiques inséparables, voir \S~\ref{ss:action} pour plus de détails. En particulier,
nous donnons une caractérisation g{\'e}ométrique des points critiques inséparables qui est importante pour la suite, voir le Corollaire~\ref{c:tres-wild}, ainsi que \cite[Theorem~B]{faber1} et~\cite{temkin} pour d'autres caract{\'e}risations.

\smallskip

Nous dirons qu'un potentiel ${g\colon\PKber \to \R}$ est \emph{divisoriel}, si son Laplacien est support\'e sur un ensemble fini des points de type~II.
Notons que tout potentiel divisoriel est continu et donc born{\'e}. Dans le cadre de la théorie du pluripotentiel sur un corps métrisé non-archimédien, les potentiels divisoriels sont aussi appelés fonctions modèles, voir, e.g., ~\cite[\S~5.4]{BFJ16}, et jouent parfois le rôle des fonctions tests en analyse non-archimédienne.

Rappelons qu'un point fixe~$x$ est dit r\'epulsif s'il appartient \`a ${\crit_R \cap \HK}$, ou s'il appartient \`a~$\PK$ et le multiplicateur de~$R$ en $x$ est de norme strictement plus grand que~$1$.
Pour chaque~$n$ dans~$\N^*$, notons~$\RFix(R^n)$ l'ensemble des points fixes r{\'e}pulsifs de~$R^n$, et ${\RFix(R^n, K) \= \RFix(R^n) \cap \PK}$.
Dans notre r{\'e}sultat suivant, on utilise la fonction ${\diam \: \PKber \to [0, +\infty[}$ d{\'e}finie en~\eqref{eq:def-diam}.
Elle s’annule pr{\'e}cis{\'e}ment sur~$\PK$ et, pour un point~$x$ de type~II, elle correspond au diam{\`e}tre projectif de l’une des composantes connexes du complémentaire de~$x$.
Comme les ensembles~$\PK$ et~$\HK$ forment une partition mesurable de~$\PKber$, l'ergodicit{\'e} de~$\rho_R$ implique que ${\rho_R(\PK) = 1}$ ou ${\rho_R(\HK) = 1}$.
On peut alors r\'esumer nos r\'esultats de la mani\`ere suivante.

\begin{letthm}\label{thm:main3}
  Soit~$R$ une fraction rationnelle {\`a} coefficients dans~$K$, et de degr{\'e}~$d$ au moins deux.
  \begin{enumerate}
  \item
    Nous avons $\chi(R) <0 $ si et seulement si $\rho_R(\wcrit_R) >0$, et lorsque c'est le cas, alors $\rho_R(\HK) = 1$.
  \item
    Si $\chi(R) =0$ et $\rho_R(\HK) = 1$, alors il existe un potentiel divisoriel~$g$ et une constante $C>0$ telles que $\log|R'| = g \circ R - g$ et ${\diam \ge C}$ sur~$J_R$.
  \item
    Si $\chi(R)> 0$, alors $\rho_R(\PK) = 1$, l'entropie m{\'e}trique de~$\rho_R$ et l'entropie topologique de~$R$ sont toutes deux \'egales \`a $\log d$, et l'on a
    \[
      \frac1{d^n} \sum_{x \in \RFix(R^n, K)} \delta_x \to \rho_R
      \text{ lorsque } {n \to +\infty}~.
    \]
  \end{enumerate}
\end{letthm}

Ce r\'esultat tranche avec son analogue complexe.
En effet, toute fraction rationnelle {\`a} coefficients complexes de degr{\'e}~$d$ au moins deux poss\`ede un exposant de \mbox{Lyapunov} plus grand ou {\'e}gal {\`a} $\frac12 \log d$.
De plus, l'un des r{\'e}sultats de rigidit{\'e} de Zdunik \cite[Theorem~2]{zdunik} montre que le cas d'\'egalit\'e caract\'erise les applications de Latt\`es.

Les preuves des points (1) et (2) du Th\'eor\`eme~\ref{thm:main3} reposent de mani\`ere cruciale sur l'in{\'e}galit{\'e}
\begin{equation}
  \label{eq:2}
  \| R' \|
  \le
  \diam \circ R / \diam
\end{equation}
valable sur~$\HK$ (Proposition~\ref{prop:cocycle}). Pour souligner que $\|R'\|$ est dominé par un cobord, nous l'appellerons
par "sous-cobord géométrique".
Ce n'est pas un cobord en général en raison de la ramification sauvage, et on {\'e}tablit un lien entre l'exposant de \mbox{Lyapunov} et la ramification sauvage de~$R$ pour chaque mesure de probabilit{\'e} ergodique et invariante par~$R$ support{\'e}e sur~$\HK$ (Th{\'e}or{\`e}me~\ref{t:exponsants}~(1)).
Nous d{\'e}duisons le Th\'eor\`eme~\ref{thm:main3}~(1), dans le cas plus difficile des fractions rationnelles s{\'e}parables, comme cons{\'e}quence de ce lien.
Nous d{\'e}duisons {\'e}galement une minoration de l'exposant de \mbox{Lyapunov} de toute mesure de probabilit{\'e} sur~$\PKber$, invariante et ergodique par~$R$, qui n'est pas support{\'e}e sur un cycle attractif (Corollaire~\ref{c:exponsants}).
Ce r{\'e}sultat g{\'e}néralise la minoration obtenue par Nie~\cite{nie}, voir aussi~\cite{jacobs-lower}.

Le point (2) est un résultat de rigidité dont la démonstration repose sur une analyse du sous-cobord g{\'e}om{\'e}trique~\eqref{eq:2} qui est r\'eminiscente de celle faite par Zdunik dans le cas complexe.
En effet, lorsque ${\chi(R) = 0}$ et ${\rho_R(\HK) = 1}$, on montre que l'{\'e}galit{\'e} dans~\eqref{eq:2} est satisfaite $\rho_R$-presque partout.
Le point clef, analogue {\`a} \cite[Lemma~2]{zdunik}, consiste {\`a} d{\'e}montrer que la fonction ${\log \diam |_{J_R}}$ se prolonge en une fonction r{\'e}guli{\`e}re (Proposition~\ref{lem:key-cstdiam}).

Dans le cas complexe, o{\`u} l'on a toujours ${\chi(R) > 0}$, l'{\'e}quidistribution des points p{\'e}riodiques r{\'e}pulsifs vers la mesure d'équilibre a {\'e}t{\'e} d{\'e}montr{\'e}e par Ljubich en dimension un~\cite{lyubich}, puis g{\'e}n{\'e}ralis{\'e}e par Briend et Duval {\`a} toute dimension~\cite{briend:Lyap}.
Pour {\'e}tablir ce r{\'e}sultat dans notre contexte non-archim{\'e}dien, nous suivons la stratégie élaborée dans ce dernier article.
Remarquons que ce r{\'e}sultat entra{\^{\i}}ne~:
\begin{equation}
  \label{eq:71}
  d^{-n} \# \RFix(R^n, K) \to 1
  \text{ lorsque } {n \to +\infty}~,
\end{equation}
et, par cons{\'e}quent, on ne peut pas le d\'eduire directement des théorèmes d'équidistribution des points p\'eriodiques de~$\PK$ en caract\'eristi\-que nulle comme énoncés dans~\cite[Th\'eor\`eme~B]{theorie-ergo} et~\cite[Theorem~1.2]{okuyama2}.
Il existe ainsi des fractions rationnelles dont l'exposant de \mbox{Lyapunov} est strictement positif et qui poss{\`e}dent un nombre infini de points p\'eriodiques non-r{\'e}pulsifs dans~$\PK$.
Les m\'ethodes utilis{\'e}es dans ces deux articles r{\'e}posent sur la th\'eorie du potentiel, tandis qu'ici, nous utilisons une approche nettement plus dynamique.

\smallskip


\begin{Rem*}
  En adaptant la m\'ethode de Ljubich~\cite{lyubich} pour les fractions rationnelles complexes, reprise par Briend et Duval~\cite{briend:IHES} en dimension quelconque, on peut montrer que, pour toute fraction rationnelle~$R$ {\`a} coefficients dans~$K$ et de degr\'e au moins deux v{\'e}rifiant ${\rho_R(\PK) = 1}$, la mesure d'\'equilibre est l'unique mesure d'entropie m{\'e}trique maximale $\log \deg(R)$.
\end{Rem*}

Nous rassemblons plusieurs exemples dans \S~\ref{ss:exemples}, notamment un exemple qui ne satisfait aucun des points~(1), (2) ou~(3) du Th\'eor\`eme~\ref{thm:main3}, voir l'Exemple~\ref{ex:atypique}.
On montre que de telles fractions rationnelles n'existent que lorsque la caract\'eristique r{\'e}siduelle de~$K$ est strictement positive, voir le Th{\'e}or{\`e}me~\ref{thm:main2} ci-dessous.

\subsection*{Propri\'et\'es ergodiques des applications mod\'er\'ees}
L'ensemble critique d'une fraction rationnelle peut {\^e}tre touffu, ce qui pose des difficultés significatives pour {\'e}tudier ses propri{\'e}t{\'e}s ergodiques.
Nous renvoyons aux travaux de \mbox{Faber} \cite{faber1,faber2} pour une \'etude de la g\'eom\'etrie de l'ensemble critique dans un cadre non dynamique, ainsi qu'{\`a} l'article d'Irokawa~\cite{irokawa} pour le cas des fractions rationnelles de degré~$3$, et {\`a}~\cite{BojPoi18,temkin,temkin2} pour des r{\'e}sultats plus généraux concernant les morphismes entre courbes.

Nous nous proposons maintenant d'{\'e}tudier les propri{\'e}t{\'e}s ergodiques des fractions rationnelles dont l'ensemble critique est le plus simple possible.
Suivant~\cite[\S~1.3]{theorie-ergo}, nous dirons qu'une fraction rationnelle {\`a} coefficients dans~$K$ est \emph{mod\'er\'ee} si son ensemble critique est inclus dans un sous-arbre fini de~$\PKber$.
Faber \cite[Corollaire~7.13]{faber1} a montr{\'e} qu'une fraction rationnelle est mod\'er\'ee si et seulement si elle n'a aucun point critique ins{\'e}parable.
Nous donnons une d{\'e}monstration ind{\'e}pendante de ce r{\'e}sultat, voir la Proposition~\ref{prop:modere}.
Nous renvoyons aussi à~\cite[Propositions~2.9 et~2.11]{trucco} pour le cas des polyn{\^o}mes.
Par cons{\'e}quent, toute fraction rationnelle {\`a} coefficients dans~$K$ est mod{\'e}r{\'e}e lorsque la caract\'eristique r{\'e}siduelle de $K$ est nulle.

Le r{\'e}sultat suivant fournit une classification dynamique des fractions rationnelles mod{\'e}r{\'e}es.
Pour un nombre r\'eel~$x$, notons par~$[x]$ le plus grand entier inf{\'e}rieur ou {\'e}gal {\`a}~$x$.

\begin{letthm}\label{thm:main2}
  Soit~$R$ une fraction rationnelle {\`a} coefficients dans~$K$ mod\'er\'ee, et de degr{\'e}~$d$ au moins deux.
  Alors ${\chi(R) \ge 0}$ et nous sommes dans l'une des trois situations suivantes.
  \begin{enumerate}
  \item
    Soit $\chi(R)>0$.
    Dans ce cas, on a $\rho_R(\PK) =1$, et l'entropie m{\'e}trique de~$\rho_R$ et l'entropie topologique de~$R$ sont toutes deux \'egales \`a~$\log d$.
  \item
    Soit $\chi(R)=0$, et $R$ n'a pas bonne r\'eduction potentielle.
    Alors $R$ est affine \mbox{Bernoulli}, d'entropie topologique~$\log k$ pour un entier $k \in \{ 2, \ldots , [\sqrt{d}] \}$, et admet une unique mesure d'entropie m{\'e}trique maximale.
  \item
    Soit $\chi(R)=0$, et $R$ a bonne r\'eduction potentielle.
    Dans ce cas $\rho_R$ est une mesure atomique support\'ee sur un point fixe r\'epulsif de type II, et l'entropie topologique de~$R$ est nulle.
  \end{enumerate}
\end{letthm}

Dans le cas o{\`u}~$K$ est de caract{\'e}ristique r{\'e}siduelle nulle, toute fraction rationnelle~$R$ {\`a} coefficients dans~$K$ et degr{\'e} au moins deux est mod{\'e}r{\'e}e et satisfait ${\chi(R) \ge 0}$ par le Th{\'e}or{\`e}me~\ref{thm:main2}.
Cette in{\'e}galit{\'e} d{\'e}coule de la combinaison de \cite[Theorem~2]{okuyama} et \cite[Theorem~1.5]{BenIngJonLev14} ou \cite[Theorem~A]{0Riv2601}, ainsi que de \cite[Theorem~A]{GOV2} avec ${r = 1}$.
Plus g{\'e}n{\'e}ralement, on montre que l'exposant de \mbox{Lyapunov} de toute mesure de probabilit{\'e} sur~$\PKber$, invariante et ergodique par~$R$, est positif ou nul, pourvu qu'elle ne soit pas support{\'e}e sur un cycle attractif (Corollaire~\ref{c:exponsants}).

Les Th{\'e}or{\`e}mes~\ref{thm:main1}, \ref{thm:main3} et~\ref{thm:main2} nous permettent d'obtenir les caract\'erisations suivantes des applications mod\'er\'ees d'exposant de \mbox{Lyapunov} nul.

\begin{letcor}
  \label{cor:main4}
  Soit~$R$ une fraction rationnelle {\`a} coefficients dans~$K$ mod\'er\'ee, et de degr{\'e}~$d$ au moins deux.
  Les propri\'et\'es suivantes sont \'equivalentes.
  \begin{enumerate}
  \item
    Ou bien~$R$ est affine Bernoulli, ou bien elle a bonne r\'eduction potentielle.
  \item
    L'entropie topologique de $R$ est strictement plus petite que~$\log d$.
  \item
    L'exposant de \mbox{Lyapunov} de $R$ est nul.
  \item
    L'ensemble de Julia de $R$ est inclus dans $\HK$.
  \item
    Tous les points p\'eriodiques de~$R$ dans~$\PK$ sont attractifs ou indiff\'erents.
  \item
    La mesure d'\'equilibre~$\rho_R$ charge un segment de~$\HK$.
  \end{enumerate}
  Lorsque ces propri{\'e}t{\'e}s {\'e}quivalentes sont satisfaites, l'entropie topologique de~$R$ est le logarithme d'un entier dans~${\{1, \ldots, [\sqrt{d}] \}}$, et l'ensemble de Julia de~$R$ est contenu dans l'ensemble critique de~$R$.
\end{letcor}

Lorsque~$K$ est de caratéristique résiduelle nulle, l'implication (5)$\Rightarrow$(1) d{\'e}coule des r{\'e}sultats de Luo~\cite[Propositions~11.4 et~11.5]{luo2}, d{\'e}montr{\'e}s avec une m{\'e}thode diff{\'e}rente.
\footnote{Cet argument combiné à la non-négativité de l'exposant de Lyapunov et au Théorème~\ref{thm:main2} donne une démonstration alternative de notre Théorème~\ref{thm:main1} en caratéristique résiduelle nulle.}
Lorsque~$R$ est un polyn{\^o}me et que~$K$ admet un sous-corps {\`a} valuation discr{\`e}te {\`a} cl{\^o}ture alg{\'e}brique dense, l'implication (5)$\Rightarrow$(1) a {\'e}t{\'e} auparavant d{\'e}montr{\'e}e par Trucco \cite[Corollary~C]{trucco}.

\medskip

Le corollaire suivant est une cons{\'e}quence imm{\'e}diate du Th{\'e}o{\`e}me~\ref{thm:main2}.

\begin{letcor}
  L'entropie topologique d'une fraction rationnelle mod\'er\'ee est le logarithme d'un entier.
\end{letcor}

\begin{Rem*}
  Dans~\cite{benedetto-connected} est explicit\'ee une construction d'une fraction rationnelle~$R$ de degr\'e~$6$ {\`a} coefficients dans~$\C_3$ dont l'entropie topolo\-gique est \'egale au logarithme de la plus grande racine r\'eelle du polyn\^ome irr\'eductible $t^3 - 4t^2 -t +6$.
  Cette fraction n'est pas mod\'er\'ee, et v\'erifie $\chi(R)<0$.

  Nous avons exhibé dans~\cite{theorie-ergo} des fractions rationnelles qui n'était pas localement injective
  sur leur ensemble de Julia. Cela rend la construction de partitions mesurables en toute généralité
  extrêmement ardue, et explique les difficultés à calculer l'entropie topologique d'une fraction rationnelle sur un corps quelconque.
\end{Rem*}

Comme toute fraction rationnelle affine Bernoulli est de degr{\'e} au moins~$4$, le corollaire suivant est une cons{\'e}quence imm{\'e}diate du Th{\'e}or{\`e}me~\ref{thm:main2}.

\begin{letcor}
  \label{c:modere-degre-petit}
  Soit~$R$ une fraction rationnelle {\`a} coefficients dans~$K$, mod\'er\'ee et de degr{\'e}~$2$ ou~$3$.
  Alors, soit~$R$ a bonne r{\'e}duction potentielle, auquel cas son entropie topologique est nulle, soit ${\chi(R) > 0}$, auquel cas son entropie topologique est {\'e}gale {\`a}~$\log \deg(R)$.
\end{letcor}

Un r{\'e}sultat analogue est valable pour tout polyn{\^o}me~$P$ {\`a} coefficients dans~$K$, mod{\'e}r{\'e} et de degr{\'e} au moins deux \cite[Corollaire~F]{theorie-ergo}.
La formule de \mbox{Przytycki}, voir la Proposition~\ref{p:pr} ou~\cite[\S~5]{okuyama} dans le cas o{\`u}~$K$ est de caract{\'e}ristique nulle, montre que le cas de bonne r{\'e}duction potentielle correspond {\`a} la situation o{\`u} l'orbite de chaque point critique dans~$K$ est born\'ee, et que l'exposant de \mbox{Lyapunov} est strictement positif d{\`e}s que l'orbite d'au moins un des points critiques finis tend vers l'infini, voir~\cite[Corollary~2.11]{kiwi-cubic}.
De plus, lorsque~$K$ admet un sous-corps {\`a} valuation discr{\`e}te {\`a} cl{\^o}ture alg{\'e}brique dense, l'ensemble ${J_P \setminus \PK}$ est soit vide, soit {\'e}gal {\`a} l'union d'un nombre fini de grandes orbites de points p\'eriodiques r{\'e}pulsifs \cite[Theorem~A]{trucco}.
Voir~\cite[Theorem~1]{kiwi} pour un r{\'e}sultat similaire pour les fractions rationnelles de degr{\'e}~$2$, dans le cas o{\`u}~$K$ est le corps des s{\'e}ries de Puiseux {\`a} coefficients dans~$\C$.

Nous donnons dans \S~\ref{sec:application} une application des Th{\'e}or{\`e}mes~\ref{thm:main3} et~\ref{thm:main2} pour les familles m\'eromorphes de fractions rationnelles complexes~$\{R_t\}_{t \in \D}$ param\'etr\'ees par le disque unit\'e~$\D$, d{\'e}fini par ${\D \= \{ t\in \C, \, |t|<1\}}$.
Nous verrons que pour une telle famille la dichotomie suivante apparaît:
\begin{itemize}
\item
  soit les multiplicateurs des points fixes de~$R_t^n$ sont uniform\'ement born\'es (en $t$ et en $n$)
  et la fraction rationnelle~$\aR$ associ\'ee \`a $\{R_t\}$ d\'efinie sur le corps des s\'eries de Puiseux {\`a} coefficients dans~$\C$ poss\`ede un exposant de \mbox{Lyapunov} nul;
\item
  soit une proportion positive de cycles p\'eriodiques ont un multiplicateur qui explose.
\end{itemize}
Nous renvoyons au Th\'eor\`eme~\ref{thm:main6} et la Proposition~\ref{p:famille-positive} pour plus de d\'etails.


\subsection*{Plan}

En dehors de cette introduction, cet article contient~$5$ sections.
Dans \S~\ref{s:preliminaires}, nous rassemblons divers r{\'e}sultats pr{\'e}liminaires et fixons les notations et conventions.

Dans \S~\ref{sec:zdunik}, nous donnons la d\'emonstration du Th\'eor\`eme~\ref{thm:main1}, qui caract\'erise les fractions rationnelles dont la mesure d'\'equilibre charge un segment de~$\HK$.
Nous {\'e}tudions le quotient~$\cT$ de~$\PKber$ par la relation d'\'equivalence identifiant deux points~$x$ et~$x'$ d\`es que ${\rho_R([x,x']) = 0}$.
Nous montrons que~$\cT$ admet une structure naturelle d'arbre r{\'e}el m{\'e}trique compact (\S\S~\ref{sec:arbre quotient}, \ref{sec:compacite de l'arbre}) et que la fraction rationnelle induit une application sur~$\cT$, poss{\'e}dant un degr{\'e} local et {\'e}tant uniform{\'e}ment expansive et ergodique (\S\S~\ref{sec:application quotient}, \ref{sec:phi-deg}, \ref{sec:expansion-uniforme}).
Enfin, nous montrons que~$\cT$ est un segment (\S~\ref{sec:simplicite de l'arbre}) et que l'application induite est affine Bernoulli (\S~\ref{sec:phi}), ce qui permet d'en d{\'e}duire le Th{\'e}or{\`e}me~\ref{thm:main1} (\S~\ref{sec:preuve main1}).

Nous donnons dans \S~\ref{Sec:cocycle} la d\'emonstration du Th\'eor\`eme~\ref{thm:main3}.
Nous commençons par rappeller des propri{\'e}t{\'e}s de la d\'eriv\'ee sph\'erique et montrer que~$\chi(R)$ est fini si et seulement si~$R$ est s{\'e}parable (\S~\ref{sec:derivee-Lyap}).
Nous donnons ensuite une formule pour l’exposant de \mbox{Lyapunov} dans le cas des polyn{\^o}mes (\S~3.2), que nous utilisons dans \S~5.1, mais qui n’est pas n{\'e}cessaire {\`a} la d{\'e}monstration du Th{\'e}or{\`e}me~\ref{thm:main3}.
Ensuite, nous exprimons la d{\'e}riv{\'e}e sph{\'e}rique de~$R$ sous la forme d'un sous-cobord multiplicatif, qui n'est pas un cobord en raison de la ramification sauvage (\S~\ref{sec:derivee}).
Ce sous-cobord motive la d{\'e}finition d'une fonction ${\HK \to \mathopen[ 0, +\infty \mathclose[}$, cohomologue {\`a} ${-\log \| R' \|}$, qui quantifie le caract{\`e}re sauvage ou ins{\'e}parable de~$R$ en chaque point de~$\HK$.
Gr{\^a}ce {\`a} cette fonction, on {\'e}tablit un lien entre l'exposant de \mbox{Lyapunov} et la ramification sauvage de~$R$ pour chaque mesure de probabilit{\'e} ergodique support{\'e}e sur~$\HK$ (\S~\ref{ss:exposants}).
Nous d{\'e}duisons le Th\'eor\`eme~\ref{thm:main3}~(1), dans le cas plus difficile des fractions rationnelles s{\'e}parables, comme cons{\'e}quence de ce lien (\S~\ref{sec:thm31}).

Le pas cl{\'e} pour d{\'e}montrer le Th\'eor\`eme~\ref{thm:main3}~(2) est que la fonction~$\diam$ est constante sur un sous-ensemble de mesure pleine d'une boule de~$\PKber$.
Pour cela, nous {\'e}tablissons tout d'abord l'int\'egrabilit\'e de la fonction~$\log \diam$ {\`a} l'aide de la th{\'e}orie du potentiel et du caract{\`e}re {\'e}quilibr{\'e} de la mesure~$\rho_R$ (\S~\ref{ss:integrable}).
Combin{\'e}e avec une version du th{\'e}or{\`e}me de diff{\'e}rentibilit{\'e} de Lebesgue et une construction de préimages itérées étales, obtenues {\`a} l'aide du Th{\'e}or{\`e}me~\ref{thm:main1} et inspir{\'e}e de la construction des branches inverses des fractions rationnelles complexes introduite indépendamment dans~\cite{FreLopMan83} et~\cite{lyubich}, cela nous permet de montrer que~$\diam$ est constante sur un sous-ensemble de mesure pleine d'une boule de~$\PKber$.
Nous en d{\'e}duisons le Th\'eor\`eme~\ref{thm:main3}~(2) en utilisant le cocycle g{\'e}om{\'e}trique (\S~\ref{sec:thm32}).

Pour d{\'e}montrer le Th\'eor\`eme~\ref{thm:main3}~(3), nous suivons la preuve d'un r{\'e}sultat analogue dans le cas complexe, donn\'ee par Briend et Duval dans~\cite[\S~3]{briend:Lyap}, qui repose sur la th\'eorie de Pesin (\S~\ref{sec:briendduval}).

Dans \S~\ref{s:modere}, nous donnons des applications aux fractions rationnelles mod{\'e}r{\'e}es (\S~\ref{ss:modere}) et aux familles m{\'e}romorphes de fractions rationnelles complexes (\S~\ref{sec:application}).
Apr{\`e}s des rappels sur les applications mod{\'e}r{\'e}es, nous montrons qu'une fraction rationnelle mod\'er\'ee dont la mesure d'\'equilibre ne charge aucun segment de~$\HK$, poss\`ede un point p{\'e}riodique r{\'e}pulsif dans~$\PK$.
La d\'emonstration repose sur la construction des branches inverses introduite ind{\'e}pendamment dans~\cite{FreLopMan83} et~\cite{lyubich}.
Nous concluons \S~\ref{ss:modere} par la d{\'e}monstration du Th\'eor\`eme~\ref{thm:main2} et du Corollaire~\ref{cor:main4}.
Enfin, nous d\'eduisons des informations sur la variation de l'exposant de \mbox{Lyapunov} dans des familles m\'eromorphes de fractions rationnelles complexes, et nous caract\'erisons celles dont les multiplicateurs des points p{\'e}riodiques restent uniform\'ement born\'es (\S~\ref{sec:application}).

Enfin, nous avons regroup\'e dans \S~\ref{sec:open} quelques exemples (\S~\ref{ss:exemples}), ainsi que des conjectures sur la distribution asymptotique des points p\'eriodiques (\S~\ref{ss:questions}).

\subsection*{Remerciements}
Nous remercions les nombreuses institutions qui nous ont financ\'ees tout au long de ce projet initi\'e en 2008,
comme le CNRS, le projet ECOS C07E01, l'ANR-Berko, l'ERC nonarcomp, la Faculdad de Matem\'aticas de la PUC au Chili, l'\'Ecole polytechnique. Nous remercions Laura DeMarco, Xander Faber, et Y\^usuke Okuyama de nous avoir transmis leurs travaux respectifs sur des probl\`emes proches de ceux trait\'es ici, ainsi que pour leurs remarques sur une premi\`ere version de cet article.

\section{Conventions, notations, rappels}
\label{s:preliminaires}

\subsection{Espaces de Berkovich}
\label{ref:def-berko}
Pour tout ce qui concerne les propri\'et\'es de base de la droite projective
au sens de Berkovich nous renvoyons au livre~\cite{baker-rumely}, \`a l'article de survol~\cite{survey-jonsson} et \`a~\cite{theorie-ergo}.

Rappelons que l'on fixe dans tout cet article un corps~${(K,|\cdot|)}$ alg\'ebriquement clos, non trivialement valu\'e, non archim\'edien, complet et de corps r{\'e}siduel~$\tK$.
Notons~$\cO_K$ l'anneau de valuation de~$K$ et~$\fM_K$ l'id\'eal maximal de~$\cO_K$.
On a donc
\begin{equation}
  \label{eq:3}
  \cO_K
  =
  \{ x \in K, |x| \le1\},
  \fM_K
  =
  \{ x \in K, |x| <1\}
  \text{ et }
  \tK
  =
  \cO_K / \fM_K~.
\end{equation}

On note $\AKber$ la droite affine au sens de Berkovich dont les points sont les semi-normes multiplicatives sur l'anneau~$K[T]$ dont la restriction \`a $K$ est \'egale \`a~$|\cdot|$.
On notera~$|P(x)|$ la valeur d'un polyn\^ome $P$ \'evalu\'e en une semi-norme $x \in \AKber$.
De plus, on notera~$\xcan$ le point de~$\AKber$ associ\'e \`a la boule unit\'e, qui correspond {\`a} la norme de Gauss sur l'anneau~$K[T]$.

On munit~$\AKber$ de la topologie de la convergence ponctuelle.
Pour cette topologie, $\AKber$ est localement compact, connexe par arcs, et poss\`ede une structure naturelle d'arbre r\'eel. Pour toute paire de points $x_0\neq x_1 \in\AKber$, il existe un chemin injectif continu $\gamma\colon [0,1] \to \AKber$ tel que $\gamma(0) = x_0$, $\gamma(1) = x_1$, qui est unique \`a reparam\'etrisation pr\`es. L'image $\gamma[0,1]$ est le segment reliant $x_0$ \`a $x_1$ pour la structure d'arbre, et on le note~$[x_0,x_1]$.

L'\'evaluation $P \mapsto |P(x)|$ pour $x \in K$ fournit un plongement
du sous-ensemble des points, dit de type I, not\'e~$\AK$ dans~$\AKber$.
Ces points sont des bouts de l'arbre $\AKber$, c'est-\`a-dire dont le compl\'ementaire reste connexe.

On d\'efinit de mani\`ere analogue la droite projective de Berkovich $\PKber$ en recollant deux copies de $\AKber$ par le morphisme $T\mapsto 1/T$. Cet espace s'obtient topologiquement comme le compactifi\'e d'Alexandorff de la droite affine, $\PKber= \AKber \cup \{ \infty \}$.
C'est \`a nouveau un arbre r\'eel qui contient $\PK= \AK \cup \{ \infty\}$.
Par convention~$\infty$ est un point de type I.

Une boule ferm\'ee de~$\AK$ est un ensemble du type $\{ z, \, |z-y| \le r\}$ pour un point $y \in \AK$ et un r\'eel $r\ge0$.
A toute boule ferm\'ee $B$ est associ\'ee une semi-norme $x_B\in \AKber$ d\'efinie par $|P(x_B)| = \sup_B |P|$.
Un point est dit de type II si il est associ\'e \`a une boule dont le rayon est strictement positif et appartient au groupe des valeurs~$|K^*|$ de~$K$.
Il est de type III si le rayon n'appartient pas \`a~$|K^*|$.
Un point de ${\AKber \setminus \AK}$ qui n'est ni de type~II ni de type~III est dit de type~IV.

On \'etend la notion de boules aux espaces $\PK, \AKber, \PKber$ de la mani\`ere suivante.
Une boule ferm\'ee de $\PK$ est soit une boule ferm\'ee de $\AK$ soit l'image d'une boule ferm\'ee de~$\AK$ par l'inversion $T \mapsto 1/T$.
Une boule ferm\'ee de $\AKber$ (resp. de $\PKber$) est l'adh{\'e}rence d'une boule ferm\'ee de $\AK$ (resp. de $\PK$).
Une boule ouverte de~$\PK$ (resp. de~$\PKber$) est le compl{\'e}mentaire d'une boule ferm{\'e}e, et une boule ouverte de~$\AK$ (resp. de~$\AKber$) est une boule ouverte de~$\PK$ (resp. de~$\PKber$) ne contenant pas~$\infty$.
Dans chacun des espaces~$\PK$ et~$\PKber$, l'ensemble des boules est préserv\'e par~$\PGL(2,K)$, et on obtient donc une action continue de~$\PGL(2,K)$ sur $\PKber$ pour laquelle on a
$|P( \phi(x))| = |(P\circ\phi) (x)|$ d\`es que $P\in K[T]$, $\phi\in \PGL(2,K)$, et $x\in\PKber$ n'est pas un p\^ole de $P\circ \phi$.

Un ouvert fondamental de~$\PKber$ est une composante connexe du compl\'ementaire d'un ensemble fini de points de type~II ou~III.
Ces ouverts forment une base de la topologie de $\PKber$.

La fonction ${|\cdot| \colon \AKber \to \mathopen[ 0, +\infty \mathclose[}$ est d\'efinie par $|x| \= |T|_x$ o\`u~$|\cdot|_x$ d\'esigne la semi-norme sur l'anneau~$K[T]$ associ\'ee au point $x\in\AKber$. De m\^eme, on d\'efinit le diam\`etre (projectif) de $x\in\PKber$ par la formule
\begin{equation}\label{eq:def-diam}\diam(x) \= \frac{\inf_{z \in K} |T-z|_x}{\max\{ 1, |x|^2\}}~.
\end{equation}
On v\'erifie que $\diam (\phi(x)) = \diam(x)$ pour toute $\phi \in \PSL(2,\cO_K)$.
Comme toute application d'\'evaluation est continue on en d\'eduit de plus le

\begin{fact}
  La fonction $x\mapsto \diam(x)$ est semi-continue sup\'erieurement; elle est born\'ee sup\'erieurement par $1$, et s'annule uniquement sur~$\PK$.
\end{fact}

L'espace $\HK$ est muni d'une m\'etrique dite hyperbolique $d_\HK$ qui est compl\`ete, invariante par~$\PGL(2,K)$ et pour laquelle $\HK$ est un arbre r\'eel. Si $x_1, x_2$ sont deux points associ\'es \`a des boules de m\^eme centre et de rayon respectivement $r_1$ et $r_2$, alors $d_\HK(x_1, x_2) = \mathopen|\log r_1 - \log r_2\mathclose|$.

Les espaces $\AKber$ et $\PKber$ sont des espaces annel\'es.
Pour tout ouvert~$U$ de~$\PKber$, on d\'efinit l'espace $\cO(U)$ des fonctions analytiques sur $U$ comme le compl\'et\'e (pour la norme du supremum sur $U \cap \PK$) des fractions rationnelles n'ayant pas de p\^oles dans $U$.

\subsection{Fractions rationnelles}
\label{ss:action}
Fixons une fraction rationnelle $R$ {\`a} coefficients dans~$K$.
Elle induit une application continue de~$\PKber$ dans lui-m\^eme qui est ouverte d\`es que~$R$ est non constante.

Rappelons~\cite[\S~2.2]{theorie-ergo} que l'on d\'efinit le degr\'e local de $R$ en un point $x\in \PKber$ comme l'entier
$$ \deg_R(x)
\=
\dim_{\kappa(R(x))} (\cO_x/ \fm_{R(x)} \cO_x)~,$$
o\`u $\cO_x$ d\'esigne l'anneau des germes de fonctions analytiques en $x$, $\fm_x$ l'id\'eal des fonctions s'annulant en $x$, $\kappa (x)$ le corps $\cO_x/\fm_x$, et on utilise~$R$ pour voir~$\cO_x$ comme un sous-module de~$\cO_{R(x)}$.
Nous renvoyons \`a~\cite[4.6-7]{survey-jonsson} et~\cite[\S2]{R1} pour une d{\'e}finition plus g{\'e}om{\'e}trique.

Soit~$x$ un point de type~II de~$\PKber$, et soient~$g$ et~$\hg$ des transformations de M{\"o}bius telles que ${g(\xcan) = x}$ et ${\hg(R(x)) = \xcan}$.
La r{\'e}duction~$\tR$ de~$\hg \circ R \circ g$ est alors de degr{\'e}~$\deg_R(x)$ \cite[Proposition~2.4]{R1}.
Le degr{\'e} d'ins{\'e}parabilit{\'e} de~$\tR$ est ind{\'e}pendant du choix des transformations de M{\"o}bius~$g$ et~$\hg$.
On le note~$\wdeg{R}(x)$ et l'appelle le \emph{degr{\'e} d'ins{\'e}parabilit{\'e} local de~$R$ en~$x$}.
Si l'on note par~$\cH(x)$ la compl{\'e}tion de~$\kappa(x)$ par rapport {\`a} la norme associ{\'e} {\`a}~$x$, alors~$\wdeg{R}(x)$ correspond au degr{\'e} d'ins{\'e}parabilit{\'e} de l'extension du corps r{\'e}siduel de~$\cH(x)$ sur celui de~$\cH(R(x))$, voir \cite[Proposition~4.13]{survey-jonsson}.
Notons que~$\wdeg{R}(x)$ est une puissance de la caract{\'e}ristique r{\'e}siduelle de~$K$ divisant~$\deg_R(x)$.
En particulier, si l'on d{\'e}signe par~$|\wdeg{R}(x)|$ et~$|\deg_R(x)|$ les normes des entiers~$\wdeg{R}(x)$ et~$\deg_R(x)$ calcul{\'e}es dans~$K$, respectivement, alors on~a
\begin{equation}
  \label{eq:4}
  \wdeg{R}(x)
  \le
  \deg_R(x)
  \text{ et }
  |\wdeg{R}(x)|
  \ge
  |\deg_R(x)|~.
\end{equation}

Lorsque~$x$ est de type III ou IV of~$\PKber$, on se ram\`ene au cas pr\'ec\'edent par un changement de base effectu{\'e} de la mani\`ere suivante.
Soit~$L$ une extension alg\'ebriquement et sph\'eriquement close de~$K$, et soit ${\sigma_{L/K}\colon\PKber \to \PLber}$ l'unique application continue qui envoie le point associ\'e \`a une boule ferm{\'e}e~$B$ de~$K$ au point de~$\PLber$ associ\'e \`a la boule de~$L$ contenant~$B$ et de m{\^e}me diam{\`e}tre, voir \cite[\S~4]{faber1} ou \cite{poineau} pour un cadre plus g\'en\'eral.
Alors~$\sigma_{L/K}(x)$ est un point de type II de~$\PLber$, et on d{\'e}finit le \emph{degr{\'e} d'ins{\'e}parabilit{\'e} local~$\wdeg{R}(x)$ de~$R$ en~$x$} comme celui de~$R$ en~$\sigma_{L/K}(x)$.

Rappelons qu'un point de~$\PKber$ est critique pour~$R$ si ${\deg_R \ge 2}$ en ce point, et qu'on note par~$\crit_R$ l'ensemble de points critiques de~$R$.
On note par~$\crit_R(K)$ l'ensemble l'ensemble des points critiques dans~$\PK$, ${\crit_R(K) \= \crit_R \cap \PK}$.
Nous dirons que~$R$ est \emph{s{\'e}parable en un point~$x$} de~$\HK$ si ${\wdeg{R}(x) = 1}$, et que~$R$ est \emph{ins{\'e}parable en~$x$} sinon, cf. \cite[\S~5.2]{rivera-periode} et~\cite{faber1}.
De plus, un \emph{point critique inséparable de~$R$} est un point de~$\HK$ o{\`u}~$R$ est ins{\'e}parable, et on note par~$\wcrit_R$ l'ensemble des points critiques inséparables de~$R$.
Notons que le degr{\'e} local en chaque point critique inséparable de~$R$ est divisible par la caract{\'e}ristique r{\'e}siduelle de~$K$.

\smallskip

Pour chaque fonction $\varphi \colon \PKber \to \R$, on d{\'e}finit
\[R_*\varphi(y) \= \sum_{x\in R^{-1}(y)} \deg_R(x) \varphi(x)~.\]
Si $\varphi$ est continue, alors $R_*\varphi$ l'est encore et $\sup|R_*\varphi| \le \deg(R) \times \sup|\varphi|$.

Si $\rho$ est une mesure de Radon sur $\PKber$, on d\'efinit $R^* \rho$ par dualit\'e en posant pour chaque fonction continue~$\varphi \colon \PKber \to \R$,
\begin{equation}
  \label{eq:5}
  \int\varphi \dd (R^*\rho)
  \=
  \int (R_*\varphi) \dd \rho~.
\end{equation}
Si $\rho$ est une mesure de probabilit\'e, alors $R^*\rho$ est une mesure positive de masse $\deg(R)$.
Notons que l'{\'e}quation~\eqref{eq:5} est encore valable lorsque~$\varphi$ est la fonction indicatrice d'un bor\'elien de~$\PKber$ par \cite[Lemme~4.4~(1)]{theorie-ergo}.

On montre l'existence d'une unique mesure de probabilit\'e $\rho_R$ dite d'\'equilibre, v\'eri\-fiant $R^* \rho_R = \deg(R) \rho_R$, et ne chargeant aucun point de~$\PK$, voir~\cite[Th{\'e}or{\`e}me~A]{theorie-ergo}.
Cette mesure est $R$-invariante, et pour tout bor\'elien~$A\subset\PKber$ on a ${\rho_R(R^{-1}(A)) = \rho_R(A)}$.
En appliquant la propri{\'e}t{\'e} $R^* \rho_R = \deg(R) \rho_R$ et l'\'equation~\eqref{eq:5} avec~$\varphi$ {\'e}gale {\`a} la fonction indicatrice de~$A$, on obtient
\begin{equation}
  \label{eq:6}
  \rho_R(A)
  \le
  \rho_R( R(A))
  \le
  \rho_R(A) \times \deg(R)~.
\end{equation}
La mesure~$\rho_R$ peut charger un point de~$\HK$.
Dans ce cas, $\rho_R$ est la masse de Dirac en un point de type II, et quitte \`a faire un changement de coordonn\'ees ad\'equats, la r\'eduction de $R$ dans~$\tK$ est de degr\'e \'egal \`a celui de~$R$, voir~\cite[Th{\'e}or{\`e}me~E]{theorie-ergo}.
On dira dans ce cas que $R$ a \emph{bonne r\'eduction potentielle}.

Nous utiliserons de plus les deux lemmes g\'en\'eraux suivants.
Le lemme suivant {\'e}tend \cite[Proposition-D{\'e}finition~2.2]{theorie-ergo} et~\cite[Proposition~2.6]{R1} des ouverts fondamentaux aux ensembles connexes.

\begin{Lem}
  \label{lem:deg}
  Pour chaque sous-ensemble connexe~$A'$ de~$\PKber$, et chaque composante connexe~$A$ de~$R^{-1}(A')$, la fonction
  \[y \mapsto \sum_{x\in A \cap R^{-1}\{ y \}} \deg_R(x)\]
  est constante sur~$A'$.
  Si l'on d{\'e}signe par~$\deg_R(A)$ la valeur de cette fonction, alors on a
  $$
  \rho_R(A)
  =
  \frac{\deg_R(A)}{\deg(R)} \times \rho_R(A')~.
  $$
\end{Lem}

\begin{proof}
  On montre que pour toute composante connexe~$A$ de $R^{-1}(A')$, on a $R(A) = A'$.
  Fixons un point~$x$ de~$A$ et un point~$y'$ de~$A'$ diff{\'e}rent de~$R(x)$.
  Comme l'ensemble~$A'$ est connexe, il est connexe par arcs, et on a~$[R(x), y'] \subseteq A'$.
  Par ailleurs, il existe~$y$ dans~$R^{-1}(y')$ tel que~$R([x, y]) = [R(x), y']$ et tel que~$R$ soit injective sur~$[x, y]$, voir~\cite[Lemme~9.2]{rivera-periode}.
  On a donc~$[x, y] \subseteq R^{-1}(A')$, et par cons{\'e}quent~$[x, y] \subseteq A$ et~$y' \in R(A)$, ce qui montre bien que~$R(A) = A'$.

  Comme chaque composante connexe de~$R^{-1}(A')$ s'envoie de fa{\c{c}}on surjective sur~$A'$, l'ensemble~$R^{-1}(A')$ a au plus~$\deg(R)$ composantes connexes.
  On peut trouver alors un sous-ensemble fini~$F_0$ de $\PKber \setminus R^{-1}(A')$ tel que toute g{\'e}od\'esique qui joint deux composantes connexes distinctes de~$R^{-1}(A')$ intersecte~$F_0$.
  L'ensemble fini~$F = R^{-1}(R(F_0))$ est alors disjoint de~$R^{-1}(A')$, et la composante connexe~$V$ de~$\PKber \setminus F$ contenant~$A$ satisfait~$V \cap R^{-1}(A') = A$.
  De plus, si l'on note~$U$ l'ouvert fondamental composante connexe de~$\PKber \setminus R(F)$ contenant~$A'$, alors~$V$ est une composante connexe de~$R^{-1}(U)$.
  La propri{\'e}t{\'e} voulue est alors donn{\'e}e par~\cite[Proposition~2.6]{R1}.
\end{proof}

Nous appellerons \emph{couronne} tout ouvert connexe de $\PKber$ ayant exactement deux points dans son bord.
Nous n'imposons aucune restriction sur les points du bord qui peuvent \^etre de n'importe quel type.
\'Etant donnés deux points distincts $x \neq \tx$, le compl\'ementaire ${\PKber \setminus \{ x, \tx\}}$ poss\`ede une unique composante connexe
dont le bord est $\{ x, \tx\}$: on l'appellera la \emph{couronne comprise entre $x$ et $\tx$}.

\begin{Lem}
  \label{lem:couronne}
  Soit $C$ une couronne de bord $\{x, \tx\}$, et $R$ une fraction rationnelle telle que le compl{\'e}mentaire de~$R(C)$ dans~$\PKber$ est disconnexe.
  Alors~$R(x) \neq R(\tx)$, $R(C)$ est {\'e}gal {\`a} la couronne~$C'$ comprise entre~$R(x)$ et~$R(\tx)$, et~$C$ est l'une des composantes connexes de~$R^{-1}(C')$.
  De plus, $R$ envoie $] x, \tx [$ de fa{\c{c}}on injective sur~$] R(x), R(\tx) [$ et pour tout~$y$ dans~$] x, \tx [$ on a~$\deg_R(y) = \deg_R(C)$.
\end{Lem}

\begin{proof}
  Comme $R$ est ouverte, l'ensemble~$R(C)$ est un ouvert connexe.
  Par hypoth{\`e}se~$\PKber \setminus R(C)$ est disconnexe, et donc
  la fronti{\`e}re de~$R(C)$ contient au moins deux points.
  Comme la fronti{\`e}re de~$C$ est {\'e}gale {\`a}~$\{x, \tx \}$ et que $R$ est ouverte,
  la fronti{\`e}re de~$R(C)$ est contenue dans~$\{ R(x), R(\tx) \}$.
  On conclut que~$R(x) \neq R(\tx)$ et que la fronti{\`e}re topologique de~$R(C)$ est {\'e}gale {\`a}~$\{ R(x), R(\tx) \}$.
  Ceci montre que~$R(C)$ est la couronne~$C'$ comprise entre~$R(x)$ et~$R(x')$ et que~$C$ est l'une des composantes connexes de~$R^{-1}(C')$.

  Il s'ensuit que~$C$ contient l'une des composantes connexes~$\ell$ de~$R^{-1}(]R(x), R(x')[)$.
  Comme $R$ est finie et ouverte,~$\ell$ est un arbre fini dont les bouts sont contenus dans $R^{-1}( \{ R(x), R(x') \})$.
  On a donc
  \[ \ell
    =
    \mathopen] x, x' \mathclose[
    =
    R^{-1}(\mathopen]R(x), R(x')\mathclose[) \cap C~. \]
  On conclut que~$\deg_R(\mathopen]x, x'\mathclose[) = \deg_R(C)$ et, \`a nouveau car~$R$ est ouverte, que la restriction $R \colon \mathopen] x, x' \mathclose[ \to \mathopen] R(x), R(x') \mathclose[$ est aussi ouverte.
  Cette application est localement injective et par cons{\'e}quent une bijection.
  De~$\deg_R(\mathopen]x, x'\mathclose[) = \deg_R(C)$, on tire que~$\deg_R$ est constante sur~$\mathopen] x, x' \mathclose[$ {\'e}gale {\`a}~$\deg_R(C)$.
\end{proof}

%
%

\subsection{Fraction rationnelles affine Bernoulli}
\label{ss:affine-Bernoulli}

Dans cette section, nous explicitons quel\-ques propri\'et\'es dynamiques des fractions rationnelles affine Bernoulli.
Rappelons que, pour chaque~$n$ dans~$\N^*$, on d{\'e}signe par~$\RFix(R^n)$ l'ensemble des points fixes r\'epulsifs de~$R^n$.

\begin{Prop}
  \label{prop:affine Bernoulli}
  Soit~$R$ une fraction rationnelle {\`a} coefficients dans~$K$ de degr{\'e}~$d \ge 2$.
  Supposons que~$R$ soit affine Bernoulli et consid{\'e}rons l'entier ${k \ge 2}$ et les intervalles $I, I_1, \ldots, I_k$ comme dans la d\'efinition, de telle sorte que ${R^{-1}(I) = I_1 \cup \cdots \cup I_k \subseteq I}$ et que, pour chaque~$j$ dans~$\{1, \ldots, k\}$, l'application ${R \colon I_j \to I}$ soit affine pour~$d_{\HK}$ de facteur de dilatation~$d_j\ge 2$.
  \begin{enumerate}
  \item
    Les propri{\'e}t{\'e}s suivantes sont v\'erifi\'ees:
    \[ 2 \le k \le \sqrt{d},\,
      \sum_{j = 1}^k d_j = d
      \text{ et }
      \sum_{j = 1}^k \frac{1}{d_j} \le 1~.\]
  \item
    Nous avons $J_R = \bigcap_{n = 1}^{+\infty} R^{-n}(I) \subset \crit_R$.
    Lorsque $\sum_{j = 1}^k \frac{1}{d_j} = 1$, $J_R$ est {\'e}gal au segment~$I$, sinon c'est un ensemble de Cantor de longueur nulle. Dans tous les cas l'enveloppe convexe de $J_R$ est un segment $I_*$ dont le bord est constitu\'e de points de type II inclus dans $J_R$ et qui v\'erifie $R(\partial I_*) \subseteq \partial I_*$.
  \item
    Il existe une conjugaison topologique~$\iota \colon \{ 1, \ldots , k \}^\N \to J_R$ entre le d\'ecalage \`a gauche sur $\{ 1, \ldots , k \}^\N$ et~$R|_{J_R}$.
    Chaque point de~$J_R$ poss{\`e}de au plus deux ant\'ec\'edents par~$\iota$ et si l'on note~$D \= \bigcup_{n = 0}^{+\infty} R^{-1}(\partial I_*)$, alors la restriction
    $$ \iota \colon \{ 1, \ldots , k \}^\N \setminus \iota^{-1}(D) \to J_R \setminus D $$
    est une bijection.
  \item
    L'entropie topologique de~$R$ est {\'e}gal {\`a} $\log k$ et~$R$ poss{\`e}de une unique mesure d'entropie m{\'e}trique maximale~$\mu_R$.
    De plus, si l'on munit $\{1, \ldots, k \}$ de la topologie discr\`ete et que l'on note par~$\rho_0$ et~$\mu_0$ les mesures sur~$\{ 1, \ldots, k \}$ donn{\'e}es par
    \begin{displaymath}
      \rho_0 \= \sum_{j = 1}^{k} \frac{d_j}{d} \times \delta_j
      \text{ et }
      \mu_0 \= \sum_{j = 1}^{k} \delta_j,
    \end{displaymath}
    alors~$\rho_R = \iota_* \left( \prod_{\N} \rho_0 \right)$ et~$\mu_R = \iota_* \left( \prod_{\N} \mu_0 \right)$.
    En particulier, $(R, \rho_R)$ est Bernoulli et on a $\mu_R = \rho_R$ si et seulement si~$d_1 = \cdots = d_k$.
  \item
    On a convergence faible des mesures lorsque ${n \to +\infty}$
    \[
      \frac1{d^n}\!\! \sum_{x \in \RFix(R^n)} \deg_{R^n}(x) \delta_x \to \rho_R
      \text{ et }
      \frac1{k^n}\!\! \sum_{x \in \RFix(R^n)} \delta_x \to \mu_R~.
    \]
  \end{enumerate}
\end{Prop}

Pour tout choix d'entiers~$k \ge 2$ et~$d_1, \ldots, d_k$ satisfaisant~$\sum_{j = 1}^k \frac{1}{d_j} \le 1$, on peut trouver une fraction rationnelle affine Bernoulli
$R$ et des segments $I, I_1, \ldots, I_k$ v\'erifiant la d\'efinition,
voir~\cite[\S~5.2]{theorie-ergo}.

\begin{proof}[D{\'e}monstration de la Proposition~\ref{prop:affine Bernoulli}]
  Notons~$\lambda_\HK$ la mesure de Hausdorff $1$-dimension\-nelle sur~$\HK$ induite par~$d_{\HK}$, et posons ${\ell \= \lambda_\HK(I)}$.

  Montrons (1). Le facteur de dilatation $d_j$ de~$R \colon I_j \to I$ est un entier et nous avons
  $\lambda_\HK(I_j)= \ell/d_j$. De l'inclusion~$\bigcup_{j = 1}^{k} I_j \subseteq I$, nous tirons~$\sum_{j = 1}^k \frac{1}{d_j} \le 1$.
  Par ailleurs, pour chaque~$j$ dans~$\{1, \ldots, k \}$ et pour tout~$x$ dans~$I_j$ on a~$\deg_R(x) \ge d_j$ avec {\'e}galit{\'e} en dehors d'un sous-ensemble fini.
  On peut donc choisir un point~$x'$ de~$I$ qui n'est pas l'un des bouts de~$I$ et tel que pour chaque~$j$ dans~$\{1, \ldots, k \}$ l'unique point~$x_j$ de~$R^{-1}(I)$ dans~$I_j$ satisfait~$\deg_R(x_j) = d_j$. Les points~$x_1$, \ldots, $x_k$ sont deux {\`a} deux distincts et on~a
  $$ R^{-1} \{ x' \} = \{ x_1, \ldots, x_k \}
  \text{ de telle sorte que }
  d = \sum_{j = 1}^k \deg_R(x_j) = \sum_{j = 1}^k d_j~. $$
  Final{e}ment, de l'in{\'e}galit{\'e} de Cauchy-Schwartz on tire
  \[
    k = \sum_j \frac1{\sqrt{d_j}} \sqrt{d_j} \le \left(\sum_j \frac1{d_j}\right)^{1/2} \, \left(\sum_j {d_j}\right)^{1/2} = \sqrt{d}~,\] ce qui termine la d{\'e}monstration de~(1).

  Pour le point (2), on observe que~$k \ge 2$ et~$\sum_{j = 1}^k \frac{1}{d_j} \le 1$ implique $d_j \ge 2$  pour tout~$j\in\{1, \ldots, k \}$.
  Combin{\'e} au fait que pour~$x$ dans~$I_j$ on~a $\deg_R(x) \ge d_j$, on obtient
  \[R^{-1}(I) = \bigcup_{j = 1}^d I_j \subset \crit_R~.\]
  Le Lemme~\ref{lem:deg} entra{\^{\i}}ne que~$\bigcap_{n = 1}^{+\infty} R^{-n}(I)$ est {\'e}gal au support de~$\rho_R$ et par cons{\'e}quent que cet ensemble est {\'e}gal {\`a}~$J_R$, voir~\cite[Th{\'e}or{\`e}me~A]{theorie-ergo}.
  Si~$\sum_{j = 1}^{k} \frac{1}{d_j} = 1$, alors
  \begin{equation}\label{eq:fullJ}
    R^{-1}(I) = \bigcup_{j = 1}^d I_j = I
    \text{ et }
    J_R = \bigcap_{n = 1}^{+\infty} R^{-n}(I) = I~.
  \end{equation}
  Montrons que~$J_R$ est un ensemble de Cantor lorsque $\sum_{j = 1}^{k} \frac{1}{d_j} < 1$.
  Notons que~$J_R$ est compact et n'a pas de point isol\'e car~$\rho$ n'a pas d'atome, voir \cite[Corollary~10.60]{baker-rumely} et \cite[Th\'eor\`eme~E]{theorie-ergo}.
  Notons de plus que pour tout entier~$n \ge 1$, l'ensemble~$R^{-n}(I)$ est une union finie d'intervalles de longueur totale~$\ell \times \left(\sum_{j = 1}^{k} \frac{1}{d_j}\right)^n$.
  Ceci entra{\^{\i}}ne que l'ensemble~$J_R = \bigcup_{j = 1}^{k} R^{-n}(I)$ est totalement discontinu, et est par cons{\'e}quent un ensemble de Cantor de longueur nulle.

  Pour montrer la derni{\`e}re assertion de~(2), soit~$I_*$ le plus petit segment de~$\HK$ contenant~$J_R$.
  Alors on a~$I_* \subseteq I$ et~$\partial I_* \subset J_R \subset \crit_R$.
  Par ailleurs, l'enveloppe convexe de~$R^{-1}(I_*)$ est {\'e}gal {\`a}~$I_*$ et par cons{\'e}quent~$\partial I_* \subset \bigcup_{i = 0}^{k} \partial I_i$ et~$R(\partial I_*) \subseteq \partial I_*$.
  Alors chaque point de~$\partial I_*$ est soit p{\'e}riodique r{\'e}pulsif de p{\'e}riode deux, ou son image est un point fixe.
  En particulier, tout point de~$\partial I_*$ est de type~II.

  Pour montrer~(3) notons que pour tout~$i,j\in\{1, \ldots, k \}$, on a
  \[ \lambda_\HK(R^{-1}(I_j) \cap I_i) = \frac{\lambda_\HK(I_j)}{d_i} \le \frac{\lambda_\HK(I_j)}2~.\]
  On en d\'eduit alors ais\'ement les propri{\'e}t{\'e}s suivantes, voir par exemple \cite[Theorem~15.1.5]{katok-hasselblatt}.
  Pour chaque~$\varepsilon = \{ \e_n \}_{n = 0}^{+\infty}$ dans $\{ 1, \ldots , k \}^\N$, l'ensemble~$\bigcap_{n \in \N} R^{-n} (I_{\e_n})$ contient un unique point qu'on note par~$\iota(\varepsilon)$.
  L'application~$\iota \colon \{ 1, \ldots , k \}^\N \to J_R$ ainsi d{\'e}finie est continue, et semi-conjugue le d\'ecalage \`a gauche avec~$R|_{J_R}$. Enfin cette conjugaison est une conjugaison sur le compl\'ementaire de $D \= \bigcup_{n = 0}^{+\infty} R^{-1}(\partial I_*)$.

  L'\'enonc\'e~(4) est donn{\'e}e par~\cite[Lemme~5.5]{theorie-ergo}.

  Finalement pour le point~(5), on note que tous les points p{\'e}riodiques r{\'e}pulsifs de~$R$ sont inclus dans~$J_R$ et que tous les points p\'eriodiques de $J_R$ sont r\'epulsifs.
  L'\'equi\-distribution des points p\'eriodiques vers l'unique mesure d'entropie m{\'e}trique maximale (resp. vers la mesure d'\'equilibre)
  suit de~\cite[Theorem~A]{cli-thom}, en notant que $\mu_R$ (resp. $\rho_R$) est l'image par $\iota$ de l'{\'e}tat d'{\'e}quilibre pour le potentiel constant (resp. pour le potentiel localement constant sur~$\{ j \} \times \{1, \ldots, k \}^{\N^*}$ {\'e}gal {\`a}~$d_j$).
  Ceci compl{\`e}te la d{\'e}monstration de~(5) et de la proposition.
\end{proof}

%
%

\subsection{Rappels de th\'eorie du potentiel}\label{sec:rappel-pot}
La th\'eorie du potentiel est un outil crucial pour d\'efinir la mesure d'\'equilibre.
Nous rappelons donc bri\`evement
quelques d\'efinitions et propri\'et\'es importantes de cette th\'eorie: nous utiliserons les m\^emes
conventions que \cite[\S~2.4]{theorie-ergo} et nous renvoyons \`a cet article pour plus de d\'etails (voir aussi la monographie~\cite{baker-rumely} et l'article de survol~\cite{survey-jonsson}). On note $\mathcal{M}$ l'espace des mesures (sign\'ees) de Radon
de masse finie.
Rappelons qu'on note~$\xcan$ le point de~$\AKber$ associ\'e \`a la boule unit\'e.

\medskip

\'Etant donn\'es trois points $x,y, z\in \HK$, on d\'efinit
$\langle x,y\rangle_z$ comme la distance pour $d_\HK$ entre
$z$ et l'unique point $z' \in [x,y] \cap [y,z] \cap [z,x]$.
\`A toute mesure bor{\'e}lienne $\rho \in \cM$, on associe une fonction
$g_\rho \colon \HK \to \R$ par
\begin{equation}\label{e:def-pot}
  g_\rho(x)
  \=
  - \rho(\PKber) - \int_\PKber \langle x , y \rangle_{\xcan}\, \dd \rho(y),
\end{equation}
et on l'appelle le {\it potentiel} de $\rho$ (sous-entendu bas{\'e} en~$\xcan$). On
a $g_\rho(\xcan) = -\rho (\PKber)$ et par la formule
ci-dessus~$g_{\delta_{\xcan}}$ est la fonction constante {\'e}gale {\`a}
$-1$ sur tout $\PKber$. Plus g{\'e}n{\'e}\-ralement, $g_{\delta_y}(x)
= -1 - \langle x , y \rangle_{\xcan}$.
On d{\'e}signe par $\cP$ l'ensemble de tous les potentiels. C'est un
espace vectoriel qui contient toutes les fonctions d{\'e}finies
sur~$\HK$ et {\`a} valeurs r{\'e}elles de la forme $x \mapsto \langle
x \, , y \rangle_{\xcan}$, et l'application $\rho
\mapsto g_\rho$ induit une bijection entre~$\cM$ et~$\cP$. On peut
donc poser
$$
\Delta g_\rho \= \rho - \rho(\PKber) \times \delta_{\xcan}~.
$$
Ceci d{\'e}finit une application lin{\'e}aire $\Delta \colon \cP \to \cM$ que
l'on appelle {\it le Laplacien}. On v{\'e}rifie $\Delta g =0 $ si et seulement
si $g$ est constante.

Par construction, pour tout $g\in \cP$ on~a $\Delta g ( \PKber) =0$.
R{\'e}ciproquement, toute mesure v{\'e}rifiant $\rho (\PKber ) =0$ est
le Laplacien d'une fonction de $\cP$. Dans toute la suite, on
appellera \emph{potentiel} d'une mesure bor{\'e}lienne $\rho$ toute
fonction $g \in \cP$ telle que $\rho = \Delta g$. Notons que:
$$
\Delta \langle \cdot \, , x \rangle_{\xcan} = \delta_{\xcan} - \delta_x~, \text{ et }
\Delta \log |P| = \sum_{x \in K, P(x) =0} \ord_x(P) \, \delta_x - \deg(P) \delta_{\infty}~,
$$
pour tout polyn\^ome {\`a} coefficients dans~$K$.
De plus, pour tout potentiel $g\in \cP$ et toute fraction
rationnelle non constante~$R$, on~a $R^* (\Delta g) = \Delta (g \circ
R)$ et $R_* \Delta (g) = \Delta (R_*g)$.

\smallskip

On dira qu'une fonction $f$ sur un sous-arbre fini $\cT$ de $\HK$ est affine par morceaux, si elle est continue, et on peut
\'ecrire $\cT$ comme union finie de segments ferm\'es telle que~$f$ est affine sur chacun d'entre eux
lorsqu'on les param\'etrise par la m\'etrique induite par~$d_\HK$.
La proposition suivante est \'equivalente \`a \cite[Proposition~7.56]{valtree}.

\begin{Prop}\label{prop:sympa-sh}~
  \begin{itemize}
  \item
    Soit $\rho$ une mesure atomique support\'ee sur des points de $\HK$,
    et $\cT$ l'enveloppe convexe de son support.
    Alors $\cT$ est un arbre fini, et tout potentiel de $\rho$ est continu, localement constant hors de $\cT$, et
    affine par morceaux sur $\cT$.
  \item
    R\'eciproquement, soit $\cT$ un sous-arbre fini de $\PKber$ dont les bouts sont dans $\HK$,
    Toute fonction de $\PKber$ dans~$\R$ localement constante hors de $\cT$, continue et affine par morceaux sur $\cT$
    est un potentiel dont le Laplacien est une mesure atomique support\'ee sur $\HK$.
  \end{itemize}
\end{Prop}

\begin{Prop}\label{prop:symmetry}
  Soient $g$ un potentiel continu, et $g'$ un potentiel arbitraire. Alors on a
  $$
  \int_{\PKber} g \, \dd \Delta g' =
  \int_{\PKber} g' \, \dd \Delta g~.
  $$
\end{Prop}
\begin{proof}
  Rappelons que pour~$x$ et~$y$ dans~$\PK$ le point~$x \wedge y$ est d{\'e}fini par la propri{\'e}t{\'e} ${[x, \infty] \cap [y , \infty] = [x \wedge y,\infty]}$ et que ${\sup \{ x, y\} \= \diam (x \wedge y)}$.
  Sous nos hypoth\`eses, \cite[Lemmes~4.3 et~4.4]{petits-points} s'appliquent, et on
  obtient la formule
  $$\int_{\AKber} g_0 \, \dd \Delta g'
  =
  \int_{\AKber} g' \, \dd \Delta g_0$$
  o\`u $g_0 (x) \= \int_{\AKber}\log \sup\{ x, y\}\, \dd(\Delta g)(y)$.
  Comme $\Delta \log \sup\{\cdot, x\} = \delta_x - \delta_\infty$ (voir \cite[(20)]{petits-points}),
  on obtient en int\'egrant $\Delta g_0 = \Delta g$, puis $g - g_0 = g(\infty)$.
  Finalement
  \begin{multline*}
    \int_{\PKber} g \, \dd \Delta g'
    =
    \int_{\AKber} g \, \dd \Delta g' + g(\infty)\Delta g' \{\infty\}
    = \\
    \int_{\AKber} g_0 \, \dd \Delta g'
    =
    \int_{\AKber} g' \, \dd \Delta g_0
    =
    \int_{\AKber} g' \, \dd \Delta g~,
  \end{multline*}
  ce qu'il fallait d\'emontrer.
\end{proof}

%
%

\section{La mesure d'\'equilibre charge un segment}\label{sec:zdunik}
Le but de cette partie est de d{\'e}montrer le Th{\'e}or{\`e}me~\ref{thm:main1}, {\'e}nonc{\'e} dans l'introduction.
On fixe~$R$ une fraction rationnelle {\`a} coefficients dans~$K$ de degr{\'e} au moins deux, telle que~$\rho_R$ ne poss{\`e}de pas d'atome et charge un segment de~$\HK$.
Pour simplifier les notations, on \'ecrit ${\rho \= \rho_R}$.
Notons~$\simeq$ la relation sur~$\PKber$ d{\'e}finie par:
$$
x \simeq x'
\text{ si et seulement si }
\rho( [x , x']) =0~.
$$
C'est une relation d'\'equivalence, car~$\rho$ ne charge pas les points.
Le point cl{\'e} est de montrer que l'espace quotient~$\PKber/ \simeq$, not{\'e}~$\cT$, est un segment.
On r\'ealise cela en plusieurs \'etapes.
On montre tout d'abord que~$\cT$ a une structure naturelle d'arbre r{\'e}el m{\'e}trique (\S~\ref{sec:arbre quotient}) et que~$R$ induit une application~$\phi$ sur~$\cT$ (\S~\ref{sec:application quotient}).
Apr\`es avoir montr\'e que~$\cT$ \'etait compact (\S~\ref{sec:compacite de l'arbre}), on d{\'e}finit un degr{\'e} local naturel pour~$\phi$ qui v\'erifie les m\^emes propri\'et\'es formelles que pour les fractions rationnelles (\S~\ref{sec:phi-deg}).
En particulier le degr\'e nous permet de contr\^oler la longueur de l'image d'une sous-partie de~$\cT$, et nous utilisons ce fait pour montrer que
$\phi$ est uniform\'ement expansive~(\S~\ref{sec:expansion-uniforme}).
Une fois que ces propri{\'e}t{\'e}s sont {\'e}tabli{e}s, on montre que~$\cT$ est un segment~(\S~\ref{sec:simplicite de l'arbre}).
On d\'ecrit alors la dynamique de~$\phi$ (\S~\ref{sec:phi}) ce qui nous permet de conclure la preuve du Th{\'e}or{\`e}me~\ref{thm:main1} (\S~\ref{sec:preuve main1}).

%
%

\subsection{L'arbre quotient}
\label{sec:arbre quotient}
Notre premier objectif est de montrer que~$\cT$ poss\`ede une structure naturelle d'arbre m{\'e}trique.
Rappelons qu'un \emph{arbre r\'eel m{\'e}trique} est un espace m{\'e}trique~$(X, d)$ tel que pour tout couple de points~$x,x'\in X$ il existe un unique chemin g\'eod\'esique joignant $x$ \`a~$x'$ dans~$X$, voir par exemple~\cite{GH90}.

Tout arbre r\'eel m{\'e}trique est aussi muni d'une topologie dite faible, dont une base d'ouverts est donn\'ee par les composantes connexes de compl\'ementaire d'ensembles finis.
La topologie induite par la m\'etrique est plus fine que la topologie faible.

\smallskip

On note $\pi \colon \PKber \to \cT$ la projection naturelle.
\begin{Prop}
  \label{prop:arbre quotient}
  La fonction $\hd_{\cT} \colon \PKber \times \PKber \to [0, 1]$ d{\'e}finie par ${\hd_{\cT}(x, x') \= \rho([x, x'])}$ se factorise par une fonction $d_{\cT} \colon \cT \times \cT \to [0, 1]$ qui induit une m\'etrique compl\`e\-te pour laquelle~$\cT$ est un arbre r\'eel m{\'e}trique.

  La mesure de Hausdorff $1$-dimensionnelle~$\lambda$
  associ\'ee \`a~$d_\cT$ est de masse totale $\le 1$, et v\'erifie $\lambda\le \pi_* \rho$.

  L'arbre $\cT$ est \`a la fois compact et s\'equentiellement compact pour la topologie faible. Enfin, pour chaque~$\tau$ en~$\cT$ l'ensemble~$\pi^{-1} \{ \tau \}$ est connexe et ferm{\'e}, et l'application $\pi \colon \PKber \to \cT$ induit une application faiblement continue, qui envoie les segments de~$\PKber$ sur des segments de $\cT$ de mani\`ere monotone.
\end{Prop}

Le reste de cette section est d{\'e}di{\'e}e {\`a} la preuve de la proposition.

Soient~$x, x', y$ et $y'$ dans~$\PKber$ tels que~$x \simeq x'$ et~$y \simeq y'$.
Notons que la diff{\'e}rence sym\'etrique de~$[x, y]$ et~$[x', y']$ est contenue dans l'ensemble~$[x, x'] \cup [y, y']$.
Comme ce d{e}rnier ensemble est de mesure nulle par rapport {\`a}~$\rho$ et que chacun des ensembles~$\pi([x, x'])$ et~$\pi([y, y'])$ est r{\'e}duit {\`a} un point, on a
\begin{equation}
  \label{eq:7}
  \rho([x, y]) = \rho([x', y'])
  \text{ et }
  \pi([x, y]) = \pi([x', y'])~.
\end{equation}
Ceci montre que la fonction~$\hd_{\cT}$ d{\'e}finie dans l'\'enonc{\'e} de la Proposition~\ref{prop:arbre quotient} se factorise par une fonction $d_{\cT} \colon \cT \times \cT \to [0, 1]$.
Pour montrer que~$d_{\cT}$ est une m{\'e}trique, notons d'abord que, comme $\rho$ n'a pas d'atome, $d_\cT(\sigma, \sigma) =0$.
Par ailleurs, la d\'efinition m\^eme de la relation d'\'equivalence $\simeq$ sur $\PKber$ implique $d_\cT (\sigma, \sigma') > 0$ d\`es que $\sigma \neq \sigma'$.
Finalement, l'in{\'e}galit{\'e} triangulaire suit du fait que pour tout triplet de points~$x$, $x'$ et~$x''$ dans~$\PKber$ on a~$[x, x'] \subseteq [x, x''] \cup [x'', x']$.

Observons que~\eqref{eq:7} entra{\^{\i}}ne que pour tout~$\tau$ et~$\tau'$ dans~$\cT$, l'ensemble
$$ [\tau, \tau']
\=
\pi([x, x']) $$
ne d{\'e}pend pas du choix de points~$x\in\pi^{-1} \{ \tau \}$ et~$x'\in\pi^{-1} \{ \tau' \}$.

\begin{Lem}
  \label{lem:propriete geodesique}
  Pour tout~$\tau$ et~$\tau'$ dans~$\cT$ et tout~$\tau''$ dans~$[\tau, \tau']$, on a
  \begin{displaymath}
    d_{\cT}(\tau, \tau'') + d_{\cT}(\tau'', \tau')
    =
    d_{\cT}(\tau, \tau')~.
  \end{displaymath}
  Par ailleurs, l'application $\sigma \mapsto d_{\cT}(\tau, \sigma)$ induit une bijection isom{\'e}trique de~$[\tau, \tau']$ sur~$[0, d_{\cT}(\tau, \tau')]$.
\end{Lem}
\begin{proof}
  Fixons~$x$ dans~$\pi^{-1} \{ \tau \}$ et~$x'$ dans~$\pi^{-1} \{ \tau' \}$.
  Pour montrer la premi{\`e}re assertion, notons que l'on peut choisir~$x''$ dans~$[x, x']$ tel que~$\pi(x'') = \tau''$.
  Comme~$\rho$ n'a pas d'atome, on a
  $$  d_{\cT}(\tau, \tau')
  =
  \rho([x, x'])
  =
  \rho([x, x'']) + \rho([x'', x'])
  =
  d_{\cT}(\tau, \tau'') + d_{\cT}(\tau'', \tau')~. $$
  Pour montrer la deuxi{\`e}me assertion, notons que la fonction~$\sigma \mapsto d_\cT(\tau, \sigma)$ restreinte {\`a}~$[\tau, \tau']$ est strictement croissante et donc injective.
  Son image co\"{\i}ncide avec l'image de l'application $y \mapsto \rho ([x, y])$ restreinte {\`a} $[x, x']$, qui est continue car~$\rho$ ne poss\`ede pas d'atome.
  Ceci montre que $\sigma \mapsto d_\cT(\tau, \sigma)$ induit une bijection de $[\tau, \tau']$ sur~$[0, d_{\cT}(\tau, \tau')]$.
  La premi{\`e}re assertion du lemme implique que c'est une isom{\'e}trie.
\end{proof}

\begin{Lem}
  \label{lem:composantes}
  Soit~$\tau_0$ un point de~$\cT$.
  Alors pour chaque~$\tau$ dans $\cT \setminus \{ \tau_0 \}$ la composante connexe de~$\cT \setminus \{ \tau_0 \}$ contenant~$\tau$ est {\'e}gale {\`a} $\{ \tau' \in \cT, \tau_0 \notin [\tau, \tau'] \}$.
\end{Lem}
\begin{proof}
  Pour chaque~$\tau$ dans~$\cT \setminus \{ \tau_0 \}$ on pose~$\cC(\tau) \= \{ \tau' \in \cT, \tau_0 \notin [\tau, \tau'] \}$.

  Notons que par~\eqref{eq:7}, pour tout~$\sigma$, $\sigma'$ et~$\sigma''$ dans~$\cT$ on a~$[\sigma, \sigma'] \subseteq [\sigma, \sigma''] \cup [\sigma'', \sigma']$.
  Ceci entra{\^{\i}}ne que $\{ \cC(\tau), \tau \in \cT \setminus \{ \tau_0 \} \}$ est une partition de~$\cT \setminus \{ \tau_0 \}$ et, au vu du Lemme~\ref{lem:propriete geodesique}, que pour chaque~$\tau$ dans~$\cT \setminus \{ \tau_0 \}$ l'ensemble~$\cC(\tau)$ est connexe.
  Il suffit donc de montrer que chacun de ces ensembles est ouvert.
  Pour cela, fixons~$\tau$ dans~$\cT \setminus \{ \tau_0 \}$ et soit~$x$ dans~$\PKber \setminus \pi^{-1} \{ \tau_0 \}$ tel que~$\pi(x)$ appartient {\`a} $\cC(\tau)$.
  Si~$y$ est un point de ${\PKber \setminus \pi^{-1} \{ \tau_0 \}}$ tel que~$\pi(y)$ n'est pas dans~$\cC(\tau) = \cC(\pi(x))$, alors~$[x, y]$ contient un point~$x_0$ dans~$\pi^{-1}\{ \tau_0 \}$.
  On a donc
  $$ d_{\cT}(\pi(x), \pi(y))
  =
  \rho([x, y])
  \ge
  \rho([x, x_0])
  =
  d_{\cT}(\pi(x), \tau_0)~. $$
  Ceci montre que la boule pour~$d_{\cT}$ de centre~$\pi(x)$ et rayon~$d_{\cT}(\pi(x), \tau_0)$ est contenue dans~$\cC(\tau)$, et par cons{\'e}quent que~$\cC(\tau)$ est ouvert.
  La preuve du lemme est donc compl{\`e}te.
\end{proof}
\begin{Lem}
  L'espace~$(\cT, d_{\cT})$ est un arbre r\'eel m{\'e}trique.
\end{Lem}
\begin{proof}
  Il reste {\`a} montrer que tout chemin injectif reliant~$\tau$ \`a~$\tau'$ a pour support le segment $[\tau, \tau']$.
  Soit $\gamma\colon [0,1] \to \cT$ un chemin continu injectif reliant ces deux points.
  Alors $\Gamma \=\gamma[0,1] \supseteq [\tau, \tau']$
  car sinon on pourrait trouver un point~$\tau''\in[\tau, \tau']$ distinct de~$\tau$ et~$\tau'$,
  pour lequel $\tau$ et~$\tau'$ sont dans la m\^eme composante connexe de~$\cT \setminus \{ \tau'' \}$
  ce qui contredit le Lemme~\ref{lem:composantes}.

  Supposons maintenant que $\gamma$ contienne un point $\tau''=\gamma(t_0)$ hors du segment $[\tau, \tau']$.
  D'apr{\`e}s ce qui pr\'ec\`ede, on a $\gamma[0,t_0] \supseteq [\tau, \tau'']$, et de m\^eme $\gamma[t_0,1] \supseteq [\tau'', \tau']$. L'injectivit\'e de $\gamma$
  implique $[\tau, \tau''] \cap [\tau'', \tau']= \{\tau''\}$. Or, nous avons $[\tau, \tau''] \cup [\tau'', \tau'] \supseteq [\tau, \tau']$. Comme deux compacts
  recouvrant un segment s'intersectent n\'ecessairement, on obtient la contradiction recherch\'ee.
\end{proof}
Soit $\lambda$ la mesure de Hausdorff $1$-dimensionnelle pour la m\'etrique $d_\cT$.
Par d\'efinition,
$$ \lambda ([\tau,\tau'])
=
d_\cT (\tau, \tau')
\le
\rho (\pi^{-1} [\tau,\tau'])
=
(\pi_*\rho )([\tau,\tau'])~. $$
On a donc $\lambda \le \pi_* \rho $. Comme
$\pi_* \rho $ est une mesure de probabilit\'e, on en d\'eduit que $\lambda$ est de masse $\le 1$, et donc que
le diam\`etre de $\cT$ est born\'e (par $1$).

\begin{Lem}
  \label{lem:ssarbre}
  Pour tout~$\tau$ dans~$\cT$, l'ensemble $\pi^{-1}\{ \tau \}$ est connexe et ferm{\'e}.
  En particulier, $\pi$ envoie segments de~$\PKber$ sur segments de $\cT$ de mani\`ere monotone.
\end{Lem}
\begin{proof}
  Si $x$ et $x'$ sont dans $\pi^{-1}\{\tau \}$ alors on a $[x,x'] \subseteq \pi^{-1}\{\tau \}$ par d\'efinition, donc~$\pi^{-1}\{\tau \}$ est connexe.

  Fixons $x_0 \in \pi^{-1}\{\tau \}$, et prenons $x \notin \pi^{-1}\{\tau \}$. On peut donc trouver
  $x'\in \mathopen]x_0, x\mathclose[$ de telle sorte que $\rho [x_0,x'] >0$. La composante connexe $U$ de $\PKber \setminus\{x'\}$
  contenant $x$ est ouverte, contient $x$, et pour tout $y\in U$ on a $\rho [x_0,y] \ge \rho [x_0,x'] >0$.
  On conclut que le compl\'ementaire de $ \pi^{-1}\{\tau \}$ est ouvert.
\end{proof}

\begin{Rem}
  Il n'est pour le moment pas exclu que $\rho (\pi^{-1}\{ \tau \}) >0$.
\end{Rem}

Rappelons que la famille d'ensembles de la forme
\begin{equation}
  \label{eq:8}
  \cU(\tau_0; \tau_1, \ldots, \tau_n)
  \=
  \bigcap_{i=1}^n \{ \tau \in \cT, \, \tau_i \notin [\tau_0, \tau ] \}
\end{equation}
pour un choix fini de points $\tau_0, \ldots, \tau_n$ dans~$\cT$  forme une base d'ouverts
pour la topologie faible sur $\cT$, que celle-ci est s\'epar\'ee et moins fine que celle induite par la m\'etrique.

Montrons que $\pi$ est faiblement continue.
Prenons $\tau_0, \ldots, \tau_n$ dans~$\cT$, et consid\'erons l'ouvert faible $\cU= \cU(\tau_0; \tau_1, \ldots, \tau_n)$.
Sans perte de g{\'e}n{\'e}ralit{\'e}, on peut supposer que les points~$\tau_0, \ldots, \tau_n$ sont distincts deux \`a deux.
Pour chaque~$i\in\{1, \ldots, n \}$ on note~$\cB_i$ la composante connexe de~$\cT \setminus \{ \tau_i \}$ contenant~$\tau_0$.
Le Lemme~\ref{lem:composantes} entra{\^{\i}}ne~$\cU = \bigcap_{i = 1}^n \cB_i$.
Il suffit donc de noter que pour tout~$i\in \{1, \ldots, n \}$ l'ensemble
$$ \pi^{-1} \cB_i
=
\left\{ x \in \PKber, [x_0, x] \cap \pi^{-1} \{ \tau_i \} = \emptyset \right\} $$
est la composante connexe de~$\PKber \setminus \pi^{-1} \{ \tau_i \}$ contenant~$x_0$, laquelle est ouverte par le Lemme~\ref{lem:ssarbre}.
Ceci montre que~$\pi$ est faiblement continue.
Comme~$\PKber$ est faiblement compact et s\'equentiellement compact, il en est de m\^eme pour~$\cT$.

\medskip

Le lemme suivant conclut la preuve de la Proposition~\ref{prop:arbre quotient}.

\begin{Lem}\label{lem:complet}
  L'espace m\'etrique $(\cT, d_\cT)$ est complet.
\end{Lem}

\begin{proof}
  Soit~$\{ \tau_n \}_{n = 1}^{+\infty}$ une suite de Cauchy pour $d_\cT$.
  Comme~$\cT$ est s\'equentiel\-lement compact, on peut trouver un point d'adh\'erence~$\tau$ pour cette suite (pour la topologie faible).
  On fixe $\e >0$, et $N$ assez grand pour
  que $d_\cT(\tau_n, \tau_m) \le \e$ si $n,m \ge N$.
  Supposons que $d_\cT ( \tau_n , \tau) \ge 2\e$ pour un entier $n\ge N$.
  Pour tout $m \ge N$, notons $\sigma_m$ l'unique point de $[\tau_n, \tau]$
  tel que $ [\tau_n , \tau_m] \cap [\tau_n, \tau] = [ \tau_n , \sigma_m]$.
  Par hypoth\`ese $d_\cT( \tau, \sigma_m) \ge \e$.
  Soit $\sigma \in [ \tau_n , \tau]$ le point \`a distance $\e/2$ de $\tau$.
  Alors la composante connexe de $\cT \setminus \{ \sigma\}$ contenant $\tau$ est \'egal \`a $\cU(\tau;\sigma)$ qui est un ouvert faible ne contenant aucun point $\tau_m$ pour $m \ge N$ ce qui est absurde.
\end{proof}

\subsection{L'application quotient}
\label{sec:application quotient}
On va maintenant \'etudier l'application induite par $R$ sur l'espace quotient $\cT$.
Dans le reste de cette section, pour tout point~$\tau$ dans~$\cT$ et tout r\'eel~$r > 0$, on note
$$ B(\tau, r) \= \{ \tau' \in \cT, d_{\cT}(\tau, \tau') < r \}~. $$

\begin{Prop}
  \label{prop:application quotient}
  Il existe une unique application~$\phi \colon \cT \to \cT$ satisfaisant $\phi \circ \pi = \pi \circ R$.
  Cette application est {\`a} fibres finies, $\deg(R)$\nobreakdash-Lipschitz par rapport {\`a}~$d_{\cT}$ et ouverte par rapport {\`a} la topologie induite par cette distance.
  De plus, la mesure de probabilit\'e~$\pi_* \rho$ est $\phi$-invariante et ergodique.
\end{Prop}

\begin{proof}
  Pour tout~$x, x'\in \PKber$ on a par~\eqref{eq:6},
  \begin{equation}
    \label{eq:9}
    \rho([R(x), R(x')])
    \le
    \rho(R([x, x']))
    \le
    \deg(R) \times \rho([x, x'])~.
  \end{equation}
  En particulier, $x \simeq x'$ entra{\^{\i}}ne~$R(x) \simeq R(x')$, et par cons{\'e}quent~$R$ induit une application~$\phi \colon \cT \to \cT$ uniquement determin{\'e}e par la relation $\phi \circ \pi = \pi \circ R$.
  La cha{\^{\i}}ne d'in{\'e}galit{\'e}s~\eqref{eq:9} entra{\^{\i}}ne que~$\phi$ est $\deg(R)$\nobreakdash-Lipschitz par rapport {\`a}~$d_{\cT}$.
  L'invariance et l'ergodicit\'e de $\pi_* \rho$ sont des cons\'equences directes de
  l'invariance et de l'ergodicit\'e de~$\rho$.

  Nous aurons besoin du lemme suivant.
  Rappelons qu'un arbre est dit fini, si il poss\`ede un nombre fini de bouts.

  \begin{Lem}\label{lem:tree}
    Pour tout sous-ensemble connexe
    $\cA'$ de $\cT$ et toute composante connexe~$\cA$ de $\phi^{-1}(\cA')$,
    l'ensemble $\pi^{-1}(\cA)$ est une composante connexe de $R^{-1}(\pi^{-1}(\cA'))$.
  \end{Lem}

  Pour montrer que~$\phi$ est {\`a} fibres finies, on choisit $\tau \in \cT$.
  Par le Lemme~\ref{lem:tree}, pour chaque composante connexe~$\cC$ de~$\phi^{-1} \{ \tau \}$ l'ensemble~$\pi^{-1}(\cC)$ est une composante conne\-xe de~$R^{-1}(\pi^{-1} \{ \tau \})$.
  Le Lemme~\ref{lem:deg} entra{\^{\i}}ne que~$\phi^{-1}(\tau)$ poss{\`e}de au plus~$\deg(R)$ compsantes connexes, et
  il suffit donc de montrer que toute composante connexe~$\cC$ de~$\phi^{-1}(\tau)$ est r{\'e}duite {\`a} un point.
  Notons simplement que pour chaque~$x$ et~$x'$ dans~$\pi^{-1} (\cC)$ l'ensemble~$R([x, x'])$ est un arbre fini contenu dans~$\pi^{-1}\{\tau \}$, et par cons{\'e}quent on a par~\eqref{eq:6}
  $$ \rho([x, x'])
  \le
  \rho(R([x, x']))
  =
  0~,$$
  ce qui termine la d{\'e}monstration du fait que~$\phi$ est {\`a} fibres finies.

  Pour voir que $\phi$ est ouverte, il suffit de montrer que $\phi(B(\tau,r)) \supseteq B(\phi(\tau), r)$ pour tout $\tau$ et $r>0$.
  Soit donc $\sigma \in B(\phi(\tau), r)$, et choisissons $x\in \pi^{-1}\{\tau\}$, $z\in \pi^{-1}\{\sigma\}$, et notons $y= R(x) \in \pi^{-1}\{\phi(\tau)\}$.
  Notons~$T$ la composante connexe de
  $R^{-1} (J)$ contenant $x$
  avec $J=[y,z]$.
  Comme~$J$ est connexe et ferm\'e, nous avons $R(T) = J$ par le Lemme~\ref{lem:deg}, et
  $$ \rho(T)
  \le
  \rho(J)
  =
  d_{\cT}(\phi(\tau), \sigma)
  \le
  r~, $$
  par~\eqref{eq:6}.
  On a donc $\pi(T) \subseteq B(\tau, r)$, et
  $$ \sigma = \pi(z) \in \pi(J) = \pi(R(T)) = \phi(\pi(T)) \subseteq \phi(B(\tau, r))~, $$
  ce qui conclut la preuve.
\end{proof}
\begin{proof}[D\'emonstration du Lemme~\ref{lem:tree}]
  Comme l'image par~$\pi$ d'un segment d'extr\'emit\'es~$x$ et~$x'$ est encore un segment d'extr\'emit\'e $\pi(x)$ et $\pi(x')$,
  l'ensemble $\pi^{-1}(\cA)$ est connexe et par cons\'equent contenu dans l'une des composantes connexes de $R^{-1}(\pi^{-1}(\cA'))$.
  Par ailleurs, pour tous points~$x, x'$ dans la m{\^e}me composante connexe de~$R^{-1}(\pi^{-1}(\cA'))$, on~a
  $$
  \phi ([\pi(x), \pi(x')])
  =
  \phi(\pi[x, x'])
  =
  \pi (R([x,x'])) \subseteq \cA'~.
  $$
  Il s'ensuit que $\pi(x)$ et $\pi(x')$ appartiennent \`a la m\^eme composante connexe de $\phi^{-1}(\cA)$, ce qui compl\`ete la d\'emonstration.
\end{proof}

Le lemme suivant nous sera aussi utile dans la suite.

\begin{Lem}\label{lem:preimage}
  La pr\'eimage par~$\phi$ de tout segment~$\cI$ de~$\cT$ est une union finie d'arbres finis dont les bouts sont inclus dans $\phi^{-1} ( \partial \cI)$.
\end{Lem}
\begin{proof}
  On peut supposer le segment~$\cI=[\tau_0,\tau_1]$ ferm\'e.
  Soit $\cC$ une composante connexe de $\phi^{-1}(\cI)$.
  Comme~$\phi$ est {\`a} fibres finies, il nous suffit de montrer que les bouts de~$\cC$ sont inclus dans $\phi^{-1} ( \partial \cI)$. Soit $\tau$ un des bouts de $\cC$, et supposons que $\phi(\tau) $ appartienne \`a l'int\'erieur de $\cI$.
  Soit $\cC_i$  la composante connexe de $\phi^{-1}[ \tau_i, \phi(\tau)]$ contenant $\tau$
  pour $i =0,1$. Notons que chacune de ces composantes contient un segment non trivial
  d'origine $\tau$.
  Comme $\tau$ est un bout de $\cC$, l'intersection $\cC_0 \cap \cC_1 = \phi^{-1}(\phi(\tau))$
  contient un segment non trivial de bord $\tau$, ce qui est absurde.
\end{proof}

\subsection{Comparaison des topologies}
\label{sec:compacite de l'arbre}
Nous allons maintenant d\'emontrer la
\begin{Prop}\label{prop:compacite de l'arbre}
  La topologie faible co{\"{\i}}ncide avec la topologie sur~$\cT$ induite par~$d_{\cT}$.
  Par ailleurs, la mesure $\pi_* \rho$ est la mesure de Hausdorff $1$-dimensionnelle de~$\cT$, c'est-{\`a}-dire ${\pi_* \rho = \lambda}$.
  En particulier, $\rho$ ne charge aucun ensemble de la forme $\pi^{-1}\{ \tau\}$.
\end{Prop}

La d{\'e}monstration de cette proposition d{\'e}pend de plusieurs lemmes.

\begin{Lem}\label{lem:noloss}
  Pour chaque sous-arbre fini~$T$ de $\PKber$, on a $\rho(T) = \lambda (\pi(T))$.
\end{Lem}
\begin{proof}
  On peut supposer l'arbre~$T$ ferm{\'e}.
  On d\'emontre le r\'esultat par r\'ecurren\-ce sur le nombre de bouts~$n$ de~$T$.
  Lorsque~$n=2$, c'est-\`a-dire dans le cas o\`u $T$ est un segment, la propri{\'e}t{\'e} d{\'e}sir{\'e}e est donn{\'e}e par la d{\'e}finition de~$d_{\cT}$.

  Dans le cas g\'en\'eral, on \'ecrit $T$ comme l'union d'un segment ferm\'e $I=[x_0,x_1]$ et d'un arbre fini
  $T'$ ayant un bout de moins que $T$ de telle sorte que
  $T'\cap I$ est r\'eduit \`a~$\{x_0\}$.
  Par l'hypoth\`ese de r\'ecurrence, on a alors $\rho(I) = \lambda (\pi(I))$ et $\rho(T') = \lambda (\pi(T'))$. Mais d'une part $\rho$ ne charge pas les points, et donc
  $\rho(T) = \rho(T') + \rho(I)$. Par ailleurs, $\pi(I)\cap\pi(\rho(T'))=\{\pi(x_0)\}$, et donc
  $\lambda(T) = \lambda(T') + \lambda(I)$, ce qui conclut la preuve.
\end{proof}

\begin{Lem}
  \label{lem:presque-fini}
  Pour tout~$\kappa > 0$, il existe un arbre fini~$\cT_0$ dans~$\cT$ tel que~$\lambda(\cT_0) \ge 1 - \kappa$.
\end{Lem}
\begin{proof}
  Par hypoth\`ese, il existe un segment~$I$ de $\HK$ de masse positive pour~$\rho$.
  L'{\'e}rgodicit{\'e} de~$\rho$ montre que l'ensemble $\bigcup_{n\ge 0} R^{-n}(I)$
  est de masse totale.
  Choisissons un entier $N$ tel que
  $\rho (\bigcup_{0\le n\le N} R^{-n}(I) ) \ge 1 - \kappa$.
  Comme~$R^{-n} (I)$ est une union finie d'arbres finis de~$\PKber$ par le Lemme~\ref{lem:preimage}, il s'ensuit que l'enveloppe convexe $T$ de $\bigcup_{n = 0}^N R^{-n}(I) $ est encore un sous-arbre fini.
  Le Lemme~\ref{lem:noloss} implique alors
  $$ \lambda(\pi(T))
  =
  \rho(T)
  \ge
  \rho \left( \bigcup_{n = 0}^N R^{-n}(I) \right)
  \ge
  1 - \kappa~,$$
  ce qui d\'emontre le lemme avec~$\cT_0 = \pi(T)$.
\end{proof}

\begin{proof}[D{\'e}monstration de la Proposition~\ref{prop:compacite de l'arbre}]
  Pour montrer la premi{\`e}re assertion, choisissons ${\tau\in\cT}$ et~$r > 0$ quelconques, et soit~$\cT_0$ l'arbre fini donn{\'e} par le Lemme~\ref{lem:presque-fini} avec ${\kappa = r/2}$.
  Quitte {\`a} augmenter~$\cT_0$, on peut supposer que~$\cT_0$ contient~$\tau$.
  L'intersection ${B(\tau, r/2) \cap \cT_0}$ est connexe, et par cons{\'e}quent est un arbre fini.
  On note~$\Delta$ l'ensemble fini des bouts de cet arbre, et~$\cU$ la composante connexe de~$\cT \setminus \partial\Delta$ contenant~$\tau$.
  C'est un ouvert fondamental, et on a
  $$ \cU
  \subseteq
  B(\cU \cap \cT_0, \lambda(\cT \setminus \cT_0))
  \subseteq
  B(\cU \cap \cT_0, r/2)
  \subseteq
  B(\tau, r). $$
  Ceci montre que la topologie faible co{\"{\i}}ncide avec la topologie sur~$\cT$ induite par~$d_{\cT}$.

  Pour montrer la deuxi{\`e}me assertion, notons que le Lemme~\ref{lem:presque-fini} entra{\^{\i}}ne que~$\lambda$ est une mesure de masse $\ge 1$.
  Comme~$\rho$ est de probabilit{\'e} et que $\pi_* \rho \ge \lambda$ par la Proposition~\ref{prop:arbre quotient}, on conclut que $\pi_* \rho =\lambda$.
\end{proof}

\subsection{Le degr\'e local}
\label{sec:phi-deg}
On va maintenant s'attacher \`a d\'efinir un degr\'e local pour $\phi$.
Com\-me~$\phi$ est \`a fibres finies (Proposition~\ref{prop:application quotient}), pour chaque~$\tau$ dans~$\cT$ l'ensemble $\{ \tau\}$ est une composante connexe de $\phi^{-1}(\phi(\tau))$.
Le Lemme~\ref{lem:tree} appliqu\'e \`a $\cA = \{ \phi(\tau)\}$ et $\cA' =\{ \tau\}$ montre que~$\pi^{-1} \{ \tau \}$ est une composante connexe de~$R^{-1}(\pi^{-1}\{ \phi(\tau) \})$.
En particulier, l'entier $\deg_R(\pi^{-1}(\tau))\le d = \deg(R)$ est bien d{\'e}fini.
\begin{Prop-def}
  \label{def:phi-deg}
  Soit $D(\phi, \cdot)\colon \cT \to \{ 1, \ldots, d \}$ la fonction d{\'e}finie par
  $$ D(\phi, \tau) \= \deg_R(\pi^{-1} \{ \tau \})~. $$
  Alors pour tout sous-ensemble connexe~$\cA'$ de~$\cT$ et toute composante connexe~$\cA$ de $\phi^{-1}(\cA')$, la quantit\'e $D(\phi,\cA) \= \sum_{\cA \cap
    \phi^{-1}(\tau')} D(\phi,\cdot)$ est ind\'ependante de $\tau' \in \cA'$
  et l'on a:
  \begin{equation}\label{e1}
    \lambda(\cA) = \frac{D(\phi,\cA)}{\deg(R)} \times\lambda (\cA')~.
  \end{equation}
\end{Prop-def}
\begin{Rem}\label{rem:surj-connexe}
  La proposition pr\'ec\'edente implique en particulier que pour toute partie connexe $\cA'$ de $\cT$ et toute composante connexe $\cA$ de $\phi^{-1}(\cA')$
  on a $\phi(\cA) = \cA'$. Par ailleurs si on applique~\eqref{e1} \`a $\cT$, on obtient \[ \sum_{\sigma \in \phi^{-1}(\tau)} D(\phi,\sigma)
    =
    \deg(R)~,\]
  pour tout $\tau\in\cT$.
\end{Rem}
\begin{proof}[D{\'e}monstration de la Proposition-D{\'e}finition~\ref{def:phi-deg}]
  Les Lemmes~\ref{lem:deg} et~\ref{lem:tree} entra\^{\i}nent que pour chaque $x' \in \pi^{-1}(\cA')$, on a
  \begin{multline*}
    D(\phi,\cA)
    \=
    \deg_R(\pi^{-1}(\cA))
    =
    \sum_{x\in \pi^{-1}(\cA) \cap R^{-1}(x')} \deg_R(x)
    \\ =
    \sum_{\sigma \in \cA \cap \phi^{-1} \{ \pi(x') \}} \, \, \sum_{x \in \pi^{-1} \{ \sigma \} \cap R^{-1}(x')} \deg_R(x)
    =
    \sum_{\sigma \in \cA\cap \phi^{-1} \{ \pi(x') \} } D(\phi, \sigma)~.
  \end{multline*}
  On obtient finalement par les Lemmes~\ref{lem:deg} et~\ref{lem:tree} et la Proposition~\ref{prop:compacite de l'arbre}
  $$
  \lambda (\cA)
  =
  \rho (\pi^{-1}(\cA))
  =
  \frac{\deg_R (\pi^{-1}(\cA))}{\deg(R)} \times \rho(\pi^{-1}(\cA'))
  =
  \frac{D(\phi,\cA)}{\deg(R)}\times \lambda (\cA')~,
  $$
  ce qui conclut la d\'emonstration.
\end{proof}

Nous aurons aussi besoin d'une notion de degr\'e local le long d'une
branche en un point.
Avant d'\'enoncer le r\'esultat, rappelons quelques d\'efinitions.

Soit $\tau$ un point de~$\cT$.
Une \emph{branche} en~$\tau$ est une classe d'\'equivalence
de points de $\cT \setminus \{ \tau\}$ pour la relation d'{\'e}quivalence $\tau' \equiv \tau''$ si et seulement si
$\mathopen]\tau, \tau'\mathclose] \cap \mathopen] \tau, \tau'' \mathclose] \neq \emptyset$.
Rappelons que la couronne comprise entre deux points~$\tau\neq \tau'$ de~$\cT$ est par d\'efinition la composante connexe du compl\'ementaire de $\{\tau, \tau'\}$ qui contient le segment ouvert $]\tau, \tau'[$.

\begin{Prop-def}
  \label{def:phi-deg-tgt}
  Soit $\tau \in \cT$ et $\vec{v}$ une branche de $\cT$ en $\tau$. Alors il existe un segment non trivial $[\tau, \tau'[$ repr\'esentant $\vec{v}$ sur lequel $\phi$ est injective, et $D(\phi,\cdot)$ est constant sur $]\tau, \tau'[$.

  En particulier, la branche~$\vec{v}'$ en~$\phi(\tau)$ contenant~$\phi(\tau')$ et la valeur de~$D(\phi, \cdot)$ sur~$] \tau, \tau']$ ne d{\'e}pendent que de la branche~$\vec{v}$. On pose~$\phi(\vec{v}) \= \vec{v}'$ et~$D(\phi, \vec{v})\=D(\phi, \cdot)|_{] \tau, \tau']}$.

  Pour toute branche~$\vec{v}'$ en~$\phi(\tau)$ on a
  \[ \sum_{\substack{\vec{v} \text{ branche en } \tau \\ \phi(\vec{v}) = \vec{v}'}} D(\phi, \vec{v})
    =
    D(\phi, \tau)~. \]
\end{Prop-def}

\begin{proof}
  Soit~$\cU$ l'unique composante connexe de~$\cT \setminus \phi^{-1}(\phi(\tau))$ qui contient~$\tau$ dans son adh{\'e}rence et un point $\htau$
  dont le segment $\mathopen]\tau, \htau\mathclose]$ d\'etermine $\vec{v}$.

  Fixons un point~$\ttau$ de~$\phi(\cU)$ et soit~$\tau'$ un point dans~$\mathopen]\tau, \htau\mathclose]$ qui est suffisamment proche de~$\tau$ de telle fa\c{c}on que la couronne~$\cC$ comprise entre~$\tau$ et~$\tau'$ soit contenue dans ${\cU \setminus \phi^{-1}(\ttau)}$.
  L'image $\phi(\cC)$ contient $\phi(\tau)$ dans son bord et s\'epare
  les deux points~$\phi(\tau)$ et~$\ttau$ qui sont alors dans des composantes connexes distinctes de~$\cT \setminus \phi(\cC)$.
  En particulier, le compl{\'e}mentaire de~$\phi(\cC)$ dans~$\cT$ est disconnexe, et par cons{\'e}quent le compl{\'e}mentaire de~$R(\pi^{-1}(\cC)) =\pi^{-1}(\phi(\cC))$ dans~$\PKber$ n'est pas connexe non plus.

  Soient~$x$ et~$x'$ dans~$\pi^{-1}(\tau)$ et~$\pi^{-1}(\tau')$, respectivement.
  Quitte {\`a} remplacer~$x$ (resp.~$x'$) par le point de~$\pi^{-1}(\tau) \cap [x, x']$ (resp. $\pi^{-1}(\tau') \cap [x, x']$) le plus proche de~$x'$ (resp.~$x$), on peut supposer que~$] x, x' [$ est disjoint de~$\pi^{-1}(\tau)$ (resp. $\pi^{-1}(\tau')$).

  La couronne~$C$ de~$\PKber$ comprise entre~$x$ et~$x'$ est disjointe de~$\pi^{-1}(\tau)$ et~$\pi^{-1}(\tau')$.
  Comme le compl{\'e}mentaire de~$\cC$ poss\`ede exactement deux composantes connexes, on conclut que~$C = \pi^{-1}(\cC)$.
  Par cons{\'e}quent, le compl{\'e}mentaire de~$R(C) = R(\pi^{-1}(\cC))$ dans~$\PKber$ est disconnexe.
  Ceci entra{\^{\i}}ne que~$R(C)$ est une couronne, et~$C$ est une composante connexe de~$R^{-1}(R(C))$. Pour tout~$z\in \mathopen] x, x' \mathclose[$ on a donc~$\deg_R(z) = \deg_R(\cC)$, voir Lemme~\ref{lem:couronne}.
  En particulier~$R$ est injective sur~$\mathopen] x, x' \mathclose[$ et par cons{\'e}quent~$\phi$ est injective sur~$] \tau, \tau' [$ et pour tout~$\sigma$ dans ce d{e}rnier ensemble on a~$D(\phi, \sigma) = \deg_R(\cC)$.

  La derni{\`e}re \'equation est une cons{\'e}quence de la premi{\`e}re partie de la proposition et de la propri{\'e}t{\'e} caract{\'e}ristique du degr{\'e} local.
\end{proof}

\subsection{Expansion uniforme}
\label{sec:expansion-uniforme}

\begin{Prop}
  \label{prop:expansion-uniforme}
  Il existe des constantes~$C > 0$ et~$\varepsilon_0, \delta\in \mathopen]0, 1\mathclose[$ telles qu'on ait la propri{\'e}t{\'e} suivante.
  Pour tout sous-ensemble~$\cA'$ de~$\cT$ satisfaisant~$\lambda(\cA') \le \varepsilon_0$, tout entier~$n \ge 1$ et toute composante connexe~$\cA$ de~$\phi^{-n}(\cA')$, on a~$\lambda(\cA) \le C \delta^n \times \lambda(\cA')$.
\end{Prop}

\begin{Lem}
  \label{lem:degre-maximal}
  Posons ${\cF \= \{ \tau \in \cT, D(\phi, \tau) = d \}}$.
  Si~$\cF$ est non-vide, alors il est connexe et compact.
\end{Lem}
\begin{proof}
  Pour montrer que~$\cF$ est connexe, supposons qu'il contiennent deux points distincts~$\tau$ et~$\tau'$.
  Par la Proposition-D{\'e}finition~\ref{def:phi-deg} on a
  $$ \phi(\tau) \neq \phi(\tau'),
  R^{-1} (\pi^{-1} \{ \phi(\tau) \}) = \pi^{-1} \{ \tau \}
  \text{ et }
  R^{-1} (\pi^{-1} \{ \phi(\tau') \}) = \pi^{-1} \{ \tau' \}~. $$
  Soient~$x$ dans~$\pi^{-1}\{ \tau \}$ et~$x'$ dans~$\pi^{-1}\{\tau'\}$.
  Quitte {\`a} remplacer~$x$ (resp.~$x'$) par le point de~$\pi^{-1} \{ \tau \} \cap [x, x']$ (resp.~$\pi^{-1} \{ \tau' \} \cap [x, x']$) le plus proche de~$x'$ (resp.~$x$), on peut supposer que~$] x, x' [$ est disjoint de~$\pi^{-1} \{\tau \}$ (resp.~$\pi^{-1} \{ \tau' \}$).
  Si~$C$ d\'esigne la couronne comprise entre~$x$ et~$x'$,
  alors~$R(C)$ est disjoint de~$\pi^{-1} \{ \phi(\tau) \}$ et~$\pi^{-1} \{ \phi(\tau') \}$ et l'adh{\'e}rence de~$R(C)$ rencontre chacun de ces ensembles.
  On conclut que le compl{\'e}mentaire de~$R(C)$ dans~$\PKber$ est disconnexe et par le Lemme~\ref{lem:couronne} que $\deg_R(y) = d$ pour tout $y\in\mathopen] x, x' \mathclose[$.
  Au vu de la Proposition-D{\'e}finition~\ref{def:phi-deg}, ceci entra{\^{\i}}ne que~$] \tau, \tau' [$ est contenu dans~$\cF$ et termine la d{\'e}monstration que~$\cF$ est connexe.

  Pour montrer que~$\cF$ est compact, soit~$\tau$ dans~$\cT \setminus \cF$.
  Alors il existe~$\tau'\neq \tau$ tel que~$\phi(\tau') = \phi(\tau)$.
  Posons~$r \= d_{\cT}(\tau, \tau')/d$ et soit~$\sigma$ dans~$B(\tau, r)$.
  On note~$\cC$ et~$\cC'$ les composantes connexes de~$\phi^{-1}(\phi[\tau, \sigma]))$ contenant~$\tau$ et~$\tau'$, respectivement.
  Alors~$\cC$ contient~$\sigma$, et par la Proposition-D{\'e}finition~\ref{def:phi-deg}, on a
  $$ \lambda(\cC \cup \cC')
  \le
  \lambda(\phi([\tau, \sigma]))
  \le
  \deg(R) \times \lambda([\tau, \sigma])
  <
  \deg(R) \times r
  =
  d_{\cT}(\tau, \tau')~. $$
  Ceci entra{\^{\i}}ne que~$\cC$ et~$\cC'$ sont disjoints.
  La Proposition-D{\'e}finition~\ref{def:phi-deg} donne ${D(\phi, \cdot) \le d - 1}$ sur~$\cC$, et le point~$\sigma$ n'appartient pas \`a~$\cF$.
  Ceci montre que~$B(\tau, r)$ est disjoint de~$\cF$ et conclut la d{\'e}monstration que~$\cF$ est compact.
\end{proof}

Une partie de la d{\'e}monstration du lemme suivant est inspir{\'e}e de \cite[Lemme~2.12]{theorie-ergo}.

\begin{Lem}
  \label{lem:degre-itere}
  Il existe un entier~$n_0 \ge 1$ tel que pour tout~$\tau$ dans~$\cT$ l'ensemble~$\phi^{-n_0} \{ \tau \}$ contient au moins deux points.
\end{Lem}
\begin{proof}
  Pour chaque entier~$j \ge 1$ on d{\'e}finit
  \[\cF_j
    \=
    \{ \tau \in \cT, D(\phi^j, \tau) = \deg(R)^j \}~.\]
  Notons que~$\cF_1$ est l'ensemble~$\cF$ d{\'e}fini dans l'{\'e}nonc{\'e} du Lemme~\ref{lem:degre-maximal}.
  Il faut donc montrer qu'il existe un entier~$n_0$ tel que~$\cF_{n_0}$ est vide.
  Supposons par l'absurde que pour chaque entier~$j \ge 1$ l'ensemble~$\cF_j$ n'est pas vide.
  Le Lemme~\ref{lem:degre-maximal} entra{\^i}ne que chacun de ces ensembles est connexe et compact.
  De plus~$\cF_j$ est d\'ecroissant avec~$j$ car $D(\phi^{j+1},\tau)= D(\phi^{j},\tau)\, D(\phi, \phi^j(\tau))$.
  L'intersection~$\cF_{\infty} \= \bigcap_{j = 1}^{+\infty} \cF_j$ est donc non-vide, connexe, compacte et invariante par~$\phi$.
  Par cons{\'e}quent~$\cF_{\infty}$ poss{\`e}de un point fixe~$\tau_*$ de~$\phi$, voir~\cite{valtree} ou la ``propri{\'e}t{\'e} de point fixe'' dans \cite{rivera-fatou}.

  Comme le point fixe~$\tau_*$ de~$\phi$ appartient {\`a}~$\cF$, on a~$\phi^{-1} \{ \tau_* \} = \{ \tau_* \}$.
  Ceci entra{\^{\i}}ne que pour chaque componsante connexe~$\cC$ de~$\cT \setminus \{ \tau_* \}$ l'ensemble~$\phi^{-1}(\cC)$ est une r{\'e}union finie de composantes connexes de~$\cT \setminus \{ \tau_* \}$.
  La Proposition-D{\'e}finition~\ref{def:phi-deg} implique que pour chaque composante connexe~$\cC$ de~$\cT \setminus \{ \tau_* \}$ le degr{\'e} $D(\phi, \cC)$ est bien d{\'e}fini et~$\phi(\cC)$ est une composante connexe de~$\cT \setminus \{ \tau_* \}$.
  Par ailleurs, l'ergodicit{\'e} de $\phi$ entra{\^{\i}}ne que les composantes connexes de~$\cT \setminus \{ \tau_* \}$ forment un cycle pour~$\phi$~: on peut {\'e}crire
  $$ \cT \setminus \{ \tau_* \}
  \=
  \{ \cC_0, \ldots, \cC_{n-1} \}~, $$
  de telle fa{\c{c}}on que l'on ait~$\phi(\cC_i) = \cC_{i + 1}$ pour tout $i$ avec la convention $\cC_n = \cC_0$.
  Notons en particulier que pour chaque~$i$ dans~$\{0, \ldots, n-1 \}$ on a~$D(\phi, \cC_i) = \deg(R)$.
  Pour chaque~$i\in \{0, \ldots, n-1 \}$ on note~$\vv_i$ la branche en~$\tau_*$ correspondante {\`a}~$\cC_i$.
  Alors on a~$D(\phi, \vv_i) =d$ et par la Proposition-D{\'e}finition~\ref{def:phi-deg-tgt} on peut choisir ${\tau_{n}= \tau_0, \tau_1, \ldots, \tau_{n-1}}$ dans~$\cC_0, \ldots, \cC_{n-1}$, respectivement, de telle fa{\c{c}}on que pour tout $i \in \{0, \ldots, n-1\}$:
  \begin{itemize}
  \item
    l'application~$\phi$ soit injective sur~$[\tau_*, \tau_i]$ ;
  \item
    pour tout~$\tau\in [\tau_*, \tau_i]$ on ait~$D(\phi, \tau) = d$;
  \item
    et~$\phi([\tau_*, \tau_i]) = [\tau_*, \tau_{i + 1}]$.
  \end{itemize}
  La Proposition-D{\'e}finition~\ref{def:phi-deg} entra{\^{\i}}ne que~$\phi^n$ induit une isom\'etrie de~$[\tau_*, \tau_1]$ sur lui m{\^e}me fixant~$\tau_*$,
  ce qui contredit l'ergodicit{\'e} de $\phi$.
\end{proof}

\begin{Lem}
  \label{lem:degre-semi-local}
  Soit~$n_0 \ge 1$ un entier tel que pour tout~$\tau$ dans~$\cT$ l'ensemble~$\phi^{-n_0} \{ \tau \}$ contienne au moins deux points.
  Alors il existe~$\varepsilon_0 \in \mathopen]0, 1\mathclose[$ tel que pour tout sous-ensemble connexe~$\cA'$ de~$\cT$ satisfaisant~$\lambda(\cA') \le \varepsilon_0$, l'ensemble~$\phi^{-n_0}(\cA')$ est disconnexe.
\end{Lem}
\begin{proof}
  On montre tout d'abord qu'il existe~$\varepsilon_0 \in \mathopen]0, 1\mathclose[$ tel que pour tout~$\tau'$ dans~$\cT$ on peut trouver~$\tau$ et~$\ttau$ dans~$\phi^{-n_0}\{ \tau' \}$ tels que~$d_{\cT}(\tau, \ttau) > \varepsilon_0$.

  On raisonne par l'absurde.
  Il existe alors une suite de points~$\{ \tau_n \}_{n = 1}^{+\infty}$ tels que pour tout~$n$ et tout $\sigma$ et~$\tilde{\sigma}$ dans~$\phi^{-n_0}\{\tau_n\}$ on a~$d_{\cT}(\sigma, \tilde{\sigma}) \le \frac{1}{n}$.
  Quitte {\`a} prendre une sous-suite, on peut supposer que~$\{ \tau_n \}_{n = 1}^{+\infty}$ converge vers un point~$\tau_{\infty}$.
  Soient~$\tau$ et~$\ttau$ dans~$\phi^{-n_0} \{ \tau_{\infty} \}$.
  Alors pour chaque~$n$ les composantes connexes~$\cC$ et~$\tcC$ de~$\phi^{-n_0}([\tau_\infty, \tau_n])$ contenant~$\tau$ et~$\ttau$, respectivement, satisfont~$\lambda(\cC \cup \tcC) \le d_{\cT}(\tau_{\infty}, \tau_n)$.
  Comme chacun des ensembles~$\cC$ et~$\tcC$ rencontre~$\phi^{-n_0}\{ \tau_n \}$, on a
  $$ d_{\cT}(\tau, \ttau)
  \le
  \frac{1}{n} + \lambda(\cC \cup \tcC)
  \le
  \frac{1}{n} + d_{\cT}(\tau_{\infty}, \tau_n)~. $$
  On conclut alors que~$d_{\cT}(\tau, \ttau) = 0$ et par cons{\'e}quent que~$\phi^{-n_0}(\tau_{\infty})$ ne contient qu'un seul point.
  Ceci contredit notre hypoth{\`e}se sur~$n_0$ et montre l'assertion d{\'e}sir{\'e}e.

  Pour montrer le lemme, soit~$\cA'$ un sous-ensemble connexe de~$\cT$ tel que~$\lambda(\cA') \le \varepsilon_0$.
  Soit~$\tau'$ un point de~$\cA'$ et soient~$\tau$ et~$\ttau$ dans~$\phi^{-n_0} \{ \tau' \}$ tels que~$d_{\cT}(\tau, \ttau) > \varepsilon_0$.
  On note~$\cC$ et~$\tcC$ les composantes connexes de~$\phi^{-n_0}(\cA')$ contenant~$\tau$ et~$\ttau$, respectivement.
  Par~\eqref{e1}, on~a
  \[\lambda(\cC \cup \tcC)
    \le
    \lambda(\cA')
    \le
    \varepsilon_0~. \]
  Ceci entra{\^{\i}}ne que~$\cC$ et~$\tcC$ sont disjoints, et termine la d{\'e}monstration du lemme.
\end{proof}

\begin{proof}[D{\'e}monstration de la Proposition~\ref{prop:expansion-uniforme}]
  Soit~$n_0$ et~$\varepsilon_0$ donn{\'e}s par les Lemmes~\ref{lem:degre-itere} et~\ref{lem:degre-semi-local}, respectivement, et posons
  $$ \delta \= \frac{(\deg(R)^{n_0} - 1)^{\frac{1}{n_0}}}{\deg(R)}
  \text{ et }
  C \= \delta^{-(n_0 - 1)}~. $$

  Pour montrer la proposition, soit~$\cA'$ un sous-ensemble connexe de~$\cT$ tel que ${\lambda(\cA') \le \varepsilon_0}$, $n \ge 1$ un entier et~$\cA$ une composante connexe de~$\phi^{-n}(\cA')$.
  Notons que dans le cas ${n \le n_0 - 1}$ on a ${C \delta^n \ge 1}$ et l'assertion d{\'e}sir{\'e}e est alors imm{\'e}diate.
  On peut donc supposer que~$n \ge n_0$.
  L'ensemble~$\phi^j(\cA)$ est une composante connexe de~$\phi^{- (n - j)}(\cA')$ pour tout $j \le n-1$ par la Remarque~\ref{rem:surj-connexe}, et donc
  $\lambda(\phi^j(\cA)) \le \lambda(\cA) \le \varepsilon_0$ pour tout~$j \in \{0, 1, \ldots, n \}$.
  Soient~$q \ge 1$ et~$r\in \{ 0, \ldots, n_0 - 1 \}$ les entiers donn\'es par division euclidienne~$n = q n_0 + r$.
  Le Lemme~\ref{lem:degre-semi-local} entra{\^{\i}}ne que pour tout~$\ell$ dans~$\{0, 1, \ldots, q - 1\}$ on a~$D(\phi^{n_0}, \phi^{\ell n_0}(\cA)) \le \deg(R)^{n_0} - 1$ et par la Proposition-D{\'e}finition~\ref{def:phi-deg}
  $$ \lambda(\phi^{\ell n_0}(\cA))
  \le
  \frac{D(\phi^{n_0}, \phi^{\ell n_0}(\cA))}{\deg(R)^{n_0}} \times \lambda(\phi^{(\ell + 1) n_0}(\cA))
  \le
  \delta^{n_0} \times \lambda(\phi^{(\ell + 1) n_0}(\cA))~. $$
  On a donc par induction
  $$ \lambda(\cA)
  \le
  \delta^{qn_0} \times \lambda(\phi^{q n_0}(\cA))
  \le
  \delta^{qn_0} \times \lambda(\cA')
  \le
  C \delta^n \times \lambda(\cA')~,$$
  ce qui termine la d{\'e}monstration.
\end{proof}

\subsection{Simplicit{\'e} de $\cT$}
\label{sec:simplicite de l'arbre}
Nous allons maintenant d\'emontrer la
\begin{Prop}\label{prop:simplicite de l'arbre}
  L'arbre $\cT$ est un segment.
\end{Prop}
L'id\'ee principale est la suivante.
En un point $\lambda$-g\'en\'erique, $\cT$ ressemble localement \`a un segment (Lemme~\ref{l:besi}).
Si~$\cT$ avait un point de branchement, alors ses pr\'eimages seraient aussi
des points de branchements et l'ergodicit\'e de $\phi$ impliquerait la
densit\'e de ces points. L'objectif est d'obtenir une contradiction \`a
partir de l\`a en faisant des estimations de
longueur quantitative (Lemme~\ref{l:estim}).

Le premier r\'esultat sur lequel nous nous appuierons est classique.
Celui donne une information quantitative sur le fait que localement en presque tout point~$\cT$ ressemble \`a un segment.

\begin{Lem}\label{l:besi}
  Pour $\lambda$-presque tout point~$\tau$ dans~$\cT$, on a
  $$ \lim_{r\to 0} \frac {\lambda
    (B(\tau,r))}{2r}= 1~.
  $$
\end{Lem}

Nous donnons une d{\'e}monstration de ce lemme ci-dessous.

\begin{proof}[Démonstration de la Proposition~\ref{prop:simplicite de l'arbre}]
  Soient~$C > 0$ et~$\varepsilon_0,\delta\in\mathopen]0, 1\mathclose[$ donn{\'e}s par la Proposition~\ref{prop:expansion-uniforme}.

  On proc\`ede par l'absurde, et on choisit $\tau_0$ un point de
  branchement de $\cT$.
  Fixons ${\e\in \mathopen]0, \e_0/3\mathclose[}$ tel qu'il existe trois segments ferm{\'e}s~$\cJ_1$, $\cJ_2$ et~$\cJ_3$ partant de $\tau_0$ qui soient disjoints deux {\`a} deux et chacun de longueur~$3 \e$.
  On note~$\cI_1$, $\cI_2$ et~$\cI_3$ les trois segments ferm{\'e}s respectivement inclus dans $\cJ_1$, $\cJ_2$ et~$\cJ_3$ partant aussi de $\tau_0$ et de longueur $\e$.
  On note $\cY = \cI_1 \cup \cI_2 \cup \cI_3$.
  On a alors le r\'esultat suivant.
  \begin{Lem}\label{l:estim}
    Pour tout entier~$n\ge 0$ et pour tout~$\tau\in\phi^{-n}(\cY)$ il existe~$r\in \mathopen]0, C \delta^n\mathclose[$ tel que
    \[
      \lambda \left(
        B(\tau,r)
      \right)\ge \frac73 \, r
      ~.
    \]
  \end{Lem}
  Comme $\lambda$ est ergodique, pour $\lambda$-presque tout~$\tau$ dans~$\cT$ l'ensemble des entiers $n$ pour
  lesquels $\tau \in \phi^{-n}(\cY)$ est infini.
  Du Lemme~\ref{l:estim}, on tire
  \[
    \frac{\lambda (B(\tau,r))}{2r}
    \ge
    \frac{7 r}{6r}
    =
    \frac76
    >
    1~,
  \]
  ce pour une infinit\'e de valeurs de $r$ s'accumulant sur $0$.
  Ceci contredit le Lemme~\ref{l:besi}, et conclut la preuve.
\end{proof}

\begin{proof}[D{\'e}monstration du Lemme~\ref{l:besi}]
  Fixons~$\varepsilon > 0$ et soit~$\cT_0$ l'arbre fini donn{\'e} par le Lemme~\ref{lem:presque-fini} avec~$\kappa = \varepsilon^2$.
  On note~$\cB$ l'ensemble fini form{\'e} des bouts et des points de branchement de~$\cT_0$.
  On d{\'e}finit par r\'ecurrence une suite de partitions~$(\sP_{\ell})_{\ell \ge 1}$ de~$\cT_0 \setminus \cB$ en segments comme suit. La partition~$\sP_0$  est celle donn\'ee par les composantes connexes de~$\cT_0 \setminus \cB$. La partition $\sP_\ell$ est obtenue en divisant en deux segments de longueur \'egale chaque segment de $\sP_{\ell -1}$.
  Chaque {\'e}l{\'e}ment~$\cL$ de~$\sP_\ell$ est donc contenu dans un {\'e}l{\'e}ment~$\cL'$ de~$\sP_{\ell - 1}$ satisfaisant~$\lambda(\cL') = 2 \times \lambda(\cL)$.
  On pose~$\sP \= \bigcup_{\ell \ge 0} \sP_\ell$ et pour chaque~$\cL$ en~$\sP$ on note~$\hcL$ le segment obtenu en prenant la fermeture de l'intersection de~$B(\cL, \lambda(\cL))$ et la composante connexe de~$\cT_0 \setminus \cB$ contenant~$\cL$.

  Comme~$\cT_0$ est connexe et ferm{\'e}, la r\'etraction~$p \colon \cT \to \cT_0$ est continue, et la mesure ${\eta \= p_* \left( \lambda|_{\cT \setminus \cT_0} \right)}$ est positive de masse au plus~$\epsilon^2$.
  On note par~$\sE$ le sous-ensemble de~$\sP$ de tous les segments~$\cL$ tels que~$\eta(\cL) \ge \varepsilon \times \lambda(\cL)$ et par~$\sE'$ ceux qui sont maximaux pour l'inclusion.
  Les {\'e}l{\'e}ments de~$\sE'$ sont deux {\`a} deux disjoints et par cons{\'e}quent l'ensemble~$\hcE \= \bigcup_{\cL \in \sE'} \hcL$ satisfait
  $$
  \lambda(\hcE)
  \le
  \sum_{\cL \in \sE'} \lambda(\hcL)
  \le
  3 \times \sum_{\cL \in \sE'} \lambda(\cL)
  \le
  3 \varepsilon^{-1} \times \sum_{\cL \in \sE'} \eta(\cL)
  \le
  3 \varepsilon^{-1} \times \lambda(\cT \setminus \cT_0)
  \le
  3 \varepsilon~.
  $$
  Par cons{\'e}quent, si l'on fixe~$r_0 > 0$ suffisamment petit pour que~$\lambda(B(\cB, 2r_0)) \le \varepsilon$, alors l'ensemble~$\cS_{\varepsilon} \= \cT_0 \setminus \left( B(\cB, 2r_0) \cup \hcE \right)$ satisfait
  \begin{equation}
    \label{eq:10}
    \lambda(\cS_{\varepsilon})
    \ge
    \lambda(\cT_0) - 4 \varepsilon
    \ge
    1 - \varepsilon^2 - 4 \varepsilon~.
  \end{equation}
  Quitte {\`a} r{\'e}duire~$r_0$, on suppose que pour chaque~$\cL$ dans~$\sP_0$ on a~$r_0 < \lambda(\cL)$.

  Fixons~$\tau$ dans~$\cS_{\varepsilon}$ et~$r$ dans~$\mathopen]0, r_0\mathclose[$, et notons qu'il existe un unique {\'e}l{\'e}ment~$\cL_0$ de~$\sP$ contenant~$\tau$ et tel que~$\lambda(\cL_0) \le r < 2 \times \lambda(\cL_0)$.
  Soit~$\tcL_0$ l'{\'e}l{\'e}ment de~$\sP_0$ contenant~$\cL_0$.
  Comme~$\tau$ n'appartient pas {\`a}~$B(\cB, 2r_0)$, la fermeture de l'ensemble~$\cL_0$ est disjointe de~$\cB$.
  Par cons{\'e}quent, il existe des segments~$\cL_1$ et~$\cL_2$ dans~$\sP$ de m{\^e}me longueur que~$\cL_0$ et chacun ayant un bout en commun avec~$\cL_0$.
  On a donc~$p(B(\tau, r)) \subseteq \cL_0 \cup \cL_1 \cup \cL_2 \subseteq \tcL_0$, d'o{\`u}
  \begin{multline}
    \label{eq:11}
    1
    \le
    \frac{\lambda(B(\tau, r))}{2r}
    \le
    \frac{\lambda(B(\tau, r) \cap \tcL_0) + \eta(\cL_0 \cup \cL_1 \cup \cL_2)}{2r}
    \\ =
    1
    +
    \frac{\eta(\cL_0) + \eta(\cL_1) + \eta(\cL_2)}{2r}~.
  \end{multline}
  Pour finir la preuve, on montrera qu'aucun des segments~$\cL_0$, $\cL_1$ ou~$\cL_2$ n'appartient {\`a}~$\sE$.
  Comme~$\cL_0$ contient~$\tau$ et par hypoth{\`e}se~$\tau$ n'est pas dans~$\hcE$, le segment~$\cL_0$ n'appartient pas {\`a}~$\sE$.
  Soit~$i$ dans~$\{1, 2 \}$ et supposons par l'absurde que~$\cL_i$ est dans~$\sE$.
  Si on note par~$\hcL_i'$ l'{\'e}l{\'e}ment de~$\sE'$ contenant~$\cL_i$, alors on a~$\cL_0 \subseteq \hcL_i \subseteq \widehat{\cL_i'} \subseteq \hcE$, ce qui contredit notre choix de~$\tau$.
  On a donc
  $$ \eta(\cL_0) + \eta(\cL_1) + \eta(\cL_2)
  \le
  3 \varepsilon \times \lambda(\cL_0)
  \le
  6 r \times \varepsilon~, $$
  et en combinaison avec~\eqref{eq:10} et~\eqref{eq:11}, ceci compl\`ete la d{\'e}monstration.
\end{proof}

\begin{proof}[D\'emonstration du Lemme~\ref{l:estim}]
  Fixons $n\ge 1$ et $\tau \in \phi^{-n}(\cY)$.
  On peut supposer sans perdre de g\'en\'eralit\'e que $\phi^n(\tau) \in \cI_1$.
  On note $\tau_1\neq \tau_0$ l'autre point du bord de~$\cI_1$, et~$\cA$ la composante connexe de $\phi^{-n}(\cI_1)$ qui contient~$\tau$.
  Notons que par la Proposition~\ref{prop:expansion-uniforme} on~a
  $$ r_n
  \= \lambda(\cA)
  \le
  C\delta^n \times \varepsilon
  \le
  C \delta^n/3~. $$
  Posons~$\cK_1 \= \overline{\cJ_1 \setminus \cI_1}$, $\cK_2 \= \cJ_2$ et~$\cK_3 \= \cJ_3$ et pour~$i\in\{1, 2, 3 \}$.
  Soit~$\sC_i$ la collection des composantes connexes de~$\phi^{-n}(\cK_i)$ qui rencontrent~$\cA$. Pour chaque $i$, on consid\`ere deux cas distincts.
  Soit on peut trouver une composante $\cC\in \sC_i$ qui n'est pas contenue dans~$B(\cA, 2r_n)$ et alors
  $$
  \lambda (B(\cA, 2r_n) \cap \phi^{-n}(\cK_i))
  \ge
  \sum_{\cC' \in \sC_i} \lambda \left(B(\cA, 2r_n) \cap \cC' \right)
  \ge
  \lambda \left(B(\cA,2r_n) \cap \cC \right)
  \ge
  2 r_n~.
  $$
  Soit on a $\cC \subseteq B(\cA, 2r_n)$ pour tout $\cC\in \sC_i$.
  Dans ce cas, pour estimer la longueur~$\lambda (\cC)$, on remarque
  que $\cC$ contient un unique point $\sigma $ dans $ \phi^{-n}(\tau_1) \cap \cA$ si $i=1$, ou dans $ \phi^{-n}(\tau_0) \cap \cA$ sinon car $\cA$ est connexe et $\phi$ est \`a fibres finies.
  Lorsque~$i = 1$ la Proposition-D{\'e}finition~\ref{def:phi-deg} entra{\^{\i}}ne alors
  \begin{equation*}
    \begin{split}
      \lambda \left( B(\cA, 2r_n) \cap \phi^{-n}(\cK_1) \right)
      & \ge
        \sum_{\cC \in \sC_1} \lambda \left( \cC \right)
      \\ & =
           \sum_{\sigma \in \phi^{-n}(\tau_1) \cap \cA} \left(\sum_{\cC \in \sC_1, \sigma \in \cC} \frac{D(\phi^n, \cC)}{\deg(R)^n} \times \lambda (\cK_1)\right)
      \\ & \mathop{=}\limits^{\eqref{e1}}
           \sum_{\sigma \in \phi^{-n}(\tau_1) \cap \cA} \frac{D(\phi^n, \sigma)}{\deg(R)^n} \times 2 \varepsilon
      \\ & =
           \frac{D(\phi^n, \cA)}{\deg(R)^n} \times 2 \lambda(\cI_1)
      \\ & =
           2r_n~.
    \end{split}
  \end{equation*}
  Lorsque~$i = 2$ ou~$3$, on a de fa{\c{c}}on similaire
  $  \lambda \left( B(\cA, 2r_n)\right) \ge 2r_n$ si il existe une composante $\cC\in \cC_i$ qui n'est pas incluse dans
  $ B(\cA, 2r_n)$, ou sinon
  \begin{multline*}
    \lambda \left( B(\cA, 2r_n) \cap \phi^{-n} (\cK_i) \right)
    \ge
    \sum_{\cC \in \sC_i} \lambda \left( \cC \right)
    \\ =
    \sum_{\sigma \in \phi^{-n}(\tau_0) \cap \cA} \left(\sum_{\cC \in \sC_i, \sigma \in \cC} \frac{D(\phi^n, \cC)}{\deg(R)^n} \times \lambda (\cK_i)\right)
    \\ \mathop{=}\limits^{\eqref{e1}}
    \sum_{\sigma \in \phi^{-n}(\tau_0) \cap \cA} \frac{D(\phi^n, \sigma)}{\deg(R)^n} \times 3 \varepsilon
    =
    \frac{D(\phi^n, \cA)}{\deg(R)^n} \times 3 \lambda(\cI_1)
    =
    3r_n~.
  \end{multline*}
  Dans tous les cas, on obtient
  \[
    \lambda\left( B(\tau, 3r_n)\right)
    \ge
    \lambda \left( B(\cA, 2r_n) \right)
    \ge
    \lambda(\cA)
    +
    \sum_{i = 1}^3 \lambda \left( B(\cA, 2r_n) \cap \phi^{-n}(\cK_i) \right)
    \ge
    7r_n~,
  \] ce qui conclut la preuve.
\end{proof}

%
%

\subsection{Simplicit{\'e} de $\phi$}
\label{sec:phi}
Maintenant que l'on sait que~$\cT$ est un segment, on peut d\'eterminer la structure de l'application induite~$\phi$.
On fixe~$\tau_0$ un bout de~$\cT$, de telle sorte que la distance $d_\cT(\cdot ,\tau_0)$ induit une isom\'etrie du segment r\'eel $[0,1]$
sur $\cT$. En particulier $\cT$ porte une structure affine, et un ordre $\le$ pour lequel $\tau_0$ est minimal.

\begin{Prop}
  \label{prop:phi}
  Il existe des entiers~$k$, $d_1$, \ldots, $d_k \ge 2$ tels que~$\sum_{i = 1}^k d_i = d$ et tels que, si l'on note par $\tau_0 < \tau_1 < \cdots < \tau_k$ les points satisfaisant~$d_{\cT}(\tau_i, \tau_{i + 1}) = d_i/d$ pour chaque~$i\in\{0, \ldots, k - 1 \}$, alors l'application restreinte $\phi\colon [\tau_i, \tau_{i+1}] \to \cT$ est affine de facteur de dilatation~$d/{d_i}$ et surjective.
\end{Prop}

\begin{proof}
  L'ensemble $\phi^{-1} (\partial \cT)$ contient $\partial \cT$ par le Lemme~\ref{lem:preimage}, et est fini par~\eqref{e1}.
  On note par ordre croissant $\tau_0 < \tau_1 < \cdots < \tau_k$ les {\'e}l{\'e}ments de~$\phi^{-1}(\partial \cT)$.

  Prenons $\tau$ dans $\cT$ qui n'est pas un bout.
  Par la Proposition-D{\'e}finition~\ref{def:phi-deg-tgt}, on peut trouver $\tau_- < \tau$, et $\tau_+ > \tau$ ainsi que deux entiers $d_+$ et $d_-$ tels que $D(\phi,\cdot) = d_\pm$ sur $] \tau, \tau_\pm [$, et la restriction de $\phi$ \`a $\mathopen] \tau, \tau_\pm \mathclose[$ est injective.
  Deux \'eventualit\'es sont alors possibles.
  Soit~$\phi$ est localement monotone en $\tau$.
  Dans ce cas, $\phi$ est injective sur $] \tau_-, \tau_+ [$ et~\eqref{e1} implique $d_+ = d_-$.
  Soit~$\phi$ n'est pas monotone en $\tau$, et $\phi(\tau) \in \partial \cT$, car $\phi$ est ouverte (Proposition~\ref{prop:application quotient}).

  Ceci montre que pour tout~$i$ dans~$\{1, \ldots, k \}$ l'application~$\phi$ est injective et que~$D(\phi,\cdot)$ est constante sur $] \tau_{i - 1}, \tau_i [$.
  Notons que~$k = 1$ entra{\^{\i}}ne que~$\phi$ est l'identit{\'e} sur~$\cT$, ce qui contredit l'ergodicit{\'e}.

  Si l'on note~$d_i$ la valeur de~$D(\phi, \cdot)$ sur~$] \tau_{i - 1}, \tau_i [$, alors on a ${d_{\cT}(\tau_i, \tau_{i + 1}) = d_i/d}$ par~\eqref{e1}, et par cons{\'e}quent~$d_i \ge 2$, car~$k \ge 2$, et~$\sum_{i = 1}^{k} d_i = d$.
  Notons aussi que~\eqref{e1} entra{\^{\i}}ne
  $$ d_\cT(\phi(\tau), \phi(\tau'))
  =
  \frac{d}{d_i} \times d_\cT(\tau, \tau')~,$$pour chaque~$i$ dans~$\{1, \ldots, k \}$ et tout~$\tau$ et~$\tau'$ dans~$\mathopen] \tau_{i - 1}, \tau_i \mathclose[$.
\end{proof}

%
%

\subsection{D\'emonstration du Th\'eor\`eme~\ref{thm:main1}}
\label{sec:preuve main1}

Nous allons montrer que~$R$ est affine Bernoulli.
Le quotient~$\cT$, d\'efini au d\'ebut de cette section, est un segment non trivial par la Proposition~\ref{prop:phi}.
Soient~$x_0$ et~$x_1$ dans~$\PKber$ tels que~$\pi(x_0)$ et~$\pi(x_1)$ sont les extr{\'e}mit{\'e}s de~$\cT$.
On peut toujours supposer que~$x_0$ et~$x_1$ sont dans $J_R$, car ce dernier ensemble est ferm{\'e} et~$\rho$ ne charge pas~$F_R$ par \cite[Th\'eor\`eme~A]{theorie-ergo}.
Nous allons montrer par l'absurde que
\begin{equation}\label{e:julia}
  J_R \subseteq [x_0,x_1]~.
\end{equation}
Soit donc~$x$ dans~$J_R \setminus [x_0,x_1]$ et soit~$U$ un ouvert fondamental contenant~$x$ et n'intersectant pas~$[x_0,x_1]$.
De~\cite{theorie-ergo}, on tire l'existence d'un entier positif non nul $n$ et d'\'el\'ements ${x'_i\in R^{-n}(x_i)\cap U}$, pour $i=0,1$.
Le segment~$[x'_0,x'_1]$ est inclus dans $U$ et son image
par $R^n$ contient $[x_0,x_1]$ dont la masse pour $\rho$ est
positive.
On en d\'eduit que $\rho( [x'_0,x'_1]) >0$ par~\eqref{eq:6} ce qui
implique que l'un des deux points $\pi(x'_0)$ ou $\pi(x'_1)$ n'appartient pas au segment
$[\pi(x_0), \pi(x_1)]$ ce qui est contradictoire.

Dans la suite de la preuve, on note $I= [x_0,x_1]$ qui contient
donc $J_R$.
\begin{Lem}\label{l:3}
  Pour tout $\tau \in \cT$, on a l'alternative suivante: soit $\pi^{-1} \{ \tau \} \cap I$ contient un unique point et ce point appartient {\`a}~$J_R$~; soit~$\pi^{-1} \{ \tau \} \cap I$ est un intervalle dont les seuls points appartenant {\`a}~$J_R$ sont ses extr{\'e}mit{\'e}s.
  De plus, l'ensemble des~$\tau$ pour lesquels $\pi^{-1} \{ \tau \} \cap I$ est un intervalle est au plus d{\'e}nombrable.
\end{Lem}
\begin{proof}
  Comme~$\pi^{-1} \{ \tau \}$ est connexe, son intersection avec~$I$ est soit r\'eduit \`a un point ou est un intervalle.
  Soit~$x$ un point de~$\pi^{-1} \{ \tau \}\cap I$ appartenant {\`a}~$F_R$.
  Alors on peut trouver~$x'$ et~$x''$ dans~$I$ tels que l'intervalle~$\mathopen] x', x'' \mathclose[$ contient~$x$ et est contenu dans~$F_R$.
  Comme~$J_R \subset I$, ceci entra{\^{\i}}ne que la couronne comprise entre~$x'$ et~$x''$ est contenue dans~$F_R$ et par cons{\'e}quent que~$\rho(\mathopen] x', x'' \mathclose[) = 0$ et~$\pi([x', x'']) = \{ \tau \}$.
  Ceci montre que dans le cas o{\`u}~$\pi^{-1} \{ \tau \} \cap I$ contient un unique point ce point est dans~$J_R$, et que dans le cas o{\`u}~$\pi^{-1} \{ \tau \} \cap I$ est un intervalle, les extr{\'e}mit{\'e}s sont contenues dans~$J_R$.

  Il reste {\`a} montrer que si~$\pi^{-1} \{ \tau \} \cap I$ est un intervalle d'extr{\'e}mit{\'e}s~$x$ et~$x'$, alors ${] x, x' [ \subset F_R}$.
  Comme~$x$ et~$x'$ d\'efinissent le m\^eme point dans~$\cT$, on a~$\rho ( [x, x']) =0$.
  Par ailleurs, $J_R \subset I$, et donc la couronne comprise entre~$x$ et~$x'$ est de masse nulle et donc incluse dans~$F_R$.
  En particulier, on a ${] x, x' [ \subset F_R}$.

  On peut identifier $I$ \`a un segment r\'eel.
  Lorsque $\tau$ poss\`ede deux ant\'ec\'edents ${x, x'\in J_R}$,
  alors le segment $]x, x'[$ est non-trivial et contient donc un
  rationnel.
  L'ensemble de ces points est donc d\'enombrable.
\end{proof}

\begin{proof}[D\'emonstration du Th\'eor\`eme~\ref{thm:main1}]
  Soient~$x_0$, $x_1$ et~$I$ comme ci-dessus et soient~$k$, $d_1, \ldots, d_k$, $\tau_0, \ldots, \tau_k$ donn{\'e}s par la Proposition~\ref{prop:phi}.
  Posons~$y_0 \= x_0$ et $y_k' \= x_1$ et pour chaque~$j$ dans~$\{2, \ldots, k - 1 \}$ soit~$y_j$ (resp.~$y_j'$) le point de~$\pi^{-1} \{ \tau_j \} \cap I$ le plus proche de~$y_0$ (resp.~$y_k'$).
  Comme $x_0, x_1\in J_R$, le Lemme~\ref{l:3} donne \[R^{-1}(\{x_0, x_1\}) = \{y_0, \ldots, y_{k - 1}, y_1', \ldots, y_k' \}~.\]
  Fixons~$j$ dans~$\{1, \ldots, k \}$, posons~$I_j = [y_{j - 1}, y_j']$ et notons~$C_j$ la couronne comprise entre~$y_{j - 1}$ et~$y_j'$.
  Clairement~$R(C_j)$ contient~$\mathopen]x_0, x_1\mathclose[$ par connexit{\'e} et~$R(C_j)$ est disjoint de~$\{x_0, x_1 \}$.
  On conclut que le compl\'ementaire de~$R(C_j)$ poss{\`e}de aux moins deux composantes connexes.
  Le Lemme~\ref{lem:couronne} implique que~$R(C_j)$ est la couronne comprise entre~$x_0$ et~$x_1$, que~$R$ envoie~$\mathopen]y_{j - 1}, y_j'\mathclose[$ de fa\c{c}on bijective sur~$]x_0, x_1[$ et que~$\deg_R$ est constante sur~$\mathopen]y_{j - 1}, y_j'\mathclose[$.
  La Proposition~\ref{prop:phi} entra{\^{\i}}ne alors qu'on a~$\deg_R(C_j) = d_j$ et par cons{\'e}quent que pour tout~$y$ dans~$\mathopen] y_{j - 1}, y_j' \mathclose[$ on a~$\deg_R(y) = d_j$.
  En particulier, $R \colon I_j \to I$ est affine et surjective de facteur de dilatation~$d_j$.
  Par ailleurs, l'{\'e}quation $\sum_{j = 1}^k d_j = d$, donn{\'e}e par la Proposition~\ref{prop:phi}, entra{\^{\i}}ne qu'on a~$R^{-1}(\mathopen] x_0, x_1 \mathclose[) = \bigcup_{i = 1}^{k} \mathopen] y_{i - 1}, y_i' \mathclose[$ et par cons{\'e}quent~$R^{-1}(I) = \bigcup_{j = 1}^{k} I_j \subseteq I$.
  Ceci montre que~$R$ est affine Bernoulli.

  Pour montrer la derni{\`e}re assertion du th{\'e}or{\`e}me, supposons que~$\rho_R$ n'est pas \'etrang\`ere {\`a} la mesure de probabilit{\'e}~$\rho_I$ sur~$I$ proportionelle {\`a} la mesure de Hausdorff $1$\nobreakdash-dimension\-nelle associ{\'e}e a~$d_{\HK}$.
  Alors~$J_R$ est de longueur positive et nous avons ${J_R = I = \bigcup_{j = 1}^{k} I_j}$ par la Proposition~\ref{prop:affine Bernoulli}~(2).
  Comme les deux mesures invariantes~$\rho_I$ et $\rho_R$ sont ergodiques, elles sont \'egales et donc, pour chaque~$j$ dans~$\{1, \ldots, k \}$, on a
  \begin{equation}
    \label{eq:12}
    \frac{d_j}d
    =
    \rho_R(I_j)
    =
    \rho_I(I_j)
    =
    \frac1{d_j}
    \text{ et }
    d_j^2
    =
    d~.
  \end{equation}
  On conclut que $R$ est \`a allure Latt\`es.
\end{proof}

%
%

\section{Exposants de \mbox{Lyapunov} et le sous-cobord g{\'e}om{\'e}trique}\label{Sec:cocycle}
Le but de cette section est de d\'emontrer le Th\'eor\`eme~\ref{thm:main3}, {\'e}nonc{\'e} dans l'introduction.
Apr{\`e}s avoir calcul{\'e} l'exposant de \mbox{Lyapunov} {\`a} l'aide de la th{\'e}orie du potentiel (\S\S~\ref{sec:derivee-Lyap}, \ref{ss:pr}), on exprime la d{\'e}riv{\'e}e sph{\'e}rique d'une fraction rationnelle sous la forme d'un sous-cobord multiplicatif (\S~\ref{sec:derivee}) et on {\'e}tablit un lien entre l'exposant de \mbox{Lyapunov} et la ramification sauvage (\S~\ref{ss:exposants}), ce qui permet de d{\'e}duire le Th\'eor\`eme~\ref{thm:main3}~(1) (\S~\ref{sec:thm31}).
Apr{\`e}s une estimation pr{\'e}liminaire (\S~\ref{ss:integrable}), {\'e}galement utilis{\'e}e dans la preuve du Th{\'e}or{\`e}me~\ref{thm:main2}, nous donnons la d{\'e}monstration du Th{\'e}or{\`e}me~\ref{thm:main3}~(2) (\S~\ref{sec:thm32}).
Enfin, on ach{\`e}ve cette partie avec la d{\'e}monstration du Th{\'e}or{\`e}me~\ref{thm:main3}~(3) (\S~\ref{sec:briendduval}).

\smallskip

Rappelons que, pour une fraction rationnelle~$R$ {\`a} coefficients dans~$K$, non constante, un point critique de~$R$ est un point de~$\PKber$ o{\`u} ${\deg_R \ge 2}$.
Nous d{\'e}signons par~$\crit_R$ l'ensemble des points critiques de~$R$ et par~$\crit_R(K)$ ceux qui appartiennent {\`a}~$\PK$.
De plus, un point critique de~$R$ est inséparable s'il appartient {\`a}~$\HK$ et~$R$ est ins{\'e}parable en ce point, et on note~$\wcrit_R$ l'ensemble des points critiques inséparables.
Voir \S~\ref{ss:action} pour de pr{\'e}sicions.

Dans cette section, pour tout~$x$ dans~$\PK$ et~$r$ dans~$\mathopen] 0, 1 \mathclose]$, on note~$B(x, r)$ la boule ouverte de~$\PK$ de diam{\`e}tre projectif~$r$.

%
%

\subsection{D\'eriv\'ee sph\'erique et exposant de \mbox{Lyapunov}}
\label{sec:derivee-Lyap}

Apr{\`e}s des rappels sur la d{\'e}riv{\'e}e sph{\'e}rique, on calcule dans cette section l'exposant de \mbox{Lyapunov} de la mesure d'{\'e}quilibre et on d{\'e}termine quand il est fini (Corollaire~\ref{cor:justif-lyap}).

Soit~$R$ une fraction rationnelle non constante {\`a} coefficients dans~$K$.
On suppose~$R$ s\'eparable, de telle sorte que sa d\'eriv\'ee $R'$ est non identiquement nulle.
On introduit alors la d\'eriv\'ee sph\'erique en un point~$x$ dans ${\AKber\setminus R^{-1}(\infty)}$ par la formule:
\begin{equation} \label{eq:def-der-spherique}
  \| R' \|(x)
  \=
  |R'(x)|\, \frac{\max \{ 1 , |x|\}^2}{\max \{ 1 , |R(x)|\}^2}~.
\end{equation}
On montre que la fonction $\|R'\|$ se prolonge contin\^ument \`a $\PKber$.
On a de plus la formule de composition
\[
  \| (R_1\circ R_2) ' \|(x)
  =
  \| (R'_1) \|(R_2(x)) \times \| R'_2 \|(x)~,
\]
ce qui montre que la d\'eriv\'ee sph\'erique est invariante par post- et pr\'e-composition par une transformation de M\"obius
fixant~$\xcan$, c'est-\`a-dire par tout \'el\'ement de~$\PSL(2,\cO_K)$.

\begin{Prop}
  \label{prop:delta}
  Pour toute fraction rationnelle~$R$ {\`a} coefficients dans~$K$, non constante et s\'eparable, la fonction $\log \| R' \|$ est un
  potentiel.
  De plus, on a
  \begin{equation*}
    \Delta \log \|R' \|
    =
    \sum_{x\in \crit_R(K)} \ord_{R'}(x) \delta_x + 2\delta_{\xcan} - 2\, \sum_{x \in R^{-1}(\xcan)} \deg_R (x) \delta_x~,
  \end{equation*}
  et la variation totale de~$\Delta \log \|R' \|$ est au plus $4\deg(R)$.
  En particulier, $\|R'\|$ est localement constante hors de l'enveloppe convexe de l'union de~$\crit_R(K)$, $\{ \xcan \}$, et~$R^{-1}(\xcan)$.
\end{Prop}

\begin{proof}
  Il suffit de calculer le laplacien de $\log \|R'\|$ restreint à $\AKber\setminus R^{-1}(\infty)$ par invariance par l'inversion.
  Ceci r\'esulte du fait que dans $\AKber\setminus R^{-1}(\infty)$ nous avons~:
  \begin{equation}
    \label{eq:13}
    \Delta \log |R'(x)|
    =
    \sum_{x \in \crit_R(K)} \ord_{R'}(x) \delta_x
    \text{ et }
    \Delta \log \max \{ 1, |R(x)|\}
    =
    R^* \delta_{\xcan}~,
  \end{equation}
  et qu'on a
  $\Delta \log \max \{ 1, |x|\} = \delta_{\xcan}$ dans $\AKber$.
\end{proof}

Rappelons que l'exposant de \mbox{Lyapunov} de~$\rho_R$  est défini par~$\chi(R)= \int \log \| R' \| \dd \rho_R$.
Cette intégrale est bien d{\'e}finie et appartient {\`a}~$\R \cup \{ -\infty \}$, car~$\| R' \|$ est born{\'e} supérieurement.

\begin{Cor}\label{cor:justif-lyap}
  Soit~$R$ une fraction rationnelle {\`a} coefficients dans~$K$, non constante.
  Alors, $\chi(R)$ est fini si et seulement si~$R$ est s\'eparable, et dans ce cas on a
  \begin{equation}
    \label{eq:15}
    \chi(R)
    =
    \int \log |R'| \dd \rho_R~.
  \end{equation}
\end{Cor}

\begin{proof}
  Lorsque~$R$ est ins{\'e}parable, la fonction~$\| R' \|$ est identiquement nulle et ${\chi(R) = -\infty}$.
  Supposons que~$R$ est s{\'e}parable et soit~$g$ un potentiel tel que ${\rho_R = \delta_{\xcan} +\Delta g}$.
  Ce potentiel est continu, voir~\cite[Proposition~3.3]{theorie-ergo}.
  L'exposant~$\chi(R)$ est alors la somme des quantit\'es finies~$\int g \dd \Delta \log \| R '\|$ et~$\log \| R'\| (\xcan)$.
  Par ailleurs, comme ${\log\|R'\| - \log |R'|}$ est un cobord dont la fonction de cobord est un potentiel, on a l'{\'e}galit{\'e} d{\'e}sir{\'e}e par la Proposition~\ref{prop:symmetry}.
\end{proof}

\subsection{Exposant de \mbox{Lyapunov} des polynômes}
\label{ss:pr}
Comme dans le cas complexe \cite[\S5]{Prz85a}, l'exposant de \mbox{Lyapunov} d'un polyn\^ome~$P$ est reli\'e {\`a} son degr{\'e} et au taux d'\'echappement des points critiques.
On rappelle que pour un polyn{\^o}me~$P$ {\`a} coefficients dans~$K$ et de degr{\'e}~$d$ au moins deux, la suite de potentiels $(d^{-n}\log\max\{1, |P^n|\})_{n = 1}^{+\infty}$ converge uniform\'ement sur~$\AKber$ vers un potentiel~$g_P$ v\'erifiant
\begin{equation}
  \label{eq:14}
  g_P
  \ge
  0,
  g_P \circ P
  =
  d g_P
  \text{ et }
  \Delta g_P
  =
  \rho_P - \delta_\infty~.
\end{equation}
On appelle~$g_P$ la \emph{fonction de Green de~$P$}.
La formule suivante est d\'emontr\'ee dans le cas de caract\'eristique nulle dans~\cite[\S5]{okuyama}.

\begin{Prop}[Formule de Przytycki]
  \label{p:pr}
  Soit~$P$ un polyn{\^o}me {\`a} coefficients dans~$K$, de degr{\'e}~$d$ au moins deux et s{\'e}parable.
  Notons~$g_P$ la fonction de Green de~$P$, $d'$ le degr{\'e} de~$P'$, et~$a$ (resp. $\gamma$) le coefficient dominant de~$P$ (resp. de~$P'$).
  Alors on a
  \begin{equation}
    \label{e:pr}
    \chi(P)
    =
    \log |\gamma| - \frac{d' \log|a|}{d-1} + \sum_{c \in \crit_P(K)\setminus\{\infty\}} \ord_{P'}(c) g_P(c)~.
  \end{equation}
  En particulier, on a
  \begin{equation}
    \label{eq:73}
    \chi(P)
    \ge
    \log |\gamma| - \frac{d' \log|a|}{d-1}~,
  \end{equation}
  avec {\'e}galit{\'e} si et seulement si aucun point critique de~$P$ dans~$K$ ne s’{\'e}chappe vers l’infini par it{\'e}ration.
\end{Prop}

Lorsque~$d$ n'est pas divisible par la caract\'eristique de~$K$, on a ${d' = d - 1}$ et ${\gamma = da}$ et donc ${\chi(P) \ge \log |d|}$, o{\`u}~$|d|$ est la norme de l'entier~$d$ calcul{\'e}e dans~$K$.
Si de plus~$d$ n'est pas divisible par la caract\'eristique r{\'e}siduelle de~$K$, on a ${|d| = 1}$ et on en d{\'e}duit que ${\chi(P) \ge 0}$.

\begin{proof}[D{\'e}monstration de la Proposition~\ref{p:pr}]
  Pour chaque~$M$ dans~$\mathopen] 0, +\infty \mathclose[$, notons~$x(M)$ le point de~$\AKber$ associ{\'e} {\`a} la boule~$\{ z \in K, |z| \le M\}$.

  Pour tout~$z$ dans~$K$ tel que~$|z|$ est suffisament grand, on a ${g_P(z) = \log |z|+\beta}$, avec ${\beta \= \frac{\log |a|}{d-1}}$.
  Alors, pour tout~$M$ suffisamment grand on a
  \begin{equation}
    \label{eq:16}
    \Delta \min \{ g_P, \log M +\beta \}
    =
    \rho_P - \delta_{x(M)}
    \text{ et }
    |P'|(x(M))
    =
    |\gamma| M^{d'}~.
  \end{equation}
  On en d{\'e}duit
  \begin{multline}
    \label{eq:17}
    \chi(P)
    =
    \int \log |P'| \dd \Delta \min \{ g_P, \log M+\beta \} + \log |P'|(x(M))
    \\ =
    \int \min \{ g_P, \log M+\beta \} \dd \Delta \log |P'| + \log |\gamma| + d' \log M~,
  \end{multline}
  voir la Proposition~\ref{prop:symmetry} et le Corollaire~\ref{cor:justif-lyap}.
  On obtient enfin
  \begin{multline}
    \label{eq:18}
    \int \min \{ g_P, \log M +\beta \} \dd \Delta \log |P'|
    =
    \sum_{c \in \crit_P(K)\setminus\{\infty\}} \ord_{P'}(c) g_P(c) - d'( \log M +\beta)~.
  \end{multline}
  Combin{\'e} {\`a}~\eqref{eq:17}, cela entra{\^{\i}}ne l'{\'e}galit{\'e} d{\'e}sir{\'e}e.
\end{proof}

%
%

\subsection{Sous-cobord g{\'e}om{\'e}trique et ramification sauvage}
\label{sec:derivee}
Dans cette section, on exprime la d{\'e}riv{\'e}e sph{\'e}rique comme un sous-cobord multiplicatif (Proposition~\ref{prop:cocycle}) et on donne une caract{\'e}risation g{\'e}om{\'e}trique des points critiques ins{\'e}parables (Corollaire~\ref{c:tres-wild}).
Rappelons que pour un point~$x$ de~$\HK$, le degr{\'e} d'ins{\'e}parabilit{\'e}~$\wdeg{R}(x)$ de~$R$ en~$x$ est d{\'e}fini dans \S~\ref{ss:action}.

\begin{Prop}
  \label{prop:cocycle}
  Pour toute fraction rationnelle~$R$ {\`a} coefficients dans~$K$, non constante et s\'eparable, et tout~$x$ dans~$\HK$ on a
  \begin{equation}
    \label{e:diam1}
    |\wdeg{R}(x)|
    \le
    \frac{\|R'\|(x) \times \diam(x)}{\diam (R(x))}
    \le
    \exp(-d_{\HK}(x, \HK \setminus \crit_R))~.
  \end{equation}
\end{Prop}

Notons que $|\wdeg{R}(x)|$ est la norme de l'entier $\wdeg{R}(x)$ calculée dans $K$.
Si $K$ est de caractéristique positive, il peut arriver que ce nombre soit nul, auquel cas la borne {\`a} gauche dans~\eqref{e:diam1} est triviale.

Le corollaire suivant est une cons{\'e}quence de la Proposition~\ref{prop:cocycle} et de \cite[Proposition~10.2~(1)]{rivera-periode}.

\begin{Cor}
  \label{c:tres-wild}
  Soit~$R$ fraction rationnelle {\`a} coefficients dans~$K$, non constante et s\'eparable.
  Alors, pour tout point~$x$ de~$\PKber$ on a
  \begin{equation}
    \label{eq:19}
    \|R'\|(x) \times \diam(x)
    \le
    \diam (R(x))~,
  \end{equation}
  et les propri{\'e}t{\'e}s suivantes sont {\'e}quivalentes~:
  \begin{enumerate}
  \item
    ${\|R'\|(x) \times \diam(x) < \diam (R(x))}$;
  \item
    $x$ appartient {\`a} l'intérieur de~$\crit_R$ pour la topologie fine de~$\PKber$;
  \item
    $x$ est un point critique inséparable de~$R$.
  \end{enumerate}
\end{Cor}

L'{\'e}quivalence des propri{\'e}t{\'e}s (2) et~(3) a {\'e}t{\'e} montrée par Faber \cite[Theorem~B]{faber1}.
On donne ici une d{\'e}monstration diff{\'e}rente.
Voir~\cite{temkin} pour d'autres caract{\'e}risations des points critiques inséparables.

La d{\'e}monstration du corollaire suit celle de la proposition.

\begin{proof}[D{\'e}monstration de la Proposition~\ref{prop:cocycle}]
  Comme la d\'eriv\'ee sph\'erique est invariante par l'action de~$\PSL(2,\cO_K)$, il suffit de montrer~\eqref{e:diam1} dans le cas o\`u ${|x| \le 1}$ et ${|R(x)| \le 1}$.
  On commence par d{\'e}montrer l'in{\'e}galit{\'e} {\`a} droite dans~\eqref{e:diam1}.
  Par continuit\'e, il suffit de traiter le cas o\`u~$x$ est de type~II et distinct de~$\xcan$.
  Soit~$B$ un composante connexe de ${\PKber \setminus \{ x \}}$ disjointe de ${R^{-1}(\xcan) \cup \{ \xcan \}}$ et de~$\crit_R(K)$.
  Alors~$B$ et~$R(B)$ sont des boules ouvertes de~$\AKber$ contenues dans la boule unit{\'e} ferm{\'e}e, et~$\| R' \|$ est constante sur~$B$ par la Proposition~\ref{prop:delta}.
  Soit~$x_0$ un point de ${B \cap \PK}$, et~$r$ dans~$\mathopen] 0, 1 \mathclose[$ suffisamment petit pour que la boule ferm{\'e}e~$B_0$ de~$\AK$ de centre~$x_0$ et diam{\`e}tre~$r$ soit contenue dans~$B$ et que~$R$ soit univalente sur~$B_0$.
  Alors le point~$\hx_0$ associ{\'e} {\`a}~$B_0$ v{\'e}rifie
  \begin{equation}
    \label{eq:20}
    \diam(R(\hx_0))
    =
    \diam(R(B_0))
    =
    |R'(x_0)| \times r
    =
    \| R' \|(\hx_0) \times \diam(\hx_0)
  \end{equation}
  par le lemme de Schwarz, voir \cite[\S~1.3.1]{R1}.
  Pour chaque~$t$ dans ${\mathopen[ 0, d_{\HK}(x, \hx_0) \mathclose]}$, notons~$y(t)$ le point de~$[x, \hx_0]$ {\`a} distance~$t$ de~$x$.
  Comme~$R$ est injective sur~$[x, \hx_0]$ et ${d_{\HK}(x, \hx_0) \ge d_{\HK}(x, \HK \setminus \crit_R)}$, on a
  \begin{align*}
    \label{eq:21}
    \log \frac{\diam(R(x))}{\diam(R(\hx_0))}
    &=
      d_{\HK}(R(x), R(\hx_0))
    \\ & =
         \int_0^{d_{\HK}(x, \hx_0)} \deg_R(y(t)) \dd t
    \\ &=
         \log \frac{\diam(x)}{\diam(\hx_0)} + \int_0^{d_{\HK}(x, \hx_0)} \deg_R(y(t)) - 1 \dd t
    \\ &\ge
         \log \frac{\diam(x)}{\diam(\hx_0)} + d_{\HK}(x, \HK \setminus \crit_R)~,
  \end{align*}
  voir \cite[Corollaire~4.8]{rivera-espace}.
  Combin{\'e} avec~\eqref{eq:20}, ceci entra{\^{\i}}ne l'in{\'e}galit{\'e} de droite dans~\eqref{e:diam1}.

  Pour d{\'e}montrer l'in{\'e}galit{\'e} de gauche dans~\eqref{e:diam1}, soit~$x$ un point arbitraire de~$\HK$, posons
  \begin{equation}
    \label{eq:22}
    d
    \=
    \deg_R(x),
    r
    \=
    \diam(x)
    \text{ et }
    r'
    \=
    \diam(R(x))~,
  \end{equation}
  et pour chaque~$\rho$ dans~$\mathopen] 0, +\infty \mathclose[$ notons~$x(\rho)$ le point de~$\AKber$ associ\'e ${\{z \in K, |z| \le \rho\}}$.
  On commence par traiter le cas particulier o{\`u}~$x$ se situe entre deux points~$x_0$ et~$x_1$ tels que ${R(x_0) \neq R(x_1)}$, et o{\`u}~$R$ envoie la couronne comprise entre~$x_0$ et~$x_1$ sur celle comprise entre~$R(x_0)$ et~$R(x_1)$.
  En prenant les points~$x_0$ et~$x_1$ plus proches de~$x$ et en changeant de coordonn{\'e}es si n{\'e}cessaire, on suppose que~$x_0$ et~$x_1$ sont de type~II, et qu'il existe~$r_0, r_1, r_0'$ et~$r_1'$ dans~$|K^*|$ tels que ${r_0 < r_1}$, ${x_0 = x(r_0)}$, ${x_1 = x(r_1)}$, ${R(x_0) = x(r_0')}$ et~${R(x_1) = x(r_1')}$.
  On a donc ${x = x(r)}$ et ${R(x) = x(r')}$.
  Par ailleurs, si l'on pose ${\od \= d}$ lorsque ${r_0' < r_1'}$ et ${\od \= -d}$ lorsque ${r_0' > r_1'}$, alors on peut d\'evelopper~$R$ en s\'erie sur la couronne ${\{ y \in K, r_0 < |y| < r_1 \}}$, $R(z) = \sum_{k = -\infty}^{+\infty} a_k z^k$, et on~a
  \begin{equation}
    \label{eq:23}
    r'
    =
    |a_{\od}| r^{\od}
    >
    \sup \{ |a_k| r^k, k \in \Z, k \neq \od \}~,
  \end{equation}
  voir la d{\'e}monstration de \cite[Lemme~5.3]{rivera-periode}.
  On a donc
  \begin{equation}
    \label{eq:24}
    \| R' \|(x)
    \ge
    |\od| \times |a_{\od}| r^{\od - 1}
    =
    |\od| \times r'/r
    =
    |\deg_R(x)| \times \diam(R(x)) / \diam(x)~.
  \end{equation}
  Ceci d{\'e}montre l'in{\'e}galit{\'e} de gauche dans~\eqref{e:diam1} dans ce cas particulier.
  Pour d{\'e}montrer cette in{\'e}galit{\'e} dans le cas g{\'e}n{\'e}ral, supposons~$x$ soit un point arbitraire de~$\PKber$.
  En changeant de base si n{\'e}cessaire, on suppose que~$x$ est de type~II.
  Prenons~$x'$ distinct de~$x$, tel que ${|x| \le 1}$, tel que~$R$ soit injective et que l'on ait ${\deg_R = \wdeg{R}}$ sur ${] x, x' ]}$, et tel que~$R$ envoie la couronne comprise entre~$x$ et~$x'$ sur celle comprise entre~$R(x)$ et~$R(x')$, voir \cite[Propositions~4.1, 4.4 et~4.6]{rivera-espace}.
  En appliquant le cas particulier {\`a} chaque~$x''$ entre~$x$ et~$x'$, on obtient
  \begin{equation}
    \label{eq:25}
    |\wdeg{R}(x)|
    \le
    \|R'\| (x'') \times \diam(x'') / \diam(R(x''))~.
  \end{equation}
  On obtient l'in{\'e}galit{\'e} {\`a} gauche dans~\eqref{e:diam1} en faisant tendre~$x''$ vers~$x$.
\end{proof}

\begin{proof}[D{\'e}monstration du Corollaire~\ref{c:tres-wild}]
  Lorsque~$x$ est dans~$\HK$, l'in{\'e}galit{\'e} {\`a} droite de~\eqref{e:diam1} entra{\^{\i}}ne~\eqref{eq:19}, avec une in{\'e}galit{\'e} stricte lorsque~$x$ appartient {\`a} l'intérieur de ${\HK \cap \crit_R}$ dans~$(\HK, d_{\HK})$.
  Il s'ensuit que l'in{\'e}galit{\'e}~\eqref{eq:19} est stricte lorsque~$x$ appartient {\`a} l'inté\-rieur de~$\crit_R$ pour la topologie fine de~$\PKber$.
  Par ailleurs, lorsque~$x$ est dans~$\PK$, l'{\'e}galit{\'e} est satisfaite dans~\eqref{eq:19}, et~$x$ n'appartient pas {\`a} l'intérieur de~$\crit_R$ pour la topologie fine de~$\PKber$, car~$R$ est suppose s{\'e}parable.
  Ceci d{\'e}montre~\eqref{eq:19} et l'implication (2)$\Rightarrow$(1) dans tous les cas.

  Pour montrer l'implication (1)$\Rightarrow$(3), notons que la propri{\'e}t{\'e}~(1) entra{\^{\i}}ne que~$x$ est dans~$\HK$, et donc que~$x$ est un point critique inséparable de~$R$ par l'in{\'e}galit{\'e} {\`a} droite de~\eqref{e:diam1}.

  Lorsque~$x$ est de type~II, l'implication (3)$\Rightarrow$(2) est donn{\'e}e par \cite[Proposition~10.2~(1)]{rivera-periode}.
  Lorsque~$x$ est de type~III ou~IV, on se ram{\`e}ne au cas pr{\'e}c{\'e}dent par un changement de base, voir \cite[\S~4]{faber1} ou~\cite{poineau} pour un cadre plus g{\'e}n{\'e}ral.
\end{proof}

%
%

\subsection{Les exposants de \mbox{Lyapunov} ne sont pas trop n{\'e}gatifs}
\label{ss:exposants}
Dans cette section, on {\'e}tablit un lien entre l'exposant de \mbox{Lyapunov} et la ramification sauvage de~$R$ pour chaque mesure de probabilit{\'e} ergodique support{\'e}e sur~$\HK$ (Th{\'e}or{\`e}me~\ref{t:exponsants}~(1)).
On en d{\'e}duit une minoration pour l'exposant de \mbox{Lyapunov} de toute mesure de probabilit{\'e} ergodique qui n'est pas support{\'e}e sur un cycle attractif (Corollaire~\ref{c:exponsants}).

Soit~$R$ une fraction rationelle {\`a} coefficients dans~$K$ et degr{\'e} au moins deux, et~$\nu$ une mesure de probabilit{\'e} sur~$\PKber$ invariante par~$R$.
La fonction~$\| R' \|$ {\'e}tant born{\'e}e supérieurement, l'int{\'e}grale~$\int \log \| R' \| \dd \nu$ est d{\'e}finie.
On la d{\'e}signe par~$\chi_{\nu}(R)$.
Lorsque ${\nu = \rho_R}$ est la mesure d'équilibre, on a ${\chi_{\rho_R}(R) = \chi(R)}$.
Notons que~$\chi_{\nu}(R)$ est un nombre r{\'e}el ou~$-\infty$.
Par exemple, lorsque~$R$ est ins{\'e}parable on a ${\log \| R' \| \equiv -\infty}$ et ${\chi_\nu(R) = -\infty}$.

\begin{Thm}
  \label{t:exponsants}
  Soit~$R$ une fraction rationelle {\`a} coefficients dans~$K$, s{\'e}parable et degr{\'e} aux moins deux, et soit ${\wf_R \: \HK\to \R}$ la fonction d{\'e}finie par
  \begin{equation}
    \label{eq:26}
    \wf_R
    \=
    -\log (\| R' \| \times \diam / \diam \circ R)~.
  \end{equation}
  Alors~$\wf_R$ est continue par rapport {\`a}~$d_{\HK}$, et on a
  \begin{equation}
    \label{eq:27}
    0
    \le
    \wf_R
    \le
    -\log |\wdeg{R}|
    \text{ et }
    \{ \wf_R > 0 \}
    =
    \wcrit_R~.
  \end{equation}
  De plus, pour toute mesure de probabilit{\'e}~$\nu$ sur~$\PKber$ invariante et ergodique par~$R$, les propri{\'e}t{\'e}s suivantes sont vérifiées.
  \begin{enumerate}
  \item
    Si $\nu(\HK)=1$, alors
    \begin{equation}
      \label{eq:J2}
      \chi_{\nu}(R)
      =
      -\int \wf_R \dd \nu~.
    \end{equation}
    En particulier, on a ${\chi_{\nu}(R) \le 0}$ avec {\'e}galit{\'e} si et seulement si~$\nu(\wcrit_R)=0$.
  \item
    Si $\nu(\PK)=1$, alors on a ${\chi_{\nu}(R) \ge 0}$ sauf si~$\nu$ est support{\'e}e sur un cycle attractif.
  \end{enumerate}
  En particulier, on a ${\chi_{\nu}(R) \ge \log |\max_{x \in \HK} \wdeg{R}(x)|}$, ou bien~$\nu$ est support{\'e}e sur un cycle attractif.
\end{Thm}

Dans le cas des fractions rationelles complexes, le point~(2) a {\'e}t{\'e} montr{\'e}e par \mbox{Przytycki} \cite[Theorem~A]{Prz93}.
Ici on suit la d{\'e}monstration dans le cas des applications de l'intervalle dans~\cite[Proposition~A.1]{rivera-expand-smooth}.

Le corollaire suivant est une cons{\'e}quence imm{\'e}diate du Th{\'e}or{\`e}me~\ref{t:exponsants}.
Lorsque~$R$ est un polynôme et~$\nu$ satisfait une propri{\'e}t{\'e} d'int{\'e}grabilit{\'e} suppl{\'e}mentaire ce corollaire est~\cite[Theorem~1.1]{nie}.

\begin{Cor}
  \label{c:exponsants}
  Soit~$R$ une fraction rationnelle {\`a} coefficients dans~$K$, s{\'e}parable et de degr{\'e} au moins deux.
  Alors, toute mesure de probabilit{\'e}~$\nu$ sur~$\PKber$, invariante et ergodique par~$R$, qui n'est pas support{\'e}e sur un cycle attractif satisfait
  \begin{equation}
    \label{eq:J21}
    \chi_{\nu}(R)
    \ge
    \log \min \{ |\ell| , \ell \in \{ 1, \ldots, \deg(R) \} \}~.
  \end{equation}
\end{Cor}

L'in{\'e}galit{\'e}~\eqref{eq:J21} est {\'e}galement satisfaite si~$\nu$ est support{\'e}e sur une orbite p{\'e}riodique attractive n'attirant aucun point critique de~$R$ dans~$\PK$ \cite[Corollary~1.2]{0Riv2601}.

On utilise le lemme suivant dans la d{\'e}monstration du Th{\'e}or{\`e}me~\ref{t:exponsants}.

\begin{Lem}
  \label{lem:univalent}
  Pour toute fraction rationnelle~$R$ {\`a} coefficients dans~$K$, non constante et s\'eparable, il existe une constante $C>0$ et une fonction ${\theta \: \PK \to \mathopen[ 0, 1 \mathclose]}$, telles que ${\theta \ge C \, \|R'\|}$, et que, pour tout~$x$ en dehors de~$\crit_R(K)$, l'application~$R$ est univalente sur~$B(x , \theta(x))$ et envoie cette boule sur~$B(R(x), \theta(x) \|R'\|(x))$.
\end{Lem}

\begin{proof}
  Soit~$C_0$ dans~$\mathopen] 0, 1 \mathclose[$, strictement plus petit que le diam\`etre de chaque pr\'eimage de~$\xcan$.
  On montre que le lemme est satisfait avec~$\theta$ d{\'e}finie par
  \begin{equation}
    \label{eq:28}
    \theta(x)
    \=
    C_0 \min \left\{ C_0 \| R' \|(x), 1 \right\}~.
  \end{equation}
  Notons d'abord que, en posant ${C \= C_0 / \max \{ \sup_{\PK} \| R' \|, 1 \}}$, on a ${\theta \ge C \| R' \|}$.

  Fixons~$x$ dans ${\PK \setminus \crit_R(K)}$.
  Par notre choix de~$C_0$, l'image de~$B(x, C_0)$ par~$R$ est contenue dans~$B(R(x), 1)$.
  En changeant de coordonn{\'e}es au d{\'e}part et {\`a} l'arriv{\'e}e si nécessaire, on peut supposer que ${x = 0}$ et ${R(x) = 0}$.
  On a alors
  \begin{equation}
    \label{eq:29}
    \| R' \|(0)
    =
    |R'(0)|
    \text{ et }
    R(B(0, C_0))
    \subset
    B(0,1)~.
  \end{equation}
  D\'eveloppons~$R$ en s\'erie sur la boule $B(0, C_0)$, ${R(z) = \sum_{k = 1}^{+\infty} a_k z^k}$.
  Alors ${|a_1| = |R'(0)|}$, et pour tout~$k$ dans~$\N^*$ on a ${|a_k| \le C_0^{-k}}$.
  Par ailleurs, pour tout~$r$ dans~$\mathopen[ 0, \theta(0) \mathclose[$ et tout entier~$k$ v{\'e}rifiant ${k \ge 2}$, on a
  \begin{equation}
    \label{eq:30}
    |a_k| r^k
    <
    C_0^{-k} r \theta(0)^{k - 1}
    \le
    |a_1| r~,
  \end{equation}
  et par con{\'e}quent $\sup_{B(0, r)} |R| = |a_1| r$.
  Ceci montre que~$R$ est univalente sur la boule $B(0, \theta(0))$ et l'envoie sur ${B(0, \theta(0) \, |R'(0)|)}$.
\end{proof}

\begin{proof}[D{\'e}monstration du Th{\'e}or{\`e}me~\ref{t:exponsants}]
  Comme~$\diam$ est continue pour la topologie fine et~$R$ et~$\| R' \|$ le sont pour la topologie faible, $\wf_R$ est continue pour~$d_{\HK}$.
  Les assertions dans~\eqref{eq:27} sont une cons{\'e}quence directe de la Proposition~\ref{prop:cocycle} et du Corollaire~\ref{c:tres-wild}.

  Pour démontrer le point (1), on fixe~$\varepsilon\in \mathopen]0, 1\mathclose[$ tel que ${\nu(\{\diam \ge \varepsilon \}) > 0}$.
  Comme $\log \| R' \|$ et~$-\wf_R$ sont born{\'e}es sup{\'e}rieurement, le th{\'e}or{\`e}me de {Birkhoff} combin{\'e} au th{\'e}or{\`e}me de convergence monotone donne l'{e}xistence de~$x\in\HK$ tel que:
  \begin{align*}
    \chi_{\nu}(R) &=
                    \lim_{n \to +\infty} \frac{1}{n} \sum_{i = 0}^{n - 1} \log \| (R^i)' \|(x)~, \\
    \int \wf_R \dd \nu &=
                         \lim_{n \to +\infty} \frac{1}{n} \sum_{i = 0}^{n - 1} \wf_R(R^i(x))~,
  \end{align*}
  et il existe une suite d'entiers strictement positifs~$(n_j)_{j = 1}^{+\infty}$ telle que~${n_j \to +\infty}$ lorsque ${j \to +\infty}$, et telle que pour tout~$j$, on a ${\diam(R^{n_j}(x)) \ge \varepsilon}$.
  On obtient alors
  \begin{multline*}
    \label{eq:J4}
    \chi_{\nu}(R)
    =
    \lim_{j \to +\infty} \frac{1}{n_j} \log \| (R^{n_j})' \|(x)
    \\ =
    \lim_{j \to +\infty} \frac{1}{n_j} \left( \log \diam (R^{n_j}(x)) - \log \diam(x) - \sum_{i = 0}^{n_j - 1} \wf_R(R^i(x)) \right)
    \\ =
    - \lim_{j \to +\infty} \frac{1}{n_j} \sum_{i = 0}^{n_j - 1} \wf_R(R^i(x))
    =
    -\int \wf_R \dd \nu~.
  \end{multline*}
  La derni{\`e}re assertion du point~(1) est donc une cons{\'e}quence de~\eqref{eq:27}.

  \smallskip

  On démontre maintenant le point (2).
  Supposons donc $\nu(\PK)=1$ et ${\chi_{\nu}(R) < 0}$.
  Par le th{\'e}or{\`e}me de convergence monotone, il existe $L\in \mathopen]0, +\infty\mathclose[$ tel que, si l'on {\'e}crit
  \begin{equation}
    \label{eq:J5}
    \varphi
    \=
    \max \{ \log \| R' \|, -L \}
    \text{ et }
    I
    \=
    \int \varphi \dd \nu~,
  \end{equation}
  alors ${I < 0}$.
  Posons ${\varphi_0 \equiv 0}$ et ${\varphi_n \= \varphi + \cdots + \varphi \circ R^{n - 1}}$ pour tout $n\in\N^*$.
  Par le th{\'e}or{\`e}me de {Birkhoff}, on peut trouver un point $x\in\PK$ d'orbite dense dans $\supp(\nu)$ tel que ${\lim_{n \to +\infty} \frac{1}{n} \varphi_n(x) = I}$.
  Fixons~$\chi\in \mathopen]0, -I\mathclose[$.

  \smallskip

  On va montrer qu'il existe~$\tau$ et~$\varepsilon$ dans~$\mathopen]0, 1\mathclose[$ tel que pour tout entier positif ou nul~$n$, on a
  \begin{equation}
    \label{eq:J6}
    R^n(B(x, \tau))
    \subseteq
    B(R^n(x), \varepsilon \exp(-\chi n))~.
  \end{equation}
  Soit~$\delta\in\mathopen]0, 1\mathclose[$ tel que pour tout~$x'\in B(\crit_R, \delta)$ et~$r\in\mathopen]0, \delta\mathclose]$, on a
  \begin{equation}
    \label{eq:J7}
    R(B(x', r))
    \subseteq
    B(R(x'), \exp(-L)r)~.
  \end{equation}
  Par ailleurs, soit~$\varepsilon\in\mathopen]0, \delta\mathclose]$ tel que pour tout~$x'\in{\PK \setminus B(\crit_R, \delta)}$ et~$r\in \mathopen]0, \varepsilon\mathclose]$ on a
  \begin{equation}
    \label{eq:J8}
    R(B(x', r))
    =
    B(R(x'), \| R' \|(x) r),
  \end{equation}
  voir Lemme~\ref{lem:univalent}.
  Par notre choix de~$x$ et~$\chi$, il existe~$\tau\in\mathopen]0, 1\mathclose[$ tel que pour tout $n\in\N^*$, on a
  \begin{equation}
    \label{eq:31}
    \tau \exp(\varphi_n(x) + n\chi)
    \le
    \varepsilon~.
  \end{equation}
  On {\'e}crit ${B_n \= B(R^n(x), \tau \exp(\varphi_n(x)))}$ et note que
  \begin{equation}
    \label{eq:J10}
    \diam(B_n)
    =
    \tau \exp(\varphi_n(x))
    \le
    \varepsilon \exp(-n\chi)
    \le
    \varepsilon
    \le
    \delta~.
  \end{equation}
  Si~$R^n(x)$ n'est pas dans~$B(\crit_R(K), \delta)$, alors par~\eqref{eq:J8} on a
  \begin{equation}
    \label{eq:J11}
    R(B_n)
    =
    B(R^{n + 1}(x), \tau \exp(\varphi_n(x)) \| R' \|(R^n(x)))
    \subseteq
    B_{n + 1}~.
  \end{equation}
  Autrement, par~\eqref{eq:J7} on a
  \begin{equation}
    \label{eq:J12}
    R(B_n)
    \subseteq
    B(R^{n + 1}(x), \tau \exp(\varphi_n(x)) \exp(-L))
    \subseteq
    B_{n + 1}~.
  \end{equation}
  On a donc dans tous les cas ${R(B_n) \subseteq B_{n + 1}}$ et on obtient~\eqref{eq:J6} par récurrence.

  \smallskip

  Pour conclure la démonstration du point (2), on remarque que~\eqref{eq:J6} fournit l'existence d'un entier strictement positif~$n$ tel que ${R^n(B(x, \tau)) \subseteq B(x, \tau/2)}$.
  Il s'ensuit que $R^n$ possède un point fixe attractif~$x_0\in B(x, \tau)$ et que tout point de~$B(x, \tau)$ converge vers~$x_0$ sous it{\'e}ration par~$R^n$.
  Comme l'orbite de~$x$ par~$R$ est dense dans~$\supp(\nu)$, on conclut que la mesure~$\nu$ est supportée sur l'orbite de~$x_0$.
\end{proof}

%
%

\subsection{D\'emonstration du Th\'eor\`eme~\ref{thm:main3}~(1)}\label{sec:thm31}
Soit~$R$ une fraction rationnelle {\`a} coefficients dans~$K$ et de degr\'e aux moins deux.
On cherche \`a montrer que $\chi(R)<0$ si et seulement si $\rho_R(\wcrit_R) >0$, et que $\rho_R(\HK) =1$ dans ce cas.
Comme~$\rho_R$ n'a pas d'atomes dans~$\PK$, le cas séparable d{\'e}coule directement du Th{\'e}or{\`e}me~\ref{t:exponsants}.
Le cas inséparable est traité par l'énoncé suivant.

\begin{Prop}\label{prop:cas-nonsep}
  Soit~$R$ une fraction rationnelle {\`a} coefficients dans~$K$, ins\'eparable.
  Alors $\chi(R) = -\infty$, $\rho_R(\wcrit_R) = 1$, et il existe une constante $A>0$ telle que
  \begin{equation}
    \label{eq:72}
    J_R
    \subseteq
    \{ x \in \HK, \, d_\HK(x, \xcan) \le A\}~.
  \end{equation}
\end{Prop}

La preuve de cette proposition suit le lemme suivant.

\begin{Lem}\label{lem:cocycle-general}
  Soit~$R$ une fraction rationnelle {\`a} coefficients dans~$K$, non constante, et soit~$\delta$ son degr{\'e} d'ins\'eparabilit\'e.
  Alors, il existe une constante $C>0$ telle que
  \begin{equation}\label{e:diam2}
    \diam \circ R
    \le
    C\, \diam^{\delta}~.
  \end{equation}
\end{Lem}
\begin{proof}
  Si $R$ est s\'eparable, cel\`a resulte du fait que $\wf_R$ et $\|R'\|$ sont born\'ees sup\'erieurement.
  Si $R$ est ins\'eparable, notons~$p$ la caract{\'e}ristique de~$K$.
  Alors, il suffit d'\'ecrire~$R$ comme une composition d'une application s\'eparable et d'une puissance du morphisme de Frobenius $F$, et de remarquer que $\diam \circ F = \diam^p$.
\end{proof}

\begin{proof}[D\'emonstration de la Proposition~\ref{prop:cas-nonsep}]
  La propri\'et\'e ${\chi(R) = -\infty}$ est automatique par définition.
  Pour chaque~$r$ dans~$\mathopen] 0, 1 \mathclose[$, soit~$U(r)$ l'ouvert d{\'e}finit par
  \begin{equation}
    \label{eq:32}
    U(r)
    \=
    \{ x\in \PKber, \, \diam (x) < r \}~.
  \end{equation}
  Le Lemme~\ref{lem:cocycle-general} implique que pour tout~$r$ assez petit, on a ${R(U(r)) \subseteq U(r)}$.
  En particulier, $\bigcup_{n = 0}^{+\infty} R^n(U(r))$ ne peut contenir~$\HK$, et par cons{\'e}quent ${U(r) \subseteq F_R}$.
  L'{\'e}galit{\'e} ${\rho_R(\wcrit_R) = 1}$ d{\'e}coule alors de ${\wcrit_R = \HK}$.
\end{proof}

\subsection{Fractions rationnelles {\`a} exposant de \mbox{Lyapunov} nul}
\label{ss:integrable}

Le but de cette section est de d{\'e}montrer le r{\'e}sultat suivant, utilis{\'e} dans la preuve du Th{\'e}or{\`e}me~\ref{thm:main3}~(2) dans \S~\ref{sec:thm32}, ainsi que celle du Th{\'e}or{\`e}me~\ref{thm:main2} dans \S~\ref{ss:modere}.

\begin{Prop}
  \label{prop:integrable}
  Soit $R$ une fraction rationnelle {\`a} coefficients dans~$K$, s\'eparable de degr\'e au moins deux, et telle que ${\chi(R) = 0}$.
  Alors on a ${\rho_R(\HK) = 1}$ si et seulement si
  \begin{equation}
    \label{eq:33}
    \bigcup_{n = 0}^{+\infty} R^n(\wcrit_R)
    \neq
    \HK~.
  \end{equation}
  Lorsque ces propri{\'e}t{\'e}s {\'e}quivalentes sont satisfaites, on a $\log \diam \in L^1(\rho_R)$.
\end{Prop}

\begin{Rem}
  L'existence d'un point p{\'e}riodique dans~$\wcrit_R \cap \HK$ implique l'{\'e}galit{\'e} ${\bigcup_{n = 0}^{+\infty} R^n(\wcrit_R) = \HK}$, voir \cite[Proposition~11.1]{rivera-periode}.
  Si de plus on a ${\chi(R) = 0}$, alors on obtient ${\rho_R(\PK) = 1}$ par la Proposition~\ref{prop:integrable}.
\end{Rem}

La preuve de la Proposition~\ref{prop:integrable} suit le lemme suivant.

\begin{Lem}
  \label{lem:key-comp}
  Soit $R$ une fraction rationnelle {\`a} coefficients dans~$K$, s\'eparable de degr\'e~$d$ aux moins deux.
  Pour tout $x \in \HK$, on a
  \begin{equation}
    \label{eq:key-comp}
    0 \le -\sum_{R^{-n}(x)} \deg_{R^n} \, \log \diam =
    \sum_{R^{-n}(x)} \deg_{R^n} \, \wf_{R^n} +
    n d^n \chi (R)+ \cO(d^n)~.
  \end{equation}
\end{Lem}

\begin{proof}
  Soit $g$ un potentiel tel que $\delta_x = \rho_R + \Delta g$.
  Cette fonction est continue et donc born{\'e}e.
  On \'ecrit
  \[
    - \log \diam (y) + \log \diam (x)
    =
    \wf_{R^n}(y) + \log \| (R^n)'\|(y)~,
  \]
  pour chaque $y\in R^{-n}(x)$.
  En multipliant chaque terme par $\deg_{R^n}(y)$, et en sommant tous les termes ainsi obtenus, on obtient
  \begin{multline*}
    -\sum_{R^{-n}(x)} \deg_{R^n}(y) \, \log \diam (y) + d^n \log \diam (x)
    \\
    \begin{aligned}
      &= \sum_{R^{-n}(x)} \deg_{R^n} \,\wf_{R^n} + \int \log \| (R^n)' \| \, \dd R^{n*} \delta_x
      \\ & =
           \sum_{R^{-n}(x)} \deg_{R^n} \, \wf_{R^n} + \int \log \| (R^n)' \| \, \dd R^{n*} \rho_R - \int \log \| (R^n)' \| \dd \Delta g \circ R^n
      \\ & =
           \sum_{R^{-n}(x)} \deg_{R^n} \, \wf_{R^n} + n d^n\chi (R)- \int g \circ \dd R^n \, \dd \Delta \log \| (R^n)' \|~.
    \end{aligned}
  \end{multline*}
  On conclut alors la preuve en utilisant le fait que $g$ est born\'ee, et que la variation totale de $\Delta \log \| (R^n)' \|$ est au plus \'egale \`a $4d^n$,
  voir Proposition~\ref{prop:delta}.
\end{proof}

\begin{proof}[D\'emonstration de la Proposition~\ref{prop:integrable}]
  Si ${\rho_R(\HK) = 1}$, alors on a ${\rho_R(\wcrit_R) = 0}$ par le Th{\'e}or{\`e}me~\ref{t:exponsants}, et donc ${\rho_R (\bigcup_{n = 0}^{+\infty} R^n(\wcrit_R)) = 0}$ par~\eqref{eq:6}.
  Ceci entra{\^\i}ne~\eqref{eq:33}.

  Supposons~\eqref{eq:33} et soit~$x$ dans~${\HK \setminus \bigcup_{n = 0}^{+\infty} R^n(\wcrit_R)}$.
  On a donc $\wf_R =0$ pour toute pr\'eimage it{\'e}r{\'e}e de $x$.
  Comme $\chi (R) = 0$, le Lemme~\ref{lem:key-comp} nous donne
  \begin{equation}
    \label{eq:34}
    - \frac1{d^n}\, \int \log \diam\, \dd R^{n*} \delta_x \le C
  \end{equation}
  pour tout entier positif~$n$, et une constante $C>0$ indépendante de $n$.
  Par cons{\'e}quent
  \begin{equation}
    \label{eq:35}
    d^{-n} R^{n*}\delta_x(\HK)
    \ge
    d^{-n} R^{n*}\delta_x (\{ \diam \ge \exp(-2C) \})
    \ge
    \frac{1}{2}~.
  \end{equation}
  Or $\{ \diam \ge \exp(-2C) \} $ est ferm\'e, et $d^{-n} R^{n*}\delta_x $ tend faiblement vers~$\rho_R$ lorsque ${n \to +\infty}$.
  On conclut donc ${\rho_R(\HK) \ge \frac{1}{2}}$, et ${\rho_R(\HK) = 1}$ par ergodicit{\'e}.

  Pour montrer la derni{\`e}re assertion, soit $g$ un potentiel tel que $\rho_R = \delta_{\xcan} +\Delta g$.
  Le potentiel~$g$ est continu et donc born{\'e}.
  Par ailleurs, consid\'erons $\cT$ l'adh\'erence de l'enveloppe convexe de l'union des pr\'eimages it{\'e}r{\'e}es de~$x$, du support de $\rho_R$, et de~$\xcan$.
  Pour chaque entier positif~$k$, soit~$\cT_k$ l'enveloppe convexe de~$\xcan$ et de $\bigcup_{i \in \{0, \ldots, k\}} R^{-i}(x)$.
  On peut \'ecrire $\cT$ comme l'adh\'erence de l'union croissante des arbres finis $\cT_k$.
  Notons $\varphi_k$ le potentiel \'egal \`a $\log \diam $ sur $\cT_k$ et localement constant en dehors.
  Alors $\varphi_k$ d\'ecroit et converge ponctuellement vers $\log \diam$ sur $\cT$.
  Pour tout entier positif~$n$, on a alors
  \begin{multline*}
    0 \le - \int \varphi_k\, \dd \rho_R
    =
    - \frac1{d^n} \int \varphi_k\, \dd R^{n*} \delta_x - \int \varphi_k \, \dd \Delta \frac{g\circ R^n}{d^n}
    \\ \le
    - \frac1{d^n} \int \log \diam \, \dd R^{n*} \delta_x + |\Delta \varphi_k| \,
    \frac{ \sup |g|}{d^n}
  \end{multline*}
  ce qui implique $0 \le - \int \varphi_k\, \dd \rho_R \le C$.
  On peut alors appliquer le th\'eor\`eme de convergence monotone pour conclure.
\end{proof}

\subsection{D\'emonstration du Th\'eor\`eme~\ref{thm:main3}~(2)}\label{sec:thm32}

Le point clef est le r\'esultat suivant.

\begin{Prop}
  \label{lem:key-cstdiam}
  Soit~$R$ une fraction rationnelle {\`a} coefficients dans~$K$, et de degr{\'e} au moins deux.
  Si ${\chi(R) = 0}$, ${\rho_R(\HK) = 1}$, et~$\rho_R$ ne charge aucun segment de~$\HK$, alors il existe une boule ouverte~$B$ de~$\PKber$ et une constante $c>0$, telles que $\rho_R(B)>0$, et $\log \diam =c $
  pour $\rho_R$-presque tout point appartenant \`a~$B$.
\end{Prop}

La d{\'e}monstration de cette proprosition est ci-dessous.

\begin{proof}[D\'emonstration du Th\'eor\`eme~\ref{thm:main3}~(2) en admettant la Proposition~\ref{lem:key-cstdiam}]
  Soit~$R$ une fraction rationnelle {\`a} coefficients dans~$K$ de degr{\'e} au moins deux, et v{\'e}rifiant ${\chi(R) = 0}$ et ${\rho_R(\HK) = 1}$.
  Notre but est de construire un potentiel divisoriel~$g$ et de montrer l'existence d'une constante $C>0$ tels que ${\log |R'| = g \circ R - g}$  et $\diam \ge C$ sur~$J_R$.

  Remarquons tout d'abord que le Th{\'e}or{\`e}me~\ref{t:exponsants}~(1) implique
  \begin{equation}
    \label{eq:36}
    \rho_R(\wcrit_R)
    =
    0
    \text{ et }
    \rho_R(\{ \wf_R = 0 \})
    =
    1~.
  \end{equation}
  Par cons{\'e}quent, il existe un sous-ensemble~$J_0$ de~$J_R$ de mesure pleine pour~$\rho_R$ sur lequel, pour tout~$n$ dans~$\N^*$, on a {\'e}galit{\'e}
  \begin{equation}
    \label{eq:37}
    \log \diam \circ R^n
    =
    \log \| (R^n)' \| + \log \diam~.
  \end{equation}

  Supposons que~$\rho_R$ charge un segment de $\HK$.
  Si~$\rho_R$ ne possède pas d'atome, alors~$J_R$ est contenu dans un segment de~$\HK$ par le Th\'eor\`eme~\ref{thm:main1} et la Proposition~\ref{prop:affine Bernoulli}~(2).
  Sinon, $J_R$ est r{\'e}duit {\`a} un point de type~II par \cite[Th{\'e}or{\`e}me~E]{theorie-ergo}.
  En changeant de coordonn\'ees si n{\'e}cessaire, supposons que ${J_R \subseteq \mathopen] 0, \xcan \mathclose[}$ dans les deux cas.
  Soit~$C$ dans ${\mathopen] 0, 1 \mathclose[ \cap |K^*|}$ tel que, si l'on note par~$x_0$ le point de~$\AKber$ associ{\'e} {\`a} la boule~$\{z, |z| \le C\}$, on ait ${J_R \subseteq [x_0, \xcan]}$.
  Alors, $\log \diam$ co{\"\i}ncide sur~$J_R$ avec le potentiel divisoriel~$g$, d{\'e}fini par
  \begin{equation}
    \label{eq:38}
    g(z)
    \=
    \log \max \{ |z|, C \} - \log \max \{ |z|, 1 \}~.
  \end{equation}
  Comme on a~\eqref{eq:37} sur~$J_0$ et ${\rho_R(J_0) = 1}$, on d{\'e}duit qu'on a ${\log\| R' \| = g \circ R - g}$ sur un ensemble de mesure pleine pour~$\rho_R$, et donc {\'e}galement sur tout~$J_R$ par continuit{\'e}.

  \smallskip

  Dans la suite, on supposera donc que~$\rho_R$ ne charge aucun segment de $\HK$.
  Soit~$B$ la boule ouverte de~$\PKber$ donn{\'e}e par la Proposition~\ref{lem:key-cstdiam}.
  En r{\'e}duisant~$B$ si n{\'e}cessaire, supposons que le point~$x_B$ dans le bord de~$B$ est de type~II.
  Comme l'ouvert~$B$ intersecte~$J_R$, il existe un entier positif ou nul~$l$ tel que ${R^l(B) \supseteq J_R}$.
  La fonction $\log \diam$ \'etant semi-continue sup\'erieurement pour la topologie faible, on d{\'e}duit de la Proposition~\ref{lem:key-cstdiam} que $\log \diam \ge c$ sur $J_R \cap B$.
  De plus, la fonction $\log \| (R^l)'\|$ \'etant lipschitzienne pour la m\'etrique hyperbolique, il en r{\'e}sulte que $\log \| (R^l)'\|$
  est uniform\'ement minor\'ee sur $J_R \cap B$ par une constante $c'$.
  On d{\'e}duit alors de~\eqref{eq:19} dans le Corollaire~\ref{c:tres-wild} que
  $$\log \diam \circ R^l
  \ge
  \log \| (R^l)'\| + \log \diam
  \ge
  c + c' $$
  sur~$J_R \cap B$.
  On conclut donc que ${\log \diam \ge c + c'}$ sur~$J_R$.

  Soit~$J_1$ l'ensemble des~$x$ dans~$J_0$ tel que pour tout~$x'$ dans ${R^{-l}(x) \cap B}$, on ait
  \begin{equation}
    \label{eq:39}
    \log \diam (x)
    =
    \log \| (R^l)' \|(x') + c~.
  \end{equation}
  Comme on a~\eqref{eq:37} sur~$J_0$, que ${\rho_R(J_0) = 1}$, et que la mesure de probablité~$\rho_R$ ne charge pas~${R^l(\{x \in B, \log \diam \neq c\})}$ par~\eqref{eq:6}, on en d{\'e}duit que ${\rho_R(J_1) = 1}$.

  Pour construire~$g$, on commence par construire un potentiel~$g_0$ dont le support de~$\Delta g_0$ est une union finie de points de type~I ou~II, et tel que ${\{ g_0 = \log \diam \}}$ sur~$J_1$.
  Une fois ce potentiel~$g_0$ obtenu, on choisit une constante~$c''$ dans ${\mathopen] -\infty, c + c' \mathclose]}$ telle que~$\exp(c'')$ appartienne {\`a}~$|K^*|$, et on d{\'e}finit le potentiel divisoriel~$g$ par ${g \= \max \{ g_0, c'' \}}$.
  Alors, on a ${g = \log \diam}$ sur~$J_1$.
  Comme on a~\eqref{eq:37} sur~$J_0$ et ${\rho_R(J_0) = 1}$, on d{\'e}duit qu'on a ${\log \| R' \| = g \circ R - g}$ sur un ensemble de mesure pleine pour~$\rho_R$, et donc {\'e}galement sur tout~$J_R$ par continuit{\'e}.

  Pour construire~$g_0$, notons~$p_B$ la projection canonique de~$\PKber$ sur l'arbre ferm{\'e} ${B \cup \{ x_B \}}$, et~$\varphi$ le potentiel d{\'e}fini par
  \begin{equation}
    \label{eq:40}
    \varphi
    \=
    \log \| (R^l)' \| \circ p_B + c~.
  \end{equation}
  Le point~$x_B$ est de type II, et on a ${\Delta \varphi = (p_B)_*(\Delta \log \| (R^l)' \|)}$. Par cons{\'e}quent le support de~$\Delta \varphi$ est un ensemble fini de points de type~I ou~II.
  Par \cite[Lemme~4.2]{rivera-espace}, il existe une boule ouverte~$B'$ de~$\PKber$ et deux entiers~$N > \hN \ge 0$, tels que le point~$x_{B'}$ dans le bord de~$B'$ est {\'e}gal {\`a} $R^l(x_B)$, et de telle sorte que pour tout~$x$ dans~$\PKber$ on ait
  \begin{equation}
    \label{eq:41}
    \sum_{x' \in R^{-l}(x) \cap B} \deg_{R^l}(x')
    =
    \begin{cases}
      N
      &\text{si } x \in B', \text{ et}
      \\
      \hN
      &\text{si } x \notin B'~.
    \end{cases}
  \end{equation}
  Soit~$\psi$ le potentiel d{\'e}fini par
  \begin{equation}
    \label{eq:42}
    \psi
    \=
    \frac{1}{N} ( R^l_* \varphi - (d^l - N) \varphi(x_B) )~.
  \end{equation}
  Le support de~$\Delta \psi$ est une union finie de points type~I ou~II, et pour tout~$x$ dans~$J_1 \cap B'$ on a
  \begin{equation}
    \label{eq:43}
    \psi(x)
    =
    \frac{1}{N} \sum_{x' \in R^{-l}(x) \cap B} \deg_{R^l(x')}(x') \varphi(x')
    =
    \log \diam(x)~,
  \end{equation}
  par~\eqref{eq:41} et les d{\'e}finitions de~$J_1$ et de~$\varphi$.

  Si~$\hN = 0$, alors ${J_1 \subseteq J_R \subseteq R^l(B) = B'}$ et le potentiel ${g_0 = \psi}$ satisfait les propri{\'e}t{\'e}s d{\'e}sir{\'e}es.
  D{\'e}sormais, on suppose ${\hN > 0}$.
  Choisissons un point~$\hx$ de type~II dans~$B$, tel que l'image par~$R^l$ de la couronne~$A$ comprise entre~$x_B$ et~$\hx$ est une couronne, et tel qu'on ait ${\deg_R(A) = N - \hN}$, voir \cite[Proposition~4.1]{rivera-espace}.
  En rapprochant~$\hx$ de~$x_B$ si n{\'e}cessaire, supposons que ${\Delta (R^l_* \log \|(R^l)'\|)}$ ne charge pas~$R^l(A)$.
  Posons ${A' \= R^l(A)}$ et ${\hx' \= R^l(\hx)}$, et notons que ${\hx' \neq x_{B'}}$, que~$A'$ est la couronne comprise entre~$x_{B'}$ et~$\hx'$, et que~$\psi$ est harmonique sur~$A'$.
  On considère la boule fermée~$\kB$ de~$\PKber$ d{\'e}finie par ${\kB \= B\setminus A}$, et on note~$\hB'$ la composante connexe de ${\PKber \setminus \{ \hx' \}}$ contenant~$x_{B'}$.
  Alors, ${B' \cup \hB' = \PKber}$, ${B' \cap \hB' = A'}$, et pour tout~$x$ dans~$\hB'$ on~a
  \begin{equation}
    \label{eq:44}
    \sum_{x' \in R^{-l}(x) \cap \kB} \deg_{R^l}(x')
    =
    \hN
  \end{equation}
  par~\eqref{eq:41}.
  Soit~$p_{\kB}$ la projection canonique de~$\PKber$ sur l'arbre ferm{\'e}~$\kB$, et soient~$\hvarphi$ et~$\hpsi$ les potentiels d{\'e}finis par
  \begin{equation}
    \label{eq:45}
    \hvarphi
    \=
    \log \| (R^l)' \| \circ p_{\kB} + c
    \text{ et }
    \hpsi
    \=
    \frac{1}{\hN} ( R^l_* \hvarphi - (d^l - \hN) \hvarphi(\hx) )~.
  \end{equation}
  Le support de chacune des mesures~$\Delta \hvarphi$ et~$\Delta \hpsi$ est une union finie de points de type~I ou~II.
  Par ailleurs, $\hpsi$ est harmonique sur~$A'$, et pour tout~$x$ dans~$J_1 \cap \hB'$ on a
  \begin{equation}
    \label{eq:46}
    \hpsi(x)
    =
    \frac{1}{\hN} \sum_{x' \in R^{-l}(x) \cap \kB} \deg_{R^l(x')}(x') \hvarphi(x')
    =
    \log \diam(x)~,
  \end{equation}
  par~\eqref{eq:44} et les d{\'e}finitions de~$J_1$ et de~$\hvarphi$.

  Pour conclure, on pose ${\ell' \= [x_{B'}, \hx']}$ et on note~$p_{\ell'}$ la projection canonique de~$\PKber$ {\`a}~$\ell'$.
  Les fonctions~$\psi$ et~$\hpsi$ {\'e}tant harmoniques sur~$A'$, leurs restrictions {\`a}~$\ell'$ sont affines pour la distance hyperbolique, et on a ${\psi \circ p_{\ell'} = \psi}$ et ${\hpsi \circ p_{\ell'} = \hpsi}$.
  Comme ${\psi = \hpsi}$ sur ${J_1 \cap A'}$ par~\eqref{eq:43} et~\eqref{eq:46}, si ${p_{\ell'}(J_1 \cap A')}$ contient au moins deux points, alors on a ${\psi = \hpsi}$ sur~$A'$.
  Dans ce cas, la fonction~$g_0$ {\'e}gale {\`a}~$\psi$ sur~$B'$ et {\`a}~$\hpsi$ sur~$\hB'$ co{\"{\i}}ncide avec~$\log \diam$ sur~$J_1$ par~\eqref{eq:43} et~\eqref{eq:46}, et est telle que~$\Delta g_0$ est une union finie de points de type~I ou~II.
  La fonction~$g_0$ satisfait donc les propri{\'e}t{\'e}s d{\'e}sir{\'e}es.
  Il reste {\`a} traiter les cas o{\`u} ${p_{\ell'}(J_1 \cap A')}$ contient au plus un point.
  On choisit un point~$x_{\bullet}$ de type~II entre~$x_{B'}$ et~$\hx'$ tel que ${p_{\ell'}(J_1 \cap A')}$ ne contient aucun point entre~$x_{B'}$ et~$x_{\bullet}'$.
  La couronne~$A_{\bullet}$ comprise entre~$x_{B'}$ et~$x_{\bullet}'$ est alors disjointe de~$J_1$.
  Notons~$B_{\bullet}$ la composante connexe de ${\PKber \setminus A_{\bullet}}$ distincte de~$B'$, et~$g_0$ l'unique potentiel harmonique sur~$A_{\bullet}$, {\'e}gal {\`a}~$\psi$ sur~$B_{\bullet}$ et {\`a}~$\hpsi$ sur ${\PKber \setminus B'}$.
  Cette fonction co{\"{\i}}ncide avec~$\log \diam$ sur~$J_1$ par~\eqref{eq:43} et~\eqref{eq:46}, et son Laplacien est une union finie de points de type~I ou~II.
  Ceci montre que la fonction~$g_0$ satisfait les propri{\'e}t{\'e}s d{\'e}sir{\'e}es et termine la d\'emonstration du Th\'eor\`eme~\ref{thm:main3}~(2).
\end{proof}

Le reste de cette partie est consacr\'e \`a la d\'emonstration de la Proposition~\ref{lem:key-cstdiam}.
Fixons une fraction rationnelle~$R$ {\`a} coefficients dans~$K$, de degr{\'e} aux moins deux, et telle que~$\rho_R$ ne charge aucun segment de~$\HK$.
On commence par construire des pr{\'e}images it{\'e}r{\'e}es {\'e}tales, en adaptant la construction des branches inverses des fractions rationnelles complexes introduite ind{\'e}pendamment dans~\cite{FreLopMan83} et~\cite{lyubich}.

Pour chaque~$n$ dans~$\N^*$, notons~$\cT_n$ l'enveloppe convexe de~$\crit_R(K)$, $\xcan$, et~$R^{-n}(\xcan)$.
Notons que~$\| (R^n)' \|$ est localement cons\-tante en dehors de $\cT_n$, voir la Proposition~\ref{prop:delta}.
Par ailleurs, on note~$\aT_n$ l'image par~$R^n$ de l'enveloppe convexe de~$R^{-n}(\cT_n)$.
C'est un arbre fini qui contient $\cT_n$ et tel que $R^{-n}(\aT_n)$ est connexe.
Par cons{\'e}quent, pour toute boule ouverte~$B$ de~$\PKber$ disjointe de~$\aT_n$, chaque composante connexe~$B'$ de~$R^{-n}(B)$ est une boule ouverte de~$\PKber$, et pour chaque~$j$ dans~$\{ 0, \ldots, n - 1 \}$ on a ${R^j(B') \cap \aT_1 = \emptyset}$.
Notons de plus que les suites~$\{ \cT_n \}_{n = 0}^{+\infty}$ et~$\{ \aT_n \}_{n = 0}^{+\infty}$ sont croissantes pour l'inclusion.

\begin{Lem}\label{lem:big-return}
  Pour tout $\varepsilon>0$, il existe~$p \in \N^*$ tel que, pour tout $n\in \N$ et toute boule ouverte~$B_0$ de~$\PKber$ disjointe de~$\aT_p$ et intersectant~$J_R$, l'union des composantes connexes~$B$ de~$R^{-n}(B_0)$ telles que pour chaque~$j$ dans~$\{ 0, \ldots, n - 1 \}$ on ait ${R^j(B) \cap \aT_1 = \emptyset}$ poss{\`e}\-de une masse d'au moins ${(1- \e)\rho_R(B_0)}$.
\end{Lem}

\begin{proof}
  Notons~$d$ le degr{\'e} de~$R$.
  Comme~$R$ n'a pas bonne r{\'e}duction potentielle, la d{\'e}monstration de~\cite[Lemme~2.12]{theorie-ergo} conduit {\`a} deux cas~:
  \begin{enumerate}
  \item
    Si~$R$ n'est pas conjugu\'ee \`a un polyn\^ome, alors il existe un entier positif~$N$ tel que ${\deg_{R^N} < d^N}$ sur~$\PKber$.
  \item
    Si~$R$ est conjugu{\'e}e {\`a} un polyn{\^o}me, alors il existe un point~$x$ de type~II tel que~$J_R$ est contenu dans un nombre fini de composantes connexes de~${\PKber \setminus \{ x \}}$, sur lesquelles on a ${\deg_R < d}$.
    Dans ce cas, on pose ${N = 1}$ et on change de coordonn{\'e}es pour que ${R(x) = \xcan}$.
  \end{enumerate}
  Dans tous les cas, on a ${\deg_{R^N} \le d^N - 1}$ sur toute boule de~$\PKber$ disjointe de~$\cT_N$ et intersectant~$J_R$.

  Notons~$\beta$ le nombre d'extr{\'e}mit{\'e}s de~$\cT_N$ et posons ${\delta \= (d^N - 1)/d^N}$.
  {\'E}tant donn{\'e}~$\varepsilon$ dans~$\mathopen] 0, +\infty \mathclose[$, soit~$p_0$ le plus petit entier tel que ${p_0 \ge 1}$ et ${\beta \sum_{j = p_0}^{+\infty} \delta^j < \varepsilon}$, et posons ${p \= N p_0}$.
  Fixons une boule ouverte~$B_0$ de~$\PKber$ disjointe de~$\aT_p$ et intersectant~$J_R$.
  Pour chaque~$n$ dans~$\N^*$, notons $\mathcal{B}_n$ l'ensemble des composantes connexes de~$R^{-n}(B_0)$ tel que pour chaque~$j$ dans~$\{ 0, \ldots, n - 1 \}$ on ait ${R^j(B) \cap \aT_1 = \emptyset}$, et posons
  \[d_n
    \=
    \sum_{B\in \mathcal{B}_n} \deg_R(B)~.
  \]
  Il suffit de montrer que pour tout~$n$ on a ${d_n \ge (1 - \varepsilon) d^n}$.

  Par construction, pour chaque~$n$ dans~$\{1, \ldots, p\}$ on a $d_n = d^n$.
  Par ailleurs, pour chaque~$n$ et~$B\in \mathcal{B}_n$, ainsi que pour toute composante connexe $B'$ de $R^{-N}(B)$, nos choix ci-dessus impliquent $\deg_R(B') \le d^N - 1$.
  De plus, si~$B'$ est disjointe de $\aT_N$, alors~$B' \in \mathcal{B}_{n+N}$.
  Comme toute boule ouverte de~$\PKber$ intersectant~$\cT_N$ contient une extr{\'e}mit{\'e} de~$\cT_N$, on en d\'eduit
  \[
    d_{n+N}
    \ge
    d_n d^N - \beta \delta^{\left\lfloor \frac{n}{N} \right\rfloor} d^n~.
  \]
  Par récurrence, pour tout~$n$ satisfaisant ${n > p}$ on obtient
  \begin{equation*}
    d_n
    \ge
    (1 - \beta \sum_{j = p_0}^{+\infty} \delta^j) d^n
    \ge
    (1-\varepsilon) d^n~.
    \qedhere
  \end{equation*}
\end{proof}

Fixons $n$ dans~$\N^*$.
Pour tout $p$ dans~$\N^*$, notons~$\mathcal{B}_n^p$ l'ensemble des composantes connexes de $R^{-n}(\PKber \setminus \aT_p)$ intersectant~$J_R$ et tel que pour chaque~$j$ dans ${\{ 0, \ldots, n - 1 \}}$ on ait ${R^j(B) \cap \aT_1 = \emptyset}$.
Notons que chaque {\'e}l{\'e}ment~$B$ de~$\mathcal{B}_n^p$ est une boule ouverte de~$\PKber$ o{\`u}~$\|(R^n)'\|$ est constante, et que~$R^n(B)$ est une composante connexe de~${\PKber \setminus \aT_p}$.
Comme $\rho_R$ ne charge aucun segment de~$\HK$, on a ${\rho_R \left( J_R \setminus \bigcup_{p = 1}^{+\infty} \aT_p \right) = 1}$, et le Lemme~\ref{lem:big-return} montre que ${\lim_{p \to +\infty} \rho_R(\bigcup \mathcal{B}_n^p) = 1}$.

Pour $\rho_R$-presque tout $x\in J_R$, il existe donc un entier $p$ et une boule
de $\mathcal{B}_n^p$ contenant~$x$.
On notera $B_n(x)$ la boule obtenue de la sorte pour $p$ minimal.
Les propri\'et\'es suivantes sont satisfaites:
\begin{itemize}
\item[(P1)]
  $R(B_{n + 1}(x))\subseteq B_n(R(x))$;
\item[(P2)]
  $B_{n + 1}(x) \subseteq B_n(x)$;
\item[(P3)]
  $\bigcap_{n = 1}^{+\infty} B_n(x) \setminus \{x\}\subseteq F_R$.
\end{itemize}
Pour le dernier point, notons que $\bigcap_{n = 1}^{+\infty} B_n(x)$ est une boule contenant $x$ dont l'intérieur
est dans~$F_R$ car toutes ses images par les it\'er\'es de $R$ \'evitent $\aT_1$.

On s'appuie maintenant sur la version suivante du lemme de diff\'erentiation de \mbox{Lebesgue}.

\begin{Lem}
  \label{lem:quasi-cst}
  Soit $\phi \in L^1(\rho_R)$. Alors
  pour $\rho_R$-presque tout $x$, on a
  \[
    \lim_{n \to +\infty} \frac1{\rho_R(B_n(x))} \int_{B_n(x)} |\phi(x) - \phi(y)| \, \dd \rho_R(y)
    =
    0~.
  \]
\end{Lem}

\begin{proof}
  La d\'emonstration suit tr\`es pr\'ecis\'ement celle du th\'eor\`e\-me de diff\'erentiabilit\'e de Lebesgue donn\'ee dans~\cite{mattila}.
  On peut toujours supposer $\phi \ge0$, et on introduit
  \[
    \underline{D}(x) = \liminf_n \frac{\int_{B_n(x)} \phi \, \dd \rho_R}{\rho_R(B_n(x))},
    \text{ et }
    \overline{D}(x) = \limsup_n \frac{\int_{B_n(x)} \phi \, \dd\rho_R}{\rho_R(B_n(x))}.
  \]
  On veut montrer l'analogue de~\cite[Lemma~2.13]{mattila}~:
  si $A$ est un bor\'elien quelconque et que $\underline{D}(x) \le t$ (resp. $\overline{D}(x)\ge t$) pour tout $x \in A$, alors $\int_A \phi\, \dd \rho_R \le t \rho_R(A)$ (resp. $\int_A \phi\, \dd \rho_R \ge t \rho_R(A)$). A partir de l\`a, la preuve de~\cite[Theorem~2.12]{mattila} s'adapte \`a notre contexte sans difficult\'e, puis les arguments de~\cite[Remark~2.15 (3)]{mattila} s'appliquent, et notre lemme s'ensuit.

  Notons qu'il suffit de traiter la situation $\underline{D}(x) \le t$ sur $A$.
  Pour tout $\e>0$ fix\'e, et pour tout entier $p\ge1$, on consid\`ere la famille de boules
  \[
    \mathcal{G}_p
    =
    \left\{ B, \, B = B_n(x) \text{ avec } n\ge p,\, x \in A, \, \int_B \phi \, \dd \rho_R \le (t+ \e) \rho_R(B) \right\}~.
  \]
  Par hypoth\`ese, pour tout point $x\in A$, et pour tout entier $p$, il existe un entier $n\ge p$ tel que $B_n(x) \in \mathcal{G}_p$ et donc
  $A \subseteq \bigcup \mathcal{G}_p$.

  Comme $\rho$ est une mesure de Radon, il existe un ouvert fondamental $U\supseteq A$  dont le bord est fini tel que $\rho(A) \ge \rho(U) - \epsilon$.

  Soit $x$ un point du bord de $U$, et supposons que $B_n(x)$ soit bien d\'efini pour tout~$n$.
  Alors $B \= \bigcap_{n = 1}^{+\infty} B_n(x)$ est une boule (ouverte ou ferm\'ee) contenant~$x$.
  Comme on~a ${R^n(B) \cap \aT_1 = \emptyset}$ pour tout $n\ge 0$, l'ensemble $B\setminus \partial B$ est inclus dans~$F_R$ et par suite $\rho_R(B) = 0$.
  Si le bord de $B_n(x)$ est dans $U$ pour tout $n$, alors $U\cup B$ est un ouvert dont le bord est $\partial U \setminus \{x\}$. Dans ce cas, on remplace $U$ par $U\cup B$.
  Sinon $B_n(x)$ contient $U$ pour tout~$n$, et on remplace~$U$ par $U\cup B'$ o\`u $B'$ est la composante connexe de l'int\'erieur de~$B$ contenant~$U$.
  Notons alors que $U$ est une boule ouverte de bord~$x$.

  Quitte \`a faire cette op\'eration un nombre fini de fois, on obtient que soit $U$ est une boule ouverte de bord $x$ et $U\subseteq B_p(x)$ pour tout~$p$;
  soit $U$ est une union finie d'ouverts fondamentaux $U\supseteq A$
  telle que pour chaque point du bord $x$ il existe un entier $p$ tel que $x \notin B_p(y)$ pour tout $y$.
  Dans ce dernier cas, on en d\'eduit que $\cup \mathcal{G}_p\subseteq U$ pour $p$ assez grand, et donc $\rho_R(\cup \mathcal{G}_p) \le \rho_R(U) \le \rho_R(A) + \epsilon$.
  Dans le premier cas, on a aussi $\cup \mathcal{G}_p\subseteq U$ car sinon on pourrait trouver deux boules $B_p(x)$ et $B_p(y)$ recouvrant $\PKber$ ce qui est impossible.

  On remarque maintenant que deux boules dans $\mathcal{G}_p$ sont soit disjointes soit incluses l'une dans l'autre. Par cons\'equent, la famille $\mathcal{G}'_p$ des boules de $\mathcal{G}_p$ maximales pour l'inclusion
  et de masse positive forme un recouvrement d\'enombrable par boules disjointes de $A$.
  On obtient alors
  \[\int_A \phi \, \dd \rho_R
    \le
    \sum_{\mathcal{B}'_p} \int_B \phi \, \dd \rho_R
    \le
    (t + \e) \, \rho_R(\cup\mathcal{B}'_p)
    \le (t+ \e) \, (\rho_R(A)+\epsilon)
  \]
  et on conclut en faisant tendre $\e$ vers $0$.
\end{proof}

\begin{proof}[D\'emonstration de la Proposition~\ref{lem:key-cstdiam}]
  Soit~$R$ une fraction rationnelle {\`a} coefficients dans le corps~$K$, de degr{\'e} au moins deux, et v{\'e}rifiant ${\chi(R) = 0}$, ${\rho_R(\HK) = 1}$, et~$\rho_R$ ne charge aucun segment de~$\HK$.
  Posons ${\phi \= \log \diam}$ et notons qu'on a ${\rho_R(\{ \phi > 0 \}) = 1}$, car~$\rho_R$ ne charge pas~$\xcan$, et que ${\phi \in L^1(\rho_R)}$ par la Proposition~\ref{prop:integrable}.
  Par le Th{\'e}or{\`e}me~\ref{t:exponsants}~(1), on a ${\rho_R(\wcrit_R) = 0}$ et ${\rho_R(\{ \wf_R = 0 \}) = 1}$.
  Il existe donc un sous-ensemble~$J_0$ de~$J_R$ de mesure pleine pour~$\rho_R$, sur lequel, pour tout~$n$ dans~$\N^*$, on a l'{\'e}galit{\'e}
  \[{\phi \circ R^n = \log \| (R^n)' \| + \phi}.\]
  Fixons $\eta >0$ arbitraire, et pour chaque~$n$ dans~$\N^*$ posons
  \[
    E_n \= \left\{ x \in J_R\cap \HK, \,
      \frac1{\rho_R(B_n(x))} \int_{B_n(x)} |\phi(x) - \phi(y)| \, \dd \rho_R(y) \le \eta^2\right\}~.
  \]
  Soit~$p$ donn{\'e} par le Lemme~\ref{lem:big-return} avec ${\varepsilon = \frac{1}{2}}$, et soit~$B_0$ une composante connexe de ${\PKber \setminus \aT_p}$ intersectant~$J_R$.
  Pour tout~$n$, l'union des composantes connexes de $R^{-n}(B_0)$ appartenant \`a $\mathcal{B}^p_n$ est de masse au moins~$\rho_R(B_0)/2$.
  Par le Lemme~\ref{lem:quasi-cst}, il existe~$n$ tel que ${\rho_R(E_n) > 1 - \rho_R(B_0)/2}$.
  Par suite, on peut trouver un point $x_0 \in E_n \cap J_0$ et~$q$ dans~$\{1, \ldots, p\}$ tel que $B_n(x_0)$ est dans~$\mathcal{B}^q_n$ et~$R^n(B_n(x_0))$ est la composante connexe~$\hB_0$ de ${\PKber \setminus \aT_q}$ contenant~$B_0$.
  Notons que pour tout~$y$ dans ${B_n(x_0) \cap J_0}$ on~a
  \[
    \phi (R^n(x_0)) - \phi (R^n(y))
    =
    \phi (y) -\phi (x_0)~.
  \]
  Par ailleurs, on a
  \begin{equation}
    \label{eq:47}
    R_*^n \rho_R|_{B_n(x_0)}
    =
    \frac{\deg_{R^n}(B_{n}(x_0))}{d^n} \rho_R|_{\hB_0}
    =
    \frac{\rho_R(B_{n}(x_0))}{\rho_R(\hB_0)} \rho_R|_{\hB_0}
  \end{equation}
  par le Lemme~\ref{lem:deg} et \cite[Lemme~4.4~(2)]{theorie-ergo}.
  On en d\'eduit que
  \begin{multline*}
    \rho_R \left\{ y \in B_0, \, |\phi( R^n(x_0)) - \phi(y)| \ge \eta \right\}
    \\
    \begin{aligned}
      &\le
        \frac1{\eta}\, \int_{\hB_0} |\phi(R^n(x_0)) - \phi(y)| \, \dd \rho_R(y)
      \\
      &=
        \frac{\rho_R(\hB_0)}{\eta\, \rho_R(B_{n}(x_0))}\, \int_{B_{n}(x_0)} |\phi(R^n(x_0)) - \phi(R^n(z))| \, \dd \rho_R(z)
      \\&=
      \frac{\rho_R(\hB_0)}{\eta}\, \frac{\int_{B_{n}(x_0)} |\phi(x_0) - \phi(z)| \, \dd \rho_R(z)}
      {\rho_R(B_{n}(x_0))}
      \\ & \le
           \eta~.
    \end{aligned}
  \end{multline*}
  En faisant tendre $\eta$ vers $0$, on conclut que~$\phi$ est constante $\rho_R$-presque partout sur $B_0$.
\end{proof}


\subsection{D\'emonstration du Th\'eor\`eme~\ref{thm:main3}~(3)}\label{sec:briendduval}
Fixons une fraction rationnelle~$R$ {\`a} coefficients dans~$K$, de degr{\'e}~$d$ au moins deux, et v{\'e}rifiant ${\chi(R) > 0}$.
Alors~$R$ est s\'eparable, et l'entropie m{\'e}trique de~$\rho_R$ et l'entropie topologique de~$R$ sont toutes deux \'egales \`a $\log d$ \cite[Th\'eor\`emes~C et~D]{theorie-ergo}.
Pour chaque~$n$ dans~$\N^*$, notons~$P_n$ la mesure de comptage sur~$\RFix(R^n, K)$, d{\'e}finie sur la tribu des bor{\'e}liens de~$\PKber$.
Notre but est de montrer que~$(d^{-n} P_n)_{n \in \N^*}$ converge vaguement vers la mesure de probabilité~$\rho_R$.
Nous suivons la preuve d'un r{\'e}sultat analogue dans le cas complexe donn\'ee par Briend et Duval dans~\cite[\S~3]{briend:Lyap}, en d\'etaillant quelques points.
Dans notre situation non-archim\'edienne, certaines estimations sont plus simples.

Remarquons que l'on a $\rho_R(\PK) = 1$ par le Th{\'e}or{\`e}me~\ref{t:exponsants}~(1).
Notons ${\sigma \: \hat{J} \to \hat{J}}$ l'extension naturelle de~$R$ restreinte {\`a} ${J_R \cap \PK}$, o{\`u}
\begin{equation}
  \label{eq:48}
  \hat{J}
  \=
  \{ (x_i)_{i \in \Z} \in (J_R \cap \PK)^\Z, \, R(x_i) = x_{i+1}\}
\end{equation}
et~$\sigma$ est le d\'ecalage \`a droite.
Notons aussi $\pi\colon \hat{J}\to J_R \cap \PK$ la projection d{\'e}finie par $\pi((x_i)_{i \in \Z}) \= x_0$, de telle sorte que $ R \circ \pi = \pi \circ \sigma$, et~$\hat{\rho}$ l'unique mesure de probabilit\'e invariante par~$\sigma$ se projetant sur $\rho_R$.

\begin{Lem}
  \label{l:comptage-repulsif}
  Pour tout~$\epsilon$ dans~$\mathopen] 0, 1 \mathclose[$, il existe~$\tau_{\epsilon}$ dans~$\mathopen] 0, 1 \mathclose[$ et un sous-ensemble~$E_\epsilon$ de~$\hat{J}$, tels que ${\hat{\rho}(E_{\varepsilon}) \ge 1 - \epsilon}$, et pour toute boule ferm{\'e}e~$B$ de~$\PK$ de diam{\`e}tre projectif au plus~$\tau_{\epsilon}$, nous avons
  \[
    \liminf_{n \to +\infty} d^{-n}P_n(B)
    \ge
    \hat{\rho} ( \pi^{-1} (B) \cap E_\epsilon)~.
  \]
\end{Lem}

La d{\'e}monstration de ce lemme est ci-dessous.
Tout d'abord, nous d{\'e}duisons le Th\'eor\`eme~\ref{thm:main3}~(3).

\begin{proof}[D\'emonstration du Th\'eor\`eme~\ref{thm:main3}~(3) en admettant le Lemme~\ref{l:comptage-repulsif}]
  Pour chaque~$r$ dans~$\mathopen] 0, 1 \mathclose[$, notons~$\cB(r)$ la collection des boules ferm{\'e}es de~$\PK$ de diam{\`e}tre projectif~$r$.
  De plus, pour chaque sous-ensemble~$X$ de~$\PKber$, notons~$\overline{X}$ sa fermeture dans~$\PKber$.

  Soit~$\nu$ une mesure obtenue comme limite d'une sous-suite de~$(d^{-n}P_n)_{n = 1}^{+\infty}$, et soient~$r_0$ dans~$\mathopen] 0, 1 \mathclose[$ et~$B_0$ dans~$\cB(r_0)$.
  De plus, soit~$\epsilon$ dans~$\mathopen] 0, 1 \mathclose[$, et soient~$\tau_\epsilon$ et~$E_\epsilon$ fournis par le Lemme~\ref{l:comptage-repulsif}.
  Alors, pour tout~$r$ dans~$\mathopen] 0, r_0 \mathclose[$ v{\'e}rifiant ${r \le \tau_{\epsilon}}$, nous avons
  \begin{equation}
    \label{eq:49}
    \sum_{B \in \cB(r), B \subset B_0} \nu(\overline{B})
    \ge
    \sum_{B \in \cB(r), B \subset B_0} \hat{\rho} ( \pi^{-1} (B) \cap E_\epsilon)
    =
    \hat{\rho} ( \pi^{-1} (B_0) \cap E_\epsilon)~,
  \end{equation}
  et par cons{\'e}quent
  \begin{equation*}
    \nu(B_0)
    =
    \lim_{r \to 0} \sum_{B \in \cB(r), B \subset B_0} \nu(\overline{B})
    \ge
    \lim_{\varepsilon \to 0} \hat{\rho} ( \pi^{-1} (B_0) \cap E_\epsilon)
    =
    \hat{\rho} ( \pi^{-1} (B_0))
    =
    \rho_R(B_0)~.
  \end{equation*}
  Comme ${\nu(\PKber) \le 1}$ et ${\sum_{B_0 \in \cB(r_0)} \rho(B_0) = \rho_R(\PK) = 1}$, on en d{\'e}duit que pour chaque~$B_0$ dans~$\cB(r_0)$, on a ${\nu(B_0) = \rho_R(B_0)}$.
  {\'E}tant donn{\'e} que ceci est valable pour tout~$r_0$ dans~$\mathopen] 0, 1 \mathclose[$, il en r{\'e}sulte que ${\nu = \rho_R}$.
  Comme~$\nu$ est une mesure qui est un point d'accumulation arbitraire de~$(d^{-n}P_n)_{n = 1}^{+\infty}$, cette suite converge vaguement vers~$\rho_R$.
\end{proof}

La preuve du Lemme~\ref{l:comptage-repulsif} suit le lemme suivant.
Posons ${\kappa \= \exp(-\chi(R)/2)}$.

\begin{Lem}
  \label{l:branch}
  Il existe un sous-ensemble~$\hat{J}_0$ de~$\hat{J}$, de mesure pleine pour~$\hat{\rho}$, et deux fonctions mesurables ${\tau, L \: \hat{J}_0 \to \mathopen] 0, +\infty \mathclose[}$, telles que pour tous~$\hat{x}$ dans~$\hat{J}_0$ et~$n$ dans~$\N^*$, la fonction $R^n$ admet une branche inverse envoyant~$x_0$ dans~$x_{-n}$, d{\'e}finie sur $B(x_0, \tau(\hat{x}))$ et {\`a} valeurs dans~$B(x_n, L(\hat{x}) \kappa^n)$.
\end{Lem}

\begin{proof}
  Comme~$\log \| R'\|$ est intégrable pour~$\rho_R$, la fonction $\hat{x} \mapsto \log \| R' \| \circ \pi$ est aussi int\'egrable pour la mesure ergodique $\hat{\rho}$, et le th\'eor\`eme de Birkhoff appliqu\'e \`a~$\sigma^{-1}$ donne que pour tout~$\hat{x}$ dans un sous-ensemble~$\hat{J}_0$ de mesure pleine pour~$\hat{J}$ on a
  \begin{equation}
    \label{eq:50}
    \lim_{n \to +\infty} \frac{1}{n} \log \| (R^n)' \| (x_{-n})
    =
    \chi(R)~.
  \end{equation}

  Soit~$C$ fournie par le Lemme~\ref{lem:univalent}.
  Pour chaque~$\hat{x}$ dans~$\hat{J}_0$ et~$n$ dans~$\N^*$, posons
  \begin{equation}
    \label{eq:51}
    L_n(\hat{x})
    \=
    C \| R' \| (x_n) \times \kappa^{n/2}
    \text{ et }
    \tau_n(\hat{x})
    \=
    L_n(\hat{x}) \times \| (R^n)' \| (x_n) \times \kappa^n~.
  \end{equation}
  Par ailleurs, notons ${L, \tau \: \hat{J}_0 \to \mathopen[ 0, +\infty \mathclose]}$ les fonctions mesurables d{\'e}finies par
  \begin{equation}
    \label{eq:52}
    L(\hat{x})
    \=
    \sup_{n \in \N^*} L_n(\hat{x})
    \text{ et }
    \tau(\hat{x})
    \=
    \inf_{n \in \N^*} \tau_n(\hat{x})~.
  \end{equation}
  Pour tout~$\hat{x}$ dans~$\hat{J}_0$, on a
  \begin{equation}
    \label{eq:53}
    \lim_{n \to +\infty} \frac1n \log \| R' \|(x_{-n})
    =
    0~,
    \lim_{n \to +\infty} L_n(\hat{x})
    =
    0
    \text{ et }
    \lim_{n \to +\infty} \tau_n(\hat{x})
    =
    +\infty
  \end{equation}
  par~\eqref{eq:50}.
  Par cons{\'e}quent, ${0 < L(\hat{x}), \tau(\hat{x}) < +\infty}$.

  Fixons~$\hat{x}$ dans~$\hat{J}_0$, et pour chaque~$n$ dans~$\N$ posons
  \begin{equation}
    \label{eq:54}
    r_n(\hat{x})
    \=
    \tau(\hat{x}) \| (R^n)' \| (x_n)^{-1}
    \text{ et }
    B_n(\hat{x})
    \=
    B(x_n, r_n(\hat{x}))~,
  \end{equation}
  et notons qu'on~a
  \begin{equation}
    \label{eq:55}
    r_n(\hat{x})
    \le
    L(\hat{x}) \kappa^n
    \text{ et }
    r_n(\hat{x})
    \le
    C \| R' \| (x_n)~.
  \end{equation}
  En combinaison avec le Lemme~\ref{lem:univalent}, ceci entra{\^\i}ne que~$R$ est univalente sur~$B_n(\hat{x})$ et que ${R(B_n(\hat{x})) = B_{n - 1}(\hat{x})}$.
  Par récurrence, $R^n$ est univalente sur~$B_n(\hat{x})$ et satisfait ${R^n(B_n(\hat{x})) = B_0(\hat{x})}$.
\end{proof}

\begin{proof}[D{\'e}monstration du Lemme~\ref{l:comptage-repulsif}]
  Pour chaque~$L$ dans ${\mathopen] 0, +\infty \mathclose[}$ et~$\tau$ dans~$\mathopen] 0, 1 \mathclose[$, notons~$E_{L, \tau}$ l'ensemble des $\hat{x}\in \hat{J}$ tels que pour tout~$n$ dans~$\N^*$, $R^{n}$ admet une branche inverse~$R^{-n}_{\hat{x}}$
  envoyant $x_0$ sur~$x_{-n}$, d{\'e}finie sur $B(x_0, \tau)$, et {\`a} valeurs dans~$B(x_{-n}, L \kappa^n)$.
  L'ensemble $\bigcup_{L,\tau >0} E_{L,\tau}$ est de mesure totale pour~$\hat{\rho}$ par le Lemme~\ref{l:branch}.

  {\'E}tant donn{\'e}~$\epsilon$ dans~$\mathopen] 0, 1 \mathopen[$, fixons~$L$ assez grand et~$\tau$ assez petit, pour que ${\hat{\rho} (E_{L,\tau}) \ge 1 - \epsilon}$.
  Soit~$r$ dans~$\mathopen] 0, \tau \mathclose[$, soit~$B$ une boule ferm{\'e}e de~$\PK$ de diam{\`e}tre projectif~$r$, et posons ${\hat{B} \= \pi^{-1} (B)}$.
  La mesure~$\rho_R$ \'etant m\'elangeante \cite[Proposition~3.5]{theorie-ergo}, la mesure $\hat{\rho}$ l'est aussi \cite[p.241]{CFS}, et
  on a donc
  \begin{equation}
    \label{eq:56}
    \lim_{n \to +\infty} \hat{\rho} \left( \sigma^{-n} (\hat{B}) \cap\hat{B} \cap E_{L,\tau}\right)
    =
    \hat{\rho} (\hat{B})\, \hat{\rho} (\hat{B}\cap E_{L,\tau})
    =
    \rho_R(B) \, \hat{\rho} (\hat{B}\cap E_{L,\tau})~.
  \end{equation}

  Pour chaque~$n$ dans~$\N^*$ satisfaisant $L \kappa^{n} < r$ et $\hat{y} \in \sigma^{-n} (\hat{B}) \cap \hat{B} \cap E_{L,\tau}$, consid{\'e}rons la branche inverse~$R^{-n}_{\hat{y}}$ de $R^{n}$ fournie par la d{\'e}finition de~$E_{L, \tau}$.
  On a
  \[
    R^{-n}_{\hat{y}} (B)
    \subseteq
    B(y_{-n}, L \kappa^n)
    \subseteq
    B
    \subseteq
    B(y_0, \tau)
    \text{ et }
    \rho_R(R^{-n}_{\hat{y}} (B))
    =
    d^{-n} \rho_R(B)~.
  \]
  Or le lemme de Schwarz produit un point fixe attractif de~$R^{-n}_{\hat{y}}$, voir \cite[\S~1.3.1]{R1}.
  C'est un point fixe r\'epulsif de~$R^n$ dans~$R^{-n}_{\hat{y}} (B)$.
  Si l'on note
  \begin{equation}
    \label{eq:57}
    \mathcal{P}_n
    \=
    \{ R^{-n}_{\hat{y}} (B), \hat{y} \in \sigma^{-n} (\hat{B}) \cap \hat{B} \cap E_{L,\tau} \}~,
  \end{equation}
  alors les {\'e}léments de~$\mathcal{P}_n$ sont disjoints deux {\`a} deux et par cons{\'e}quent ${P_n(B) \ge \# \mathcal{P}_n}$.
  Par ailleurs, on a
  \begin{equation}
    \label{eq:58}
    \bigcup_{B' \in \cP_n} \sigma^n(\pi^{-1}(B'))
    =
    \sigma^{-n} (\hat{B}) \cap \hat{B} \cap E_{L,\tau}~,
  \end{equation}
  et par cons{\'e}quent
  \begin{equation}
    \label{eq:59}
    d^{-n} P_n(B) \rho_R(B)
    =
    \sum_{B' \in \cP_n} \hat{\rho}(\pi^{-1}(B'))
    =
    \hat{\rho} \left( \sigma^{-n} (\hat{B}) \cap \hat{B} \cap E_{L,\tau}\right)~.
  \end{equation}
  Combin{\'e} avec~\eqref{eq:56}, ceci implique le lemme avec ${\tau_{\epsilon} = \tau/2}$ et ${E_{\epsilon} = E_{L, \tau}}$.
\end{proof}

\section{Applications}
\label{s:modere}

Dans cette section, nous donnons des applications des résultats précédents aux fractions rationnelles mod{\'e}r{\'e}es (\S~\ref{ss:modere}) ainsi qu'aux familles m{\'e}romorphes de fractions rationnelles complexes (\S~\ref{sec:application}).

\subsection{Fractions rationnelles mod\'er\'ees}
\label{ss:modere}

Dans cette section on traite la dynamique des fractions rationnelles mod\'er\'ees.
Nous commencons par rappeler la d\'efinition et les principales propri\'et\'es de ces applications (Proposition~\ref{prop:modere} et Corollaire~\ref{cor:critere simple}).
Nous montrons ensuite qu'une fraction rationnelle mod\'er\'ee, dont la mesure d'\'equilibre ne charge aucun segment de~$\HK$, poss\`ede un point p{\'e}riodique r{\'e}pulsif dans~$\PK$ (Proposition~\ref{prop:expansion}).
Nous expliquons ensuite comment ce r{\'e}sultat, combin\'e aux r\'esultats des sections pr\'ec\'edentes, implique le Th\'eor\`eme~\ref{thm:main2} et le Corollaire~\ref{cor:main4} {\'e}nonc{\'e}s dans l'introduction.

Rappelons que, pour une fraction rationnelle~$R$ {\`a} coefficients dans~$K$, non constante, $\crit_R$ est l'ensemble o{\`u} ${\deg_R \ge 2}$, qu'on note ${\crit_R(K) = \crit_R \cap K}$, et que~$R$ est mod\'er\'ee si~$\crit_R$ est inclus dans un sous-arbre fini de~$\PKber$.
Par ailleurs, un point critique inséparable de~$R$ est un point de~$\HK$ o{\`u}~$R$ est ins{\'e}parable, et la fonction ${\wf_R \: \HK \to ] 0, +\infty [}$ est d{\'e}finie par
\begin{equation}
  \label{eq:60}
  \wf_R
  =
  -\log (\| R' \| \times \diam \circ R / \diam)~.
\end{equation}

\begin{Prop}\label{prop:modere}
  Pour toute fraction rationnelle $R$ {\`a} coefficients dans~$K$, non constante, les propri\'et\'es suivantes sont \'equivalentes~:
  \begin{enumerate}
  \item
    $\wf_R$ est identiquement nulle;
  \item
    l'intérieur de~$\crit_R$ pour la topologie fine de~$\PKber$ est vide;
  \item
    $R$ est s{\'e}parable en chaque point de type~II;
  \item
    tout segment de~$\HK$ contient au plus un nombre fini de points o{\`u}~$\deg_R$ est divisible par la caract{\'e}risque r{\'e}siduelle de~$K$;
  \item
    $R$ est mod\'er\'ee.
  \end{enumerate}
  Lorsque ces propri{\'e}t{\'e}s {\'e}quivalentes sont v{\'e}rifi{\'e}es, $R$ est s{\'e}parable et poss{\`e}de les propri{\'e}t{\'e}s suivantes~:
  \begin{enumerate}
  \item [(i)]
    Chaque composante connexe de~$\crit_R$ est l'enveloppe convexe d'un sous-ensemble de~$\crit_R(K)$ contenant au moins deux points.
    En particulier, $\crit_R$ est une union finie de sous-arbres finis de l'enveloppe convexe de~$\crit_R(K)$.
  \item [(ii)]
    La fraction rationnelle $R$ n'a aucun point critique inséparable.
    De plus, chaque point de~$\PKber$ o{\`u}~$\deg_R$ est divisible par la caract{\'e}ristique r{\'e}siduelle de~$K$ est de type~II, et il y en au plus un nombre fini.
  \end{enumerate}
\end{Prop}

Le corollaire suivant est une cons{\'e}quence imm{\'e}diate de l'implication (3)$\Rightarrow$(5).

\begin{Cor}
  \label{cor:critere simple}
  Soit~$R$ une fraction rationnelle {\`a} coefficients dans~$K$, non constante.
  Si la caract\'eristique r{\'e}siduelle de~$K$ est nulle ou strictement plus grande que~$\deg(R)$, alors~$R$ est mod{\'e}r{\'e}e.
\end{Cor}

L'{\'e}quivalence des propri{\'e}t{\'e}s (2) et~(3) dans la Proposition~\ref{prop:modere} r{\'e}sulte de \cite[Proposition~10.2~(1)]{rivera-periode}.
Celle des propri{\'e}t{\'e}s~(2), (5) et~(i) a {\'e}t{\'e} d{\'e}montr{\'e}e par Faber dans \cite[Corollary~7.13]{faber1}.
Nous en donnons ici une d{\'e}monstration indép{e}ndante.
Voir aussi \cite[Proposition~2.9 et~2.11]{trucco} pour le cas des polyn{\^o}mes.

Une condition n{\'e}cessaire pour qu'une fraction rationnelle soit mod{\'e}r{\'e}e est que le degr{\'e} local en chaque point critique dans~$\PK$ ne soit pas divisible par la caract\'eristique r{\'e}siduelle de~$K$, voir Proposition~\ref{prop:modere}~(ii).
L'Exemple~\ref{ex:rigidiment-non-sauvage} montre que cette condition n'est pas suffisante.

\begin{proof}[D{\'e}monstration de la Proposition~\ref{prop:modere}]
  Rappelons que nous notons~$\tK$ le corps r{\'e}siduel de~$K$ et soit~$p$ sa caract{\'e}ristique.
  Pour chaque~$\rho$ dans~$\mathopen] 0, +\infty \mathclose[$, notons~$x(\rho)$ le point de~$\AKber$ associ\'e \`a la boule ${\{z \in K, |z| \le \rho\}}$.

  Les implications (i)$\Rightarrow$(5) et (5)$\Rightarrow$(2) sont imm{\'e}diates.
  Par le Corollaire~\ref{c:tres-wild} et la densit{\'e} des points de type~II pour la topologie fine, les propri{\'e}t{\'e}s (1), (2) et~(3) sont {\'e}quivalentes, et elles impliquent que~$R$ n'a aucun point critique inséparable et donc que~$R$ est s{\'e}parable.
  Par ailleurs, la propri{\'e}t{\'e}~(4) implique que tout point de~$\PKber$ o{\`u}~$\deg_R$ est divisible par~$p$ est de type~II par \cite[Propositions~4.5 et~4.6]{rivera-espace}.
  Combin{\'e}e {\`a} la propri{\'e}t{\'e}~(i), ceci entra{\^{\i}}ne la propri{\'e}t{\'e}~(ii).
  Pour achever la d{\'e}monstration de la proposition, il suffit donc de montrer les implications (1)$\Rightarrow$(4) et (4)$\Rightarrow$(i).

  Pour montrer l'implication (1)$\Rightarrow$(4), supposons que la propri{\'e}t{\'e}~(1) soit v{\'e}rifi{\'e}e mais que la propri{\'e}t{\'e}~(4) ne le soit pas.
  Il existe alors un multiple~$d$ de~$p$ et un segment~$\ell$ de~$\HK$ contenant une infinit{\'e} de points o{\`u} ${\deg_R = d}$.
  Soit~$\hell$ le plus petit segment ferm{\'e} de~$\PKber$ contenant~$\ell$.
  Il existe donc un point de~$\hell$ qui est accumul{\'e} par des points o{\`u}~$\deg_R$ est {\'e}gal {\`a}~$d$.
  On peut alors trouver des points de type~II distincts, $x$ et~$\hx$, appartenant {\`a}~$\ell$, tels qu'on ait ${\deg_R = d}$ sur~$\mathopen] x, \hx \mathclose[$, ${R(x) \neq R(\hx)}$, et que~$R$ envoie la couronne comprise entre~$x$ et~$\hx$ sur celle comprise entre~$R(x)$ et~$R(\hx)$, voir \cite[Propositions~4.5 et~4.6]{rivera-espace}.
  En changeant de coordonn{\'e}es si n{\'e}cessaire, on suppose qu'il existe~$r$, $\hr$, $r'$ et~$\hr'$ dans~$\mathopen] 0, 1 \mathclose[$ tels que ${x = x(r)}$, ${\hx = x(\hr)}$, ${R(x) = x(r')}$ et ${R(\hx) = x(\hr')}$.
  Alors on peut d\'evelopper~$R$ en s\'erie sur la couronne ${\{ y \in K, r < |y| < \hr \}}$, $R(z) = \sum_{k = -\infty}^{+\infty} a_k z^k$, et pour tout~$\rho$ dans~$\mathopen] r, \hr \mathclose[$ on a
  \begin{equation}
    \label{eq:61}
    \diam(R(x(\rho)))
    =
    |a_d| \rho^d
    >
    \sup \{ |a_k| r^k, k \in \Z, k \neq d \}~,
  \end{equation}
  voir la d{\'e}monstration de \cite[Lemme~5.3]{rivera-periode}.
  Comme~$d$ est divisible par~$p$, on~a
  \begin{equation}
    \label{eq:62}
    \| R' \|(x(\rho))
    <
    |a_d| \rho^d
    =
    \diam(R(x(\rho))) / \diam(x(\rho))
    \text{ et }
    \wf_R(x(\rho))
    >
    0~,
  \end{equation}
  ce qui est absurde.
  Cette contradiction montre l'implication (1)$\Rightarrow$(4).

  Il reste {\`a} montrer l'implication (4)$\Rightarrow$(i).
  Supposons que la propri{\'e}t{\'e}~(4) soit v{\'e}rifi{\'e}e.
  On commence par d{\'e}montrer que~$\crit_R$ ne contient aucun point de type~IV.
  Soit~$x_0$ un point de type~IV.
  En changeant de coordonn{\'e}es si n{\'e}cessaire, on suppose ${|x_0| \le 1}$ et ${|R(x_0)| \le 1}$.
  Posons ${r_0 \= \diam(x_0)}$ et ${d_0 \= \deg_R(x_0)}$, et notons que ${r_0 < 1}$.
  Par ailleurs, $d_0$ n'est pas divisible par~$p$ par la propri{\'e}t{\'e}~(5) et \cite[Proposition~4.6]{rivera-espace}.
  Pour chaque~$r$ dans~$\mathopen] r_0, 1 \mathclose[$, soit~$\Bber(r)$ la boule ouverte de~$\AKber$ contenant~$x_0$ et de diam{\`e}tre~$r$.
  Fixons~$r$ suffisamment proche de~$r_0$ de sorte que~$\Bber(r)$ ne contienne aucun point critique ni p{\^o}le de~$R$ et que le point~$x$ de~$\AKber$ associ{\'e} {\`a}~$\Bber(r)$ v{\'e}rifie ${\deg_R(x) = \deg_R(x_0)}$, voir \cite[Proposition~4.1]{rivera-espace}.
  En changeant de coordonn{\'e}es si nécessaire, on peut supposer que~$\Bber(r)$ contient~$0$ et que ${R(0) = 0}$.
  Alors, $R$ admet un d{\'e}veloppement en s{\'e}rie sur~$B(0, r)$, ${R(z) = \sum_{k = 1}^{+\infty} a_k z^k}$, tel que pour tout~$\rho$ dans ${\mathopen] 0, r \mathclose[}$ suffisamment proche de~$r$, on~a
  \begin{equation}
    \label{eq:63}
    \diam(R(x(\rho)))
    =
    |a_{d_0}| \rho^{d_0}
    >
    \sup \{ |a_k| \rho^k, k \in \N, k \neq d_0 \}~.
  \end{equation}
  Comme~$d_0$ n'est pas divisible par~$p$, on obtient
  \begin{equation}
    \label{eq:64}
    |d_0 a_{d_0}| \rho^{d_0}
    >
    \sup \{ |k a_k| \rho^k, k \in \N, k \neq d_0 \}~.
  \end{equation}
  Si l'on avait ${d_0 > 1}$, alors~$R'$ aurait~$d_0$ z{\'e}ros dans~$B(0, \rho)$, et donc dans~$B$, compt{\'e}s avec multiplicit{\'e}, ce qui contredirait notre choix de~$B$.
  On en d{\'e}duit que ${d_0 = 1}$ et que~$x_0$ n'appartient pas {\`a}~$\crit_R$.

  Pour achever la d{\'e}monstration de l'implication (4)$\Rightarrow$(i), soit~$\cT$ une composante con\-nexe de~$\crit_R$.
  Alors~$\cT$ est ferm{\'e}e et n'est pas r{\'e}duite {\`a} un point, voir \cite[Propositions~4.4, 4.5 et~4.6]{rivera-espace}.
  Il suffit donc de d{\'e}montrer que tout bout de~$\cT$ est dans~$\PK$.
  Aucun bout de~$\cT$ n'est de type~III, car tout point de type~III appartient {\`a} un segment ouvert de~$\HK$ o{\`u}~$\deg_R$ est constant, voir \cite[Proposition~4.6]{rivera-espace}.
  Comme~$\cT$ ne contient pas de point de type~IV, il ne reste plus qu'{\`a} montrer qu'aucun bout de~$\cT$ n'est de type~II.
  Supposons qu'il y avait un bout~$x_0$ de~$\cT$ de type~II.
  En changeant de coordonn{\'e}es si n{\'e}cessaire, on suppose ${R(x) = x = \xcan}$.
  La r{\'e}duction~$\tR$ de~$R$ est de degr{\'e} au moins deux, car~$\xcan$ appartient {\`a}~$\crit_R$, et elle poss{\`e}de un unique point critique, car~$\xcan$ est un bout de~$\cT$, voir \cite[Proposition~2.4]{R1}.
  En changeant de coordonn{\'e}es si n{\'e}cessaire, on suppose que~$\infty$ est le seul point critique de~$\tR$ et que ${\tR(\infty) = \infty}$.
  Soient~$P$ et~$Q$ des polyn{\^o}mes {\`a} coefficients dans~$\tK$, sans facteur commun, et tels que ${\tR(\zeta) = P(\zeta)/Q(\zeta)}$.
  Posons
  \begin{equation}
    \label{eq:65}
    d
    \=
    \deg_{\tR}(\infty),
    e
    \=
    \deg(P)
    \text{ et }
    f
    \=
    \deg(Q)~,
  \end{equation}
  et notons que ${d = e - f}$.
  Comme~$\infty$ est le seul point critique de~$\tR$, tout z{\'e}ro de~$Q$ est simple, et tout z{\'e}ro de ${P'Q - PQ'}$ est aussi un z{\'e}ro de~$Q$.
  Il s'ensuit que ${P'Q - PQ'}$ est constant.
  Par cons{\'e}quent, si l'on note par~$a$ et~$b$ les coefficients dominants de~$P$ et~$Q$, respectivement, alors on a ${eab =fab}$ et donc ${e = f}$ dans~$\tK$.
  C'est-{\`a}-dire, $d$ est divisible par~$p$.
  Ceci contredit la propri{\'e}t{\'e}~(4) par \cite[Propositions~4.6]{rivera-espace}, et termine la d{\'e}monstration de l'implication (4)$\Rightarrow$(i) ainsi que celle de la proposition.
\end{proof}

Les d\'emonstrations du Th\'eor\`eme~\ref{thm:main2} et du Corollaire~\ref{cor:main4} reposent sur la proposition suivante, qui s'appuie sur la construction des pr{\'e}images it{\'e}r{\'e}es {\'e}tales donn{\'e}e par le Lemme~\ref{lem:big-return}.
Lorsque~$K$ est de caractéristique résiduelle nulle, ce r{\'e}sultat d{\'e}coule des r{\'e}sultats de Luo~\cite[Propositions~11.4 et~11.5]{luo2}, d{\'e}montr{\'e}s avec une m{\'e}thode diff{\'e}rente.

\begin{Prop}\label{prop:expansion}
  Soit~$R$ une fraction rationnelle {\`a} coefficients dans~$K$, mod\'er\'ee et de degr{\'e} au moins deux.
  Si~$\rho_R$ ne charge aucun segment de~$\HK$, alors~$R$ poss{\`e}de un point p{\'e}riodique r{\'e}pulsif dans~$\PK$.
\end{Prop}

\begin{proof}
  On reprend les notations introduites dans le paragraphe pr{\'e}c{\'e}dant le Lemme~\ref{lem:big-return}.
  Notre hypoth{\`e}se que~$R$ est mod{\'e}r{\'e}e implique qu'on a ${\crit_R \subseteq \aT_1}$, voir Proposition~\ref{prop:modere}.
  Soit~$p$ l'entier fourni par le Lemme~\ref{lem:big-return} avec ${\varepsilon = \frac{1}{2}}$, et~$B_0$ une composante connexe de ${\PKber \setminus \aT_p}$ intersectant~$J_R$.
  Pour chaque~$n$ dans~$\N^*$, soit~$G_n$ l'union des composantes connexes~$B$ de~$R^{-n}(B_0)$, tels que pour tout~$j$ dans ${\{0, \ldots, n - 1\}}$ on ait ${R^j(B) \cap \aT_1 = \emptyset}$.
  Notons que~$B$ est une boule, et que~$R^n$ envoie~$B$ sur~$B_0$ de fa\c{c}on univalente.

  Pour tout~$n$ dans~$\N^*$, on a ${\rho_R(G_n) \ge \rho_R(B_0)/2}$ par le Lemme~\ref{lem:big-return}.
  Il existe alors~$k$ et~$n$ dans~$\N^*$, tels que ${k < n}$ et ${G_k \cap G_n \neq \emptyset}$.
  Fixons~$x$ dans ${G_k \cap G_n}$, notons~$B$ la composante conexe de~$R^{-n}(B_0)$ contenant~$x$, et posons ${B' \= R^k(B)}$.
  Alors~$B'$ est une boule disjointe de~$\aT_1$ qui intersecte~$B_0$.
  Par cons{\'e}quent, on a ${B' \subseteq B_0}$, car~$B_0$ est un composante connexe de ${\PKber \setminus \aT_1}$, et ${B' \neq B_0}$, car~$B_0$ intersecte~$J_R$.
  Comme~$R^{n - k}$ envoie~$B'$ sur~$B_0$ de fa\c{c}on univalente, le Lemme de Schwarz implique alors que~$B'$ contient un point p{\'e}riodique r{\'e}pulsif dans~$\PK$ de~$R$, voir \cite[\S~1.3.1]{R1}.
\end{proof}

\begin{proof}[D\'emonstration du Th\'eor\`eme~\ref{thm:main2}]
  Soit~$R$ une fraction rationnelle {\`a} coefficients dans~$K$ mod\'er\'ee, et de degr{\'e} au moins deux.
  Notons que $\chi(R) \ge 0$ par le Th\'eor\`eme~\ref{thm:main3}~(1).
  Si cette in{\'e}galit{\'e} est stricte, alors le point~(1) d{\'e}coule du Th\'eor\`eme~\ref{thm:main3}~(3).

  Supposons que ${\chi(R) = 0}$.
  Si $R$ a bonne r\'eduction potentielle, alors le point~(3) d{\'e}coule de~\cite[Th\'eor\`eme~E]{theorie-ergo}.
  Supposons maintenant que~$R$ n'a pas bonne r\'eduction potentielle.
  Alors on a $\rho_R(\HK) = 1$ par la Proposition~\ref{prop:integrable}, ${J_R \subset \HK}$ par le Th\'eor\`eme~\ref{thm:main3}~(2), et~$\rho_R$ charge un segment de~$\HK$ par la Proposition~\ref{prop:expansion}.
  Comme~$\rho_R$ n'a pas d'atomes d'apr{\`e}s~\cite[Th\'eor\`eme~E]{theorie-ergo}, la fraction rationelle~$R$ est affine Bernoulli par le Th{\'e}or{\`e}me~\ref{thm:main1}, et le point~(3) d{\'e}coule de la Proposition~\ref{prop:affine Bernoulli}~(1, 4).
\end{proof}

\begin{proof}[D\'emonstration du Corollaire~\ref{cor:main4}]
  Lorsque le point~(1) est v{\'e}rifié, les points~(2) et~(4), ainsi que les assertions {\`a} la fin du corollaire, d{\'e}coulent de \cite[Th{\'e}or{\`e}me~E]{theorie-ergo} si~$R$ a bonne r{\'e}duction potentielle, et de la Proposition~\ref{prop:affine Bernoulli}~(1, 2) si~$R$ est affine Bernoulli.

  Les implications (2)$\Rightarrow$(3) et (3)$\Rightarrow$(1) d{\'e}coulent du Th{\'e}or{\`e}me~\ref{thm:main2}.
  Ceci d{\'e}montre l'{\'e}quivalence des points~(1), (2) et~(3).
  Par ailleurs, l'implication (4)$\Rightarrow$(5) est imm{\'e}diate, (5)$\Rightarrow$(6) d{\'e}coule de la Proposition~\ref{prop:expansion}, et (6)$\Rightarrow$(1) du Th{\'e}or{\`e}me~\ref{thm:main1}.
  Ceci d{\'e}montre l'{\'e}quivalence des points~(1), (4), (5) et~(6), et compl{\`e}te la d{\'e}monstration du corollaire.
\end{proof}

\subsection{Familles m\'eromorphes de fractions rationnelles complexes}
\label{sec:application}

Dans cette section, nous {\'e}tudions la variation de l'exposant de \mbox{Lyapunov} dans des familles m\'eromorphes complexes.
L'\'etude de la d\'eg\'en\'erescence des exposants de \mbox{Lyapunov} dans des familles m\'eromorphes \`a une variable a \'et\'e initi\'ee par \mbox{DeMarco} et Faber \cite{demarco-faber1,demarco-faber2}, puis poursuivie en toute dimension par le premier auteur~\cite{favre-degeneration}, voir aussi~\cite{bianchi-okuyama} et~\cite{poineau22}.
Le comportement des multiplicateurs des points p\'eriodiques dans des familles m\'eromorphes de fractions rationnelles a \'et\'e {\'e}tudi{\'e} par Luo~\cite{luo2}, et pr{\'e}c{\'e}demment par DeMarco et McMullen~\cite{demarco-mcmullen} dans le cas des polyn{\^o}mes.
Nous nous appuierons sur le travail de Gauthier, Okuyama et Vigny~\cite{GOV1}, qui fournit une approximation effective des exposants de \mbox{Lyapunov} en termes des multiplicateurs aux points périodiques.

Pour {\'e}noncer nos r{\'e}sultats, nous introduisons quelques notations.
Rappelons que la d\'eriv\'ee sph\'erique d'une fraction rationnelle complexe~$R$ non constante est donn\'ee par
\[
  \|R'\|(z)
  \=
  \frac{|R'(z)|}{1 +|R(z)|^2} \, (1 + |z|^2)~.\]
Cette fonction se prolonge en une fonction continue et strictement positive, d{\'e}finie sur la sph\`ere de Riemann~$\PC$.
Si le degr\'e de~$R$ est au moins deux, alors pour tout entier~$n$ v{\'e}rifiant ${n \ge 1}$, on pose
\[
  \lambda_n(R)
  \=
  \frac1n \log\left(\sup_{R^n(z) =z} \|(R^n)'\|(z) \right)~.
\]
On remarque que $\lambda_n(R) \le \sup_{\PC} \log \|R'\|$.

\medskip

On note de plus ${\D \= \{ |t| < 1, t \in \C \}}$ et ${\D^* \= \{ t \in \D, t \neq 0 \}}$.
Pour un entier~$d$ v{\'e}rifiant ${d \ge 2}$, nous dirons que $\{R_t\}_{t \in \D}$ est une \emph{famille m\'eromorphe de fractions rationnelles de degr\'e~$d$} (sous-entendu
param\'etr\'ee par le disque unit\'e et avec au plus un p\^ole en $t=0$),
si
\begin{itemize}
\item
  $R_t$ est le quotient de deux polyn\^omes $P_t, Q_t$ dont les coefficients sont holomorphes sur~$\D^*$ et m\'eromorphes en $0$;
\item
  $P_t$ et $Q_t$ n'ont pas de z\'ero commun si $t \neq0$;
\item
  $\max \{ \deg(P_t), \deg(Q_t) \} = d$ si $t \neq 0$.
\end{itemize}

Il n'est pas difficile de voir que si $\{R_t\}_{t \in \D}$ est une famille m\'eromorphe de fractions rationnelles, alors la fonction
$\lambda_n(t) \= \lambda_n(R_t)$ est sous-harmonique et continue sur~$\D^*$ et v\'erifie
$\lambda_n(t) \le C \log|t|^{-1} + C'$ pour des constantes $C, C' >0$ ind\'ependantes de $n$.

\medskip

\`A toute famille m\'eromorphe $\{R_t\}_{t \in \D}$ de fractions rationnelles de degr\'e~$d$ est canoniquement associ\'ee une fraction rationnelle de m\^eme degr\'e sur le corps des s\'eries de Laurent~$\C(\!(t)\!)$.
Ce corps est complet pour la norme $t$-adique, mais il n'est pas alg\'ebriquement clos.
Suivant~\cite{kiwi-cubic}, nous noterons~$\Pu$ le compl\'et\'e de la clot\^ure alg\'ebri\-que de $\C(\!(t)\!)$: c'est un corps m\'etris\'e non-archim\'edien complet et alg\'ebriquement clos.
Nous noterons par~$\aR$ la fraction rationnelle {\`a} coefficients dans~$\Pu$ de degr\'e~$d$ induite par~$\{R_t\}_{t \in \D}$.
Comme $\C(\!(t)\!)$ est de caract\'eristique r{\'e}siduelle nulle, $\aR$ est mod\'er\'ee, voir le Corollaire~\ref{cor:critere simple}.

\begin{Thm}
  \label{thm:main6}
  Soit $\{R_t\}_{t \in \D}$ une famille m\'eromorphe de fractions rationnelles de degr\'e au moins deux.
  Alors, les assertions suivantes sont \'equivalentes.
  \begin{enumerate}
  \item
    La fonction $t \mapsto \chi(R_t)$ est born\'ee en $0$.
  \item
    L'exposant $\chi(\aR)$ est nul.
  \item
    Pour tout~$r$ dans ${\mathopen] 0, 1 \mathclose[}$, il existe une constante $C>0$ telle que
    $$
    \sup_{n \in \N^*} \sup_{t \in \D^*, |t|\le r} \lambda_n(R_t) \le C~.
    $$
  \end{enumerate}
  Lorsque ces assertions {\'e}quivalentes sont satisfaites, la fonction ${t \mapsto \chi(R_t)}$ se prolonge en une fonction sous-harmonique born{\'e}e d{\'e}finie sur~$\D$, et~$\aR$ est soit affine Bernoulli, soit a bonne r\'eduction potentielle.
\end{Thm}

L'{\'e}quivalence (1)$\Leftrightarrow$(2) r{\'e}sulte de la combinaison de~\cite[Theorem~3.1]{GOV1} et~\cite[Theorem~C]{favre-degeneration}.
La d{\'e}monstration que l'assertion~(3) est {\'e}quivalente aux assertions~(1) et~(2) s'appuie sur les Th{\'e}or{\`e}mes~\ref{thm:main3} et~\ref{thm:main2}.

Luo~\cite{luo1,luo2} a r\'ecemment donn\'e une caract\'erisation g\'eom\'etrique
des composantes hyperboliques de l'espace des modules des fractions rationnelles dans lesquelles les multiplicateurs restent born\'es.
Le Th\'eor\`eme~\ref{thm:main6} s'applique sans hypoth\`ese sur le lieu de bifurcation.

Comme toute fraction rationnelle affine Bernoulli est de degr\'e au moins~$4$, on en d\'eduit le corollaire suivant.

\begin{Cor}\label{cor:degen}
  Soit $\{R_t\}_{t \in \D}$ une famille m\'eromorphe de fractions rationnelles de degr\'e~$2$ ou~$3$.
  Si $\{R_t\}_{t \in \D}$ diverge dans l'espace des modules des fractions rationnelles lorsque ${t \to 0}$, alors pour tout~$N$ dans~$\N^*$ il existe un cycle p\'eriodique dont le multiplicateur est born\'e inf\'erieu\-rement par~$|t|^{-N}$.
\end{Cor}

Ce corollaire implique qu'il n'existe aucune courbe alg\'ebrique de l'espace des modules des fractions rationnelles de degr\'e~$2$ ou~$3$ sur laquelle la famille est stable au sens de \cite{mcmullen}.
En degr\'e $2$ ce r\'esultat est une cons\'equence facile de la param\'etrisation de Milnor des fractions quadratiques en termes des multiplicateurs de leurs points fixes~\cite{milnor-quadra}.
En degr\'e~$3$, nous n'avons pas de param\'etrisation explicite en termes de multiplicateurs de points p\'eriodiques, voir cependant~\cite{lloyd}.
Notons que la rigidit\'e de Thurston implique la non-existence de famille alg\'ebrique stable en tout degr\'e non carr\'e, voir~\cite[Theorem~2.2]{mcmullen}, ce qui n'est pas une cons\'equence directe du Corollaire~\ref{cor:degen}.
Nous renvoyons {\`a} l'article récent~\cite{JiXie23} pour une approche non-archimé\-dienne à la rigidité de McMullen.

Pour d{\'e}montrer le Th{\'e}or{\`e}me~\ref{thm:main6} et le Corollaire~\ref{cor:degen}, on utilise la cons{\'e}quence suivante de la combinaison de~\cite[Theorem~3.1]{GOV1} et de~\cite[Theorem~C]{favre-degeneration}.

\begin{Prop}
  \label{p:famille-positive}
  Soit $\{R_t\}_{t \in \D}$ une famille m\'eromorphe de fractions rationnelles de degr\'e au moins deux, telle que ${\chi(\aR) > 0}$.
  Alors, pour tout $\epsilon>0$ il existe $\eta>0$ tel que pour tout~$n$ assez grand et tout~$t$ assez petit, la fraction rationnelle $R_t^n$ admet au moins~$\eta d^n$ points fixes~$z$ v\'erifiant
  \[
    \frac1n \log \|R_t^n\|(z)
    \ge
    (\chi(\aR) - \epsilon) \, \log |t|^{-1}~.
  \]

\end{Prop}

\begin{proof}
  Supposons ${\chi(\aR) > 0}$ et choisissons une constante~$C$ de telle sorte que
  \[
    C
    >
    \chi(\aR)
    \text{ et }
    \sup_{t \in \D^*, |t| \le 1/2} \sup_{\PC} \log \|R_t\|
    \le
    C \log|t|^{-1}~.
  \]
  De plus, pour chaque~$t$ dans~$\D^*$ et~$n$ dans~$\N^*$ posons
  \[
    K_t(n)
    \=
    \# \left\{ z, \, R_t^n(z) =z \text{ et } \frac1n \, \log^+ \|R^n_t\|(z) \ge (\chi(\aR) - \epsilon) \log|t|^{-1} \right\}
  \]
  et
  \[
    L^+_n(R_t)
    \=
    \frac1{d^n} \sum_{z \in \PC, R_t^n(z) =z} \frac1n \, \log^+ \|R^n_t\|(z)~,
  \]
  o\`u ${\log^+ \= \max \{0, \log \}}$.
  On tire de~\cite[Theorem~3.1]{GOV1} appliqu\'e \`a $r=1$ et de~\cite[Theorem~C]{favre-degeneration}
  \[
    \left| L^+_n(R_t) - \chi(\aR)\, \log|t|^{-1} \right|
    \le
    \left| L^+_n(R_t) - \chi(R_t) \right|
    +
    \frac{\epsilon}{3} \log|t|^{-1}
    \le \frac{2 \epsilon}{3} \log|t|^{-1}~, \]
  pour tout $n$ assez grand et pour tout $t$ assez petit (ind\'ependant de $n$).
  On obtient
  \begin{equation*}
    \chi(\aR) - \frac{2 \epsilon}{3}
    \le
    \frac{L^+_n(R_t)}{\log|t|^{-1}}
    \le
    C \frac{K_t(n)}{d^n} + (\chi(\aR) - \epsilon) \times \left(1 - \frac{K_t(n)}{d^n} \right)
  \end{equation*}
  et
  \[
    \frac{K_t(n)}{d^n}
    \ge
    \frac{\epsilon}{3(C - \chi(\aR) + \epsilon)}
    >
    0~.
    \qedhere
  \]
\end{proof}

\begin{proof}[D\'emonstration du Th\'eor\`eme~\ref{thm:main6}]
  L'implication (1)$\Rightarrow$(2) est une cons{\'e}quence directe de~\cite[Theorem~C]{favre-degeneration}, et l'implication (3)$\Rightarrow$(2) d{\'e}coule de l'in{\'e}galit{\'e} ${\chi(R) \ge 0}$ donn{\'e} par le Th{\'e}or{\`e}me~\ref{thm:main2} et la Proposition~\ref{p:famille-positive}.
  Pour montrer les implications (2)$\Rightarrow$(1) et (2)$\Rightarrow$(3), supposons que ${\chi(\aR) = 0}$.
  Alors on a $\chi(R_t) =o(\log|t|^{-1})$ par~\cite[Theorem~C]{favre-degeneration}.
  Il s'ensuit que pour tout $\epsilon>0$ la fonction ${t \mapsto \chi(R_t) + \epsilon \log|t|}$ est sous-harmonique sur~$\D^*$ et localement born\'ee sup\'erieurement.
  Elle se prolonge donc en une fonction sous-harmonique d{\'e}finie sur~$\D$.
  Par le principe du maximum, on a donc $\chi(R_t) + \epsilon \log|t| \le C'$ pour une constante~$C'$ ad\'equate et pour tout~$\epsilon$ et~$t$ v{\'e}rifiant ${|\epsilon| \le 1}$ et ${|t| \le 1/2}$.
  En faisant $\epsilon \to0$, on en d\'eduit que $\chi(R_t)$ est born\'ee sup\'erieurement et se prolonge donc aussi en une fonction sous-harmonique d{\'e}finie sur~$\D$.
  Elle est donc born\'ee car positive.
  Cela montre l'assertion~(1).
  Pour montrer l'assertion~(3), notons que soit~$\aR$ est affine Bernoulli, soit elle a bonne r{\'e}duction potentielle par le Th{\'e}or{\`e}me~\ref{thm:main2}.
  En particulier, tout cycle p\'eriodique de~$\aR$ dans~$\PPu$ est indiff\'erent ou attractif.
  Il s'ensuit que pour tout~$n$ dans~$\N^*$, la fonction ${t\mapsto\lambda_n(R_t)}$ se prolonge en une fonction sous-harmonique born\'ee d{\'e}finie sur~$\D$.
  Le principe du maximum implique alors que pour tout~$r$ dans~$\mathopen] 0, 1 \mathclose[$ on a
  $$
  \sup_{n \in \N^*} \sup_{t \in \D^*, |t|\le r} \lambda_n(R_t)
  \le
  \sup_{t \in \D, |t| = r} \lambda_n(R_t)
  \le
  \sup_{t \in \D, |t| = r} \sup_{\PC} \log \|R_t\|~.$$
  Le dernier membre de droite \'etant ind\'ependant de $n$, on obtient (3).
\end{proof}

\begin{proof}[D\'emonstration du Corollaire~\ref{cor:degen}]
  Si~$\chi(\aR)$ est non nul, alors on a ${\chi(\aR) > 0}$ par le Th{\'e}or{\`e}me~\ref{thm:main2}, et l'assertion d{\'e}sir{\'e}e est une cons{\'e}quence de la Proposition~\ref{p:famille-positive}.

  Supposons que ${\chi(\aR) = 0}$.
  On peut conjuguer~$\aR$ par un \'el\'ement~$\phi$ de~$\PGL(2, \C(\!(t)\!))$ \`a une fraction ayant bonne r\'eduction, voir le Th{\'e}or{\`e}me~\ref{thm:main2}.
  Pour chaque~$N$ dans~$\N^*$, notons~$\phi_N$ la fraction rationnelle \`a coefficients dans $\C[t]$ obtenue en tronquant les coefficients de~$\phi$ \`a l'ordre $N$.
  Pour~$N$ assez grand, c'est une transformation de M\"obius, et les coefficients de $\phi_N^{-1} \circ \aR \circ \phi_N$ sont proches dans~$\Pu$ des coefficients de $\phi^{-1} \circ \aR \circ \phi$.
  Il s'ensuit que~$\phi_N^{-1} \circ \aR \circ \phi_N$ a bonne r\'eduction.
  Comme les coefficients de~$\phi_N$ sont holomorphes sur~$\D$, on en d{\'e}duit que~$\phi_N(t) \circ R_t \circ \phi_N(t)^{-1}$ converge vers la r{\'e}duction de~$\phi_N^{-1} \circ \aR \circ \phi_N$, lorsque ${t \to 0}$, qui est une fraction rationnelle complexe de m{\^e}me degr{\'e} que~$\{ R_t \}_{t \in \D}$.
  Par cons{\'e}quent, cette famille ne diverge pas dans l'espace des modules.
\end{proof}

\begin{Rem*}
 L’article~\cite{favre25} rend le Corollaire~\ref{cor:degen} effectif en degré $3$ et démontre une version plus forte de la Proposition~\ref{p:famille-positive}. 
 D’autre part, Gong~\cite{gong25} établit une version du Théorème~\ref{thm:main6} pour les suites.
\end{Rem*}


\section{Exemples et conjectures}\label{sec:open}

Nous présentons quelques exemples (\S~\ref{ss:exemples}) ainsi que des conjectures (\S~\ref{ss:questions}).

\subsection{Exemples}
\label{ss:exemples}

Notre premier exemple concerne l'exposant de \mbox{Lyapunov} des fractions rationnelles ayant une bonne r{\'e}duction et des fractions rationnelles affines Bernoulli.
Les exemples suivants pr{\'e}sentent un exposant de \mbox{Lyapunov} strictement n{\'e}gatif, l'un dont l'ensemble de Julia est un ensemble de Cantor contenu dans~$\HK$ (Exemple~\ref{ex:negativement-vide}) et l'autre dont l'ensemble de Julia intersecte~$\PK$ (Exemple~\ref{ex:negativement-non-vide}).
Nous poursuivons avec un exemple d'un polyn{\^o}me qui ne satisfait aucun des points~(1), (2) ou~(3) du Th\'eor\`eme~\ref{thm:main3} (Exemple~\ref{ex:atypique}), ainsi qu'un polyn{\^o}me non mod{\'e}r{\'e} sans aucun point critique ins{\'e}parable dans~$\PK$ (Exemple~\ref{ex:rigidiment-non-sauvage}).

\begin{Exemple}
  Soit~$R$ une fraction rationnelle {\`a} coefficients dans~$K$ et de degr{\'e} au moins deux.
  Supposons tout d'abord que~$R$ a bonne r\'eduction.
  La r{\'e}duction~$\tR$ de~$R$ dans~$\tK$ est donc de m\^eme degr\'e que~$R$.
  Comme la mesure d'{\'e}quilibre de~$R$ est la masse de Dirac au point~$\xcan$, le Corollaire~\ref{c:tres-wild} entra{\^{\i}}ne qu'on a ${\chi(R) \le 0}$ avec {\'e}galit{\'e} si et seulement si~$\tR$ est s{\'e}parable.

  Si~$R$ est affine Bernoulli, alors le Th{\'e}or{\`e}me~\ref{thm:main3}~(1, 3) entra{\^{\i}}ne que ${\chi(R) \le 0}$, avec {\'e}galit{\'e} si et seulement si aucun des facteurs de dilatation de~$R$ n'est disivible par la caract\'eristique r{\'e}siduelle de~$K$.
\end{Exemple}

\begin{Exemple}
  \label{ex:negativement-vide}
  Supposons que la caract\'eristique de~$K$ soit nulle et que sa caract\'eristique r{\'e}siduelle~$p$ soit strictement positive.
  Soit~$P_0$ le polyn{\^o}me {\`a} coefficients dans~$K$ d{\'e}fini par
  \begin{equation}
    \label{eq:66}
    P_0(z)
    \=
    (z^p - z^{p^2})/p~.
  \end{equation}
  Alors~$J_{P_0}$ est un ensemble de Cantor o{\`u} ${\diam = p^{1/(p-1)}}$ et ${| \cdot | \le 1}$, voir \cite[Exemple~6.3]{R1} et \cite[Proposition~5.6]{theorie-ergo}.
  Par ailleurs, un calcul montre que ${\| P_0' \| = | \cdot |^{p - 1}}$ sur~$J_{P_0}$, d'o{\`u} l'on d{\'e}duit que ${\chi(P_0) < 0}$ car~$\rho_{P_0}$ charge~$\{ x \in \AKber, |x| < 1 \}$.
  En effet, on a ${\chi(P_0) = \frac{1}{p + 1} \log |p|}$ par la formule de Przytycki, voir~\cite[\S5]{okuyama} ou la Proposition~\ref{p:pr}.
\end{Exemple}

\begin{Exemple}
  \label{ex:negativement-non-vide}
  Supposons que la caract\'eristique de~$K$ soit nulle et que sa caract\'eristique r{\'e}siduelle~$p$ soit strictement positive.
  Fixons~$\lambda$ dans~$K$ tel que ${|p|^{-1} < |\lambda| < |p|^{-(p^2 - 1)/p)}}$ et soit~$P_1$ le polyn{\^o}me d{\'e}fini par
  \begin{equation}
    \label{eq:67}
    P_1(z)
    \=
    \lambda (z^{p^2 + p} - z^{p^2})~.
  \end{equation}
  On va montrer que ${\chi(P_1) < 0}$ et cependant que ${J_{P_1} \not\subseteq \HK}$.
  En effet, une analyse du polyg{\^o}ne de Newton de ${P_1(z) - z}$ montre que~$P_1$ poss{\`e}de~$p$ points fixes dans~$K$ de norme {\'e}gale {\`a}~$1$.
  Comme ${|P_1'| = |p\lambda| > 1}$ sur~$\{ z \in K \: |z| = 1\}$, chacun de ces points fixes est r{\'e}pulsif et, par cons{\'e}quent, ${J_{P_1} \not\subseteq \HK}$.
  Par ailleurs, un calcul direct montre que chaque point critique~$c$ de~$P_1$ dans~$K$ diff{\'e}rent de~$0$ satisfait ${c^p = p/(p + 1)}$.
  Cela implique que tous les points critiques de~$P_1$ dans~$K$ appartiennent {\`a} la boule ouverte~$B_0$ de~$K$ de centre~$0$ et diam{\`e}tre~$|\lambda|^{-1/(p^2 - 1)}$.
  Combin{\'e} {\`a} ${P_1(B_0) = B_0}$ et {\`a} la formule de Przytycki, voir~\cite[\S5]{okuyama} ou la Proposition~\ref{p:pr}, cela entra{\^{\i}}ne ${\chi(P_1) = \log |p^2 + p| = \log |p| < 0}$.
\end{Exemple}

\begin{Exemple}
  \label{ex:atypique}
  Supposons que la caract\'eristique de~$K$ soit nulle, et que sa caract\'eristique r{\'e}siduelle~$p$ est strictement positive.
  Pour un param{\`e}tre~$a$ dans~$K$ bien choisi, le polyn\^ome~$P_2$ {\`a} coefficients dans~$K$ d{\'e}fini par
  \begin{equation}
    \label{eq:68}
    P_2(z)
    \=
    (1-a) z^{p+1} + a z^p~,
  \end{equation}
  admet~$1$ comme point fixe r{\'e}pulsif et v{\'e}rifie ${\chi(P_2) = 0}$.
  Ce polyn{\^o}me, introduit par \mbox{Benedetto} dans~\cite{benedetto}, ne satisfait donc aucun des points~(1), (2) ou~(3) du Th\'eor\`eme~\ref{thm:main3}.
\end{Exemple}

\begin{Exemple}
  \label{ex:rigidiment-non-sauvage}
  Supposons que la caract\'eristique de~$K$ soit nulle, et que sa caract\'eristique r{\'e}siduelle~$p$ soit un nombre premier impair.
  Pour chaque~$t$ dans~$K$, soit~$P_t$ le polyn{\^o}me {\`a} coefficients dans~$K$ d{\'e}fini par
  \begin{equation}
    \label{eq:69}
    P_t(z)
    \=
    p z^{p + 1} - z^p + ptz~.
  \end{equation}
  Lorsque ${|t| < 1}$, le polyn{\^o}me~$P_t$ a une r{\'e}duction ins{\'e}parable et, par cons{\'e}quent, il n'est pas mod{\'e}r{\'e}.
  Par ailleurs, le lemme de Hensel montre que, lorsque~$|t|$ est assez petit mais non nul, $P_t$ poss{\`e}se ${p + 1}$ points critiques dans~$K$ distincts deux {\`a} deux~: l'infini, un proche de~$1/(p + 1)$, et un proche de chaque racine $(p - 1)$-i{\`e}me de~$t$.
  Le degr{\'e} local de~$P_t$ \`a l'infini est ${p + 1}$, et celui en chaque un de ses points critiques finis est~$2$.
  Cela montre que~$P_t$ n'a aucun point critique ins{\'e}parable, bien que~$P_t$ ne soit pas mod{\'e}r{\'e}.
\end{Exemple}

\subsection{Conjectures}
\label{ss:questions}

Fixons une fraction rationnelle~$R$ {\`a} coefficients dans~$K$ et de degr{\'e} au moins deux.
Rappelons qu'un point fixe~$x$ est dit r\'epulsif s'il appartient \`a ${\crit_R \cap \HK}$, ou s'il appartient \`a~$\PK$ et le multiplicateur de~$R$ en $x$ est de norme strictement plus grand que~$1$.
De plus, pour chaque~$n$ dans~$\N^*$, nous d{\'e}signons par~$\RFix(R^n)$ l'ensemble des points fixes r\'epulsifs de~$R^n$.

Lorsque~$K$ est un corps de caract\'eristique nulle, l'\'equidistribution des points p\'eriodi\-ques dans~$\PK$ de~$R$ est donn\'ee par~\cite[Th\'eor\`eme~B]{theorie-ergo}, voir aussi~\cite[Theorem~1.2]{okuyama2}.
Pour~$K$ arbitraire et lorsque ${\chi(R) > 0}$, le Th\'eo\-r\`eme~\ref{thm:main3}~(3) raffine ce r\'esultat et {\'e}tablit l'{\'e}quidistribution des points p\'eriodiques r\'epulsifs dans~$\PK$.

\begin{conj}
  \label{conj:rep-equi}
  On a
  \[
    \frac1{d^n} \sum_{\RFix(R^n)} \deg_{R^n}(x) \delta_x \to \rho_R
    \text{ lorsque }
    n \to +\infty~.
  \]
\end{conj}

\begin{conj}
  \label{conj:maxi}
  La fraction rationnelle~$R$ admet une unique mesure d'entropie m{\'e}trique maximale~$\mu_R$, et on a
  \[ \frac1{\# \RFix(R^n)} \sum_{\RFix(R^n)}\delta_x \to \mu_R
    \text{ lorsque }
    n \to +\infty~.
  \]
\end{conj}

Lorsque~$R$ est mod\'er\'ee et ${\chi(R) =0}$, le Th\'eor\`eme~\ref{thm:main2}~(2, 3) et la Proposition~\ref{prop:affine Bernoulli} r{\'e}solvent ces conjectures de mani{\`e}re affirmative.
Dans le cas non mod\'er\'e, ces conjectures semblent d\'elicates.

\begin{conj}
  \label{conj:equi}
  Si $\chi(R) >0$, alors on a
  \begin{equation}
    \label{eq:70}
    \lim_{n \to +\infty} \frac{1}{d^n} \sum_{x\in \RFix(R^n) \cap \HK} \deg_{R^n}(x)
    =
    0~.
  \end{equation}
\end{conj}

Il n'est pas difficile de voir que la Conjecture~\ref{conj:equi} est v{\'e}rifi{\'e}e dans le cas des polyn\^omes.
Un r{\'e}sultat de Kiwi~\cite[Theorem~1]{kiwi} r{\'e}sout cette conjecture de mani{\`e}re affirmative dans le cas des fractions rationnelles de degr\'e~$2$
sur un corps de caractéristique résiduelle nulle extension du corps des s\'eries de Puiseux {\`a} coefficients dans~$\C$.

Une solution affirmative {\`a} la Conjecture~\ref{conj:equi} combin{\'e}e avec le Th{\'e}or{\`e}me~\ref{thm:main3}~(3) permettrait d'apporter une solution affirmative aux Conjectures~\ref{conj:rep-equi} et~\ref{conj:maxi} dans le cas o{\`u} ${\chi(R) > 0}$.


\end{document}